\documentclass[final,review,onefignum,onetabnum]{siamart190516}

\usepackage{lipsum,version}
\usepackage{amsfonts}
\usepackage{graphicx}
\usepackage{subfigure}
\usepackage{epstopdf}
\usepackage{algorithmic}

\ifpdf
  \DeclareGraphicsExtensions{.eps,.pdf,.png,.jpg}
\else
  \DeclareGraphicsExtensions{.eps}
\fi

\renewcommand\theequation{\thesection.\arabic{equation}}
\theoremstyle{plain}
\newtheorem{thm}{Theorem}[section]

\newtheorem{remark}{\textbf{Remark}}[section]

\usepackage{subfigure}
\usepackage{threeparttable}

\usepackage{amssymb}
\usepackage{lineno} 
\DeclareMathOperator{\artanh}{artanh}

\title{ Efficient and accurate exponential SAV algorithms with relaxation for dissipative system
\thanks{This work is supported by the National Natural Science Foundation of China  grants  12271302, 12131014 and 11971407.}
}

   \author{Yanrong Zhang\thanks{ Department of Applied Mathematics, The Hong Kong Polytechnic University, Hung Hom, Hong Kong. Email: yanrongzhang\_math@163.com}.
       \and Xiaoli Li\thanks{Corresponding Author. School of Mathematics, Shandong University, Jinan, 250100, China. Email: xiaolimath@sdu.edu.cn.}}

\begin{document}

\maketitle

\begin{abstract}
In this paper, we construct two kinds of exponential SAV approach with relaxation (R-ESAV) for dissipative system. The constructed schemes are linear and unconditionally energy stable. They can guarantee the positive property of SAV without any assumption compared with R-SAV and R-GSAV approaches, preserve all the advantages of the ESAV approach and satiesfy dissipation law with respect to  a modified energy which is directly related to the original free energy. Moreover  the second version of R-ESAV approach is easy to construct high-order BDF$k$ schemes. Especially for Navier-Stokes equations, we construct wo kinds of novel schemes based on the R-ESAV method.  Finally, ample numerical examples are presented to exhibit that the proposed approaches are accurate and effective.

\end{abstract}



\pagestyle{myheadings}
\thispagestyle{plain}

 \section{Introduction} 
 \label{sec:Intro}
Many significant scientific and engineering problems, such as complex fluids, new composite materials, the non-convex function optimization, etc., can be modeled by the 
dissipative system. From the numerical perspective, it's critical to maintain the discrete energy dissipation law so as to obviate non-physics numerical solutions. 
Over the past few decades, a large effort has been devoted to construct efficient energy stable time discretized schemes for the dissipative system. 
Existing and popular approaches can be classified into several categories: stabilized linearly implicit approach \cite{zhu1999coarsening, shen2010numerical}, exponential time differencing (ETD) approach \cite{wang2016efficient, du2019maximum, du2021maximum}, convex splitting approach \cite{elliott1993global, eyre1998unconditionally, shen2012second, baskaran2013convergence}, invariant energy quadratization (IEQ) approach \cite{yang2016linear, yang2017efficient, yang2018efficient, SheY20}, scalar auxiliary variable (SAV) approach \cite{shen2018scalar, shen2018convergence, shen2019new} and so on. 

Among these approaches, SAV method has become a very efficient and popular tool to construct energy stable schemes and has been successfully applied to gradient flow \cite{shen2018scalar, shen2019new, cheng2019highly, li2020stability, SheY20} and general dissipative systems \cite{ lin2019numerical, li2020error, li2020sav}. 
SAV method possesses a lot of attractive superiorities, but there are still some deficiencies that need to be improved. 
For example, 
(\romannumeral1) it requires to solve two linear systems at each time step;
(\romannumeral2) it requires that nonlinear part of free energy has a lower bound; 
(\romannumeral3) it satisfies unconditional energy stability according to modified energy rather than original energy. 
Recently there have been some corresponding improvements for these shortcomings. 
A generalized SAV approach which only requires solving one linear equation with constant coefficients has been proposed by Huang et al. \cite{huang2020highly, huang2021implicit}. Cheng et al.  \cite{cheng2020new, cheng2020global} proposed a novel Lagrange multiplier approach which dissipates original energy and do not require the nonlinear part of the free energy to be bounded from below. The constructed scheme 
requires solving a nonlinear algebraic equation which brings some additional computation costs and theoretical analysis difficulties, and may not exist a suitable solution when the time step is insufficiently small. Besides, some SAV approaches in more general form have been developed to extend its applicability in \cite{liu2020exponential, liu2021highly,cheng2021generalized}. The constructed schemes, especially for the exponential SAV methods, don't need to require that the nonlinear part of the free energy has a lower bound and can guarantee the positive property of SAV without any assumption. Very recently, the relaxation technique has been adopted to the SAV approach to improve the accuracy of numerical solutions  in \cite{jiang2021improving, zhang2022generalized}. The key idea in this approach is to make the modified energy link to the original energy closely by updating the auxiliary variable.

Our aim in this paper is to propose two kinds of relaxed exponential SAV (R-ESAV) approaches for dissipative system. The constructed schemes directly link the SAV to the free energy by introducing relaxation technique, and have outstanding advantages in the
following aspects:
\begin{itemize}
\item R-ESAV schemes are unconditionally energy stable with regard to a modified energy which is closer and directly linked to the original free energy, and can improve the accuracy of the solution noticeably compared with the original ESAV approach;
\item Only one linear system with constant coefficients needs to be solved;
\item The constructed schemes do not need the bounded below limitation of the nonlinear part of free energy;
\item The positive property of SAV without any assumption can be guaranteed.
\end{itemize}
In addition, our numerical results show that, R-ESAV schemes can improve the accuracy of SAV $\xi^{n+1}$, and for the plenty numerical simulations that we tested, the modified energy of our R-ESAV schemes all most always equals to the original free energy. 
Moreover, two kinds of BDF$k$ ($1 \leq k \leq 4$) numerical schemes based on two different decoupled approaches are constructed for Navier-Stokes equations. 
To the author's knowledge, it is the first time to apply relaxation technique to the general dissipative systems.

The remainder of this paper is organized as follows. 
In Section \ref{sec:R-ESAV-1}, we present the first version of relaxed exponential SAV approach for gradient flow. 
Then we extend the R-ESAV-1 approach to gradient systems with multiple components or multiple nonlinear potentials. 
In Section \ref{sec:R-ESAV-2}, we give the second version of relaxed exponential SAV approach for general dissipative system and construct two BDF$k$ schemes for Navier-Stokes equation. 
In Section \ref{sec:numerical-simulations}, we present several numerical experiments by using the new approaches, and provide ample numerical simulations to validates its generality and efficiency. 
In Section \ref{sec:conclusion}, we provide some concluding remarks.

 \section{The first version of relaxed exponential SAV approach}
 \label{sec:R-ESAV-1}
In this section, we consider the improvement of ESAV approach \cite{liu2020exponential} by introducing a relaxation factor to modified the SAV, which is abbreviated as R-ESAV-1 approach for convenience.

 \subsection{The R-ESAV-1 approach for gradient flow} 
Without losing generality, we consider a free energy given by  
\begin{equation}
	E(\phi)=\frac{1}{2}(\mathcal{L} \phi, \phi) + \int_{\Omega} F(\phi) \mathrm{d} \boldsymbol{x},
\end{equation}
where $\mathcal{L}$  is a self-adjoint linear  elliptic operator, $F(\phi)$ is a nonlinear energy density function.
Then, the gradient flow derived from the above free energy by energy-variational principle can be written as follows
\begin{equation} \label{eq:gradient-flow-single}
	\left\{\begin{aligned}
		& \frac{\partial \phi}{\partial t}=-\mathcal{G} \mu, \\
		& \mu=\frac{\delta E(\phi)}{\delta \phi}=\mathcal{L} \phi+F^{\prime}(\phi),
	\end{aligned}\right.
\end{equation}
where  $\mathcal{G}$ is a positive definite operator which signifies the dissipative mechanism of the system, e.g. $\mathcal{G}=I$ for the $L^2$ gradient flow and $\mathcal{G}=-\Delta$ for the $H^{-1}$ gradient flow.

We introduce an exponential SAV
\begin{equation}\label{eq:ESAV1-r}
	r(t) = \exp(E_{1}(\phi))=\exp(\int_{\Omega} F(\phi) \mathrm{d} \boldsymbol{x}),
\end{equation}
and by taking the derivative of \eqref{eq:ESAV1-r} with respect to $t$, we have
\begin{equation}\label{eq:ESAV1-dr}
\frac{\mathrm{d} r}{\mathrm{d} t}=r \int_{\Omega} F^{\prime}(\phi) \phi_{t} \mathrm{d} \boldsymbol{x}=\frac{r^{2}}{\exp \left(E_{1}(\phi)\right)} \int_{\Omega} F^{\prime}(\phi) \phi_{t} \mathrm{d} \boldsymbol{x}.
\end{equation}
Denote 
$
U(r, \phi)=\frac{r}{\exp \left(E_{1}(\phi)\right)} F^{\prime}(\phi),
$
and notice that equation \eqref{eq:ESAV1-dr} can be rewritten equivalently into 
\begin{equation}
\frac{\mathrm{d} \ln (r)}{\mathrm{d} t}=\left(U(r, \phi), \phi_{t}\right),
\end{equation}
then we rewrite the system \eqref{eq:gradient-flow-single} as follows
\begin{equation}
\left\{\begin{aligned}\label{eq:gradient-flow-single-ESAV} 
&\frac{\partial \phi}{\partial t} =-\mathcal{G} \mu, \\ 
&\mu =\mathcal{L} \phi+U(r, \phi), \\ 
&\frac{\mathrm{d} \ln (r)}{\mathrm{d} t}=\left(U(r, \phi), \phi_{t}\right).
\end{aligned}\right.
\end{equation}
By taking the inner product of the first two equations in \eqref{eq:gradient-flow-single-ESAV} with $\mu$ and $\frac{\partial \phi}{\partial t}$, respectively, we have
\begin{equation}
\frac{\mathrm{d}}{\mathrm{d} t}\left[\frac{1}{2}(\mathcal{L} \phi, \phi)+\ln (r)\right]=-(\mathcal{G} \mu, \mu) \leq 0.
\end{equation}

Inspired by the idea of relaxation factor in \cite{jiang2021improving,  zhang2022generalized}, we can construct R-ESAV-1/BDF$k$ ($1 \leq k \leq 6$) schemes:

Given $r^{n-k}, ..., r^{n}, \phi^{n-k}, ..., \phi^{n}$, we determine $r^{n+1}, \phi^{n+1}$ via two steps as follows:

\textbf{Step 1:} Compute an intermediate solution $(\tilde{r}^{n+1}, \phi^{n+1})$ by using  the ESAV approach:
\begin{eqnarray}
\label{eq:R-ESAV1-BDFk-1}
& & \frac{\alpha_{k} \phi^{n+1}-A_{k}\left(\phi^{n}\right)}{\delta t} =-\mathcal{G} \mu^{n+1}, \\
\label{eq:R-ESAV1-BDFk-2}
& & \mu^{n+1} =\mathcal{L} \phi^{n+1}+U\left(B_{k}\left(r^{n}\right), B_{k}\left(\phi^{n}\right)\right), \\ 
\label{eq:R-ESAV1-BDFk-3}
& & \frac{\alpha_{k}\ln(\tilde{r}^{n+1})-A_{k}\left(\ln(r^{n})\right)}{\delta t} =\left(U\left(B_{k}\left(r^{n}\right), B_{k}\left(\phi^{n}\right)\right), \frac{\alpha_{k} \phi^{n+1}-A_{k}\left(\phi^{n}\right)}{\delta t}\right), 
\end{eqnarray}
where $\alpha_{k}$,  $A_{k}$ and $B_{k}$ are related parameter and operators of BDF$k$ schemes, which can  be obtained  by Taylor expansion. For the convenience of readers, we provide the form of $k=1,2,3,4$ as follows:

First-order:
\begin{equation} \label{eq:perameter-BDF1}
	\alpha_{1}=1, \quad A_{1}\left(\phi^{n}\right)=\phi^{n}, \quad B_{1}\left(\phi^{n}\right)=\phi^{n};
\end{equation}
 
Second-order:
\begin{equation}\label{eq:perameter-BDF2}
	\alpha_{2}=\frac{3}{2}, \quad A_{2}\left(\phi^{n}\right)=2 \phi^{n}-\frac{1}{2} \phi^{n-1}, \quad B_{2}\left(\phi^{n}\right)=2 \phi^{n}-\phi^{n-1};
\end{equation}

Third-order:
\begin{equation}\label{eq:perameter-BDF3}
	\alpha_{3}=\frac{11}{6}, \quad A_{3}\left(\phi^{n}\right)=3 \phi^{n}-\frac{3}{2} \phi^{n-1}+\frac{1}{3} \phi^{n-2}, \quad B_{3}\left(\phi^{n}\right)=3 \phi^{n}-3 \phi^{n-1}+\phi^{n-2};
\end{equation}

Fourth-order:
\begin{equation}\label{eq:perameter-BDF4}
	\begin{aligned}
		& \alpha_{4}=\frac{25}{12}, \quad A_{4}\left(\phi^{n}\right)=4 \phi^{n}-3 \phi^{n-1}+\frac{4}{3} \phi^{n-2}-\frac{1}{4} \phi^{n-3}, \\
		& B_{4}\left(\phi^{n}\right)=4 \phi^{n}-6 \phi^{n-1}+4 \phi^{n-2}-\phi^{n-3}.
	\end{aligned}
\end{equation}

\textbf{Step 2:} Update $r^{n+1}$ via relaxation factor as follows: 
\begin{equation}\label{eq:R-ESAV1-BDFk-4}
	r^{n+1} = \theta_{0}^{n+1} \tilde{r}^{n+1} + (1-\theta_{0}^{n+1})\exp\left(E_{1}(\phi^{n+1})\right), \quad \theta_{0}^{n+1} \in \mathcal{W},
\end{equation}
where, $\mathcal{W}$ is a set defined as follows: \\
First-order:
\begin{equation}\label{eq:set-condition-ESAV1-BDF1}
	\mathcal{W}=\left\lbrace \theta\in [0,1] \; s.t. \; \ln(r^{n+1}) - \ln(\tilde r^{n+1}) \leq \delta t \gamma \left(\mathcal{G} \mu^{n+1}, \mu^{n+1}\right) \right\rbrace;
\end{equation}
Second-order:
\begin{equation}\label{eq:set-condition-ESAV1-BDF2}
	\mathcal{W}=\left\lbrace \theta\in [0,1] \; s.t. \;  \frac{3}{2}\ln(r^{n+1}) - \frac{3}{2}\ln(\tilde r^{n+1}) \leq \delta t \gamma \left(\mathcal{G} \mu^{n+1}, \mu^{n+1}\right) \right\rbrace;
\end{equation}
with $\gamma \in [0, 1]$ is a adjustable parameter. 

We explain below how to choose   $\theta_{0}^{n+1}$. 
For BDF$1$ scheme, plugging \eqref{eq:R-ESAV1-BDFk-4} into \eqref{eq:set-condition-ESAV1-BDF1}, we derive that if we choose  $\theta_0^{n+1}$ such that
\begin{equation}\label{eq:ESAV1-cond_zeta-BDF1}
\left(\tilde{r}^{n+1}-\exp\left(E_{1}\left(\phi^{n+1}\right)\right)\right)\theta_0^{n+1} \leq \exp \left(\delta t \gamma \left(\mathcal{G} \mu^{n+1}, \mu^{n+1}\right) + \ln(\tilde{r}^{n+1})\right)-\exp(E_{1}\left(\phi^{n+1})\right),
\end{equation}
then $\theta_0^{n+1}\in \mathcal{W}$. 
Similar to BDF$2$ scheme, we need to choose  $\theta_0^{n+1}$ satiesfy following inequlity  
\begin{equation}\label{eq:ESAV1-cond_zeta-BDF2}
\left(\tilde{r}^{n+1}-\exp\left(E_{1}\left(\phi^{n+1}\right)\right)\right)\theta_0^{n+1} \leq \exp \left(\frac{2}{3}\delta t \gamma \left(\mathcal{G} \mu^{n+1}, \mu^{n+1}\right) + \ln(\tilde{r}^{n+1})\right)-\exp\left(E_{1}(\phi^{n+1})\right).
\end{equation} 
Denote $S=\exp \left(\delta t \gamma \left(\mathcal{G} \mu^{n+1}, \mu^{n+1}\right) + \ln(\tilde{r}^{n+1})\right)$ for BDF$1$ scheme and $S=\exp \left(\frac{2}{3}\delta t \gamma \left(\mathcal{G} \mu^{n+1}, \mu^{n+1}\right)\right.$ $\left. + \ln(\tilde{r}^{n+1})\right)$ for BDF$2$ scheme, the next theorem summarizes the choice of $\theta_0^{n+1}$.

\begin{thm}\label{Th:ESAV1-stability}
We choose $\theta_{0}^{n+1}$ in \eqref{eq:R-ESAV1-BDFk-4} as follows:
\begin{enumerate}
\item If $\tilde{r}^{n+1} = \exp\left(E_{1}\left(\phi^{n+1}\right)\right)$, we set $\theta_{0}^{n+1}=0$.
\item If $\tilde{r}^{n+1} > \exp\left(E_{1}\left(\phi^{n+1}\right)\right)$, we set $\theta_{0}^{n+1}=0$. 
\item If $\tilde{r}^{n+1} < \exp\left(E_{1}\left(\phi^{n+1}\right)\right)$ and $S-\exp\left(E_{1}(\phi)\right) \geq 0$, we set $\theta_{0}^{n+1}=0$.  
\item If $\tilde{r}^{n+1} < \exp\left(E_{1}\left(\phi^{n+1}\right)\right)$ and $S-\exp\left(E_{1}(\phi)\right) < 0$, we set $\theta_{0}^{n+1}=\frac{S-\exp\left(E_{1}(\phi^{n+1})\right)}{\tilde{r}^{n+1}-\exp\left(E_{1}(\phi^{n+1})\right)}$.
	\end{enumerate}
Then, \eqref{eq:ESAV1-cond_zeta-BDF1} (resp. \eqref{eq:ESAV1-cond_zeta-BDF2}) for BDF$1$ (resp. BDF$2$) scheme is satisfied in all cases above and  $\theta_0^{n+1}\in \mathcal{W}$.
Moreover, we have $r^{n+1}>0$, and the scheme \eqref{eq:R-ESAV1-BDFk-1}-\eqref{eq:R-ESAV1-BDFk-4} with the above choice of  $\theta_0^{n+1}$ satiesfies unconditionally energy stability in the sense that:\\
First-order:
\begin{equation}\label{eq:ESAV1-stability-BDF1}
	R_{R-ESAV-BDF1}^{n+1} - R_{R-ESAV-BDF1}^{n} \leq -\delta t (1-\gamma) \left(\mathcal{G} \mu^{n+1}, \mu^{n+1}\right) \leq 0,
\end{equation}
where $R_{R-ESAV-BDF1}^{n+1}=\frac{1}{2}\left(\mathcal{L} \phi^{n+1}, \phi^{n+1}\right) +\ln(r^{n+1})$;\\
Second-order:
\begin{equation}\label{eq:ESAV1-stability-BDF2}
	R_{R-ESAV-BDF2}^{n+1} - R_{R-ESAV-BDF2}^{n} \leq -\delta t (1-\gamma) \left(\mathcal{G} \mu^{n+1}, \mu^{n+1}\right)\leq 0,
\end{equation}
where 
$R_{R-ESAV-BDF2}^{n+1}=\frac{1}{4}\left(\left(\mathcal{L} \phi^{n+1}, \phi^{n+1}\right) + \left(\mathcal{L}\left(2\phi^{n+1}-\phi^{n}\right), 2\phi^{n+1}-\phi^{n}\right)\right)+\frac{1}{2}\left(3\ln(r^{n+1})-\ln(r^{n})\right)$.
Furthermore, we have 
\begin{equation}
	r^{n+1}\le \exp\left(E_{1}(\phi^{n+1})\right), \quad\forall n\ge 0.
\end{equation}
\end{thm}
\begin{proof}
It can be verified easily that the above choice of $\theta_{0}^{n+1}$ satiesfies \eqref{eq:ESAV1-cond_zeta-BDF1} (resp. \eqref{eq:ESAV1-cond_zeta-BDF2}) for BDF$1$ (resp. BDF$2$) scheme in all cases such that $\theta_0^{n+1}\in \mathcal{W}$.

We can obtain that $\tilde r^{n+1} > 0$  from \eqref{eq:R-ESAV1-BDFk-3}, and thanks to $\exp\left(E_{1}(\phi^{n+1})\right)>0$, we have $r^{n+1}>0$.

Taking the inner product of \eqref{eq:R-ESAV1-BDFk-1}-\eqref{eq:R-ESAV1-BDFk-2} with $\mu^{n+1}$ and $\frac{\phi^{n+1}-\phi^{n}}{\delta t}$ (resp. $\frac{3 \phi^{n+1}-4\phi^{n}+\phi^{n-1}}{2\delta t}$), respectively, combining them with the equation \eqref{eq:R-ESAV1-BDFk-3} and \eqref{eq:set-condition-ESAV1-BDF1} (resp. \eqref{eq:set-condition-ESAV1-BDF2}), we can derive the discrete energy dissipation law \eqref{eq:ESAV1-stability-BDF1} (resp. \eqref{eq:ESAV1-stability-BDF2}) for BDF$1$ (resp. BDF$2$) scheme. 

For Cases 1-3,  we have $\theta_0^{n+1}=0$ so $r^{n+1}= E_{1}(\phi^{n+1})$. 
For Case 4, since 
$$\theta_{0}^{n+1}=\frac{S-\exp\left(E_{1}(\phi^{n+1})\right)}{\tilde{r}^{n+1}-\exp\left(E_{1}(\phi^{n+1})\right)}\in [0,1]$$ and $\tilde r^{n+1}<  \exp\left(E_{1}(\phi^{n+1})\right)$, we can obtain that $r^{n+1}\le \exp\left(E_{1}(\phi^{n+1})\right)$ from \eqref{eq:R-ESAV1-BDFk-4}. 
The proof is complete.
\end{proof}

\begin{remark}
 	To prevent the solution ``blowing up'' due to the rapid growth of exponential function, we can add a positive constant $C$ to redefine the exponential SAV $r(t) = \exp\left(\frac{E_{1}(\phi)}{C}\right)$, which is similar to \cite{liu2020exponential}. 
\end{remark}

 \subsection{The application of R-ESAV-1 approach for gradient flows of multiple functions}
 We provide below the R-ESAV-1 approach for gradient flows of multiple functions by considering the following energy functional
  \begin{equation}
E\left(\Phi\right)=\sum_{i,j=1}^{m}\frac{1}{2}d_{ij}(\phi_{i}, \mathcal{L} \phi_{j})+E_{1}\left(\Phi\right),
\end{equation}
where $\Phi=\left[\phi_{1}, \phi_{2}, \cdots, \phi_{m}\right]^{T}$, $\mathcal{L}$ is a self-adjoint linear positive definite operator, and the constant matrix $(d_{ij}), i, j=1, \cdots, m$ is symmetric positive definite. 
We set $W_{i}\left(\Phi\right)=\frac{\delta E_{1}\left(\Phi\right)}{\delta \phi_{i}}$, then the associated gradient flow is given by 
\begin{equation} \label{eq:gradient-flow-multiple}
	\left\{\begin{aligned}
		& \frac{\partial \phi_{i}}{\partial t}=-\mathcal{G} \mu_{i}, \\
		& \mu_{i}=\frac{\delta E}{\delta \phi_{i}}=\sum_{j=1}^{m}d_{ij} \mathcal{L} \phi_{j}+Q_{i}\left(\Phi\right),
\end{aligned}\right.
\end{equation}
where $\mathcal{G}$ is a nonnegative operator.
Taking the inner products of \eqref{eq:gradient-flow-multiple} with $\mu_{i}$ and $\frac{\partial \phi_{i}}{\partial t}$ respectively, summing over $i$, and thaks to the self-adjoint of $\mathcal{L}$ and $d_{ij}=d_{ji}$, we have the energy dissipation law as follows
\begin{equation}
\frac{\mathrm{d}}{\mathrm{d} t} E\left(\Phi\right)=\frac{\mathrm{d}}{\mathrm{d} t}\left\{\frac{1}{2} \sum_{i, j=1}^{m} d_{i j}\left(\phi_{i}, \mathcal{L} \phi_{j}\right)+E_{1}\left(\Phi\right)\right\}=-\sum_{i=1}^{m}\left(\mathcal{G} \mu_{i}, \mu_{i}\right) \leq 0.
\end{equation}
 
We construct numerical schemes for gradient flow of multiple functions by using R-ESAV-1 approach. 
 Introducing an exponential SAV 
 \begin{equation}
 	r(t)=\exp\left(E_{1}\left(\Phi\right)\right),
 \end{equation}
then \eqref{eq:gradient-flow-multiple} can be rewritten as 
\begin{equation}\label{eq:gradient-flow-multiple-ESAV} 
\left\{\begin{aligned}
& \frac{\partial \phi_{i}}{\partial t} =\mathcal{G} \mu_{i}, \\ 
& \mu_{i} =\sum_{j=1}^{m}d_{ij} \mathcal{L} \phi_{j}+\frac{r}{\exp\left(E_{1}\left(\Phi\right)\right)} Q_{i}\left(\Phi\right), \\ 
& \frac{\mathrm{d}\ln(r)}{\mathrm{d} t} =\frac{r}{\exp\left(E_{1}\left(\Phi\right)\right)}\sum_{i=1}^{m}\left(Q_{i}\left(\Phi\right), \frac{\partial \phi_{i}}{\partial t} \right).
\end{aligned}\right. 
\end{equation}
Denote 
$$U_{i}(r, \Phi)=\frac{r}{\exp\left(E_{1}\left(\Phi\right)\right)} Q_{i}\left(\Phi\right),$$
then the system \eqref{eq:gradient-flow-multiple-ESAV} can be simplified into
\begin{equation}\label{eq:gradient-flow-multiple-ESAV-1} 
\left\{\begin{aligned}
& \frac{\partial \phi_{i}}{\partial t} =\mathcal{G} \mu_{i}, \\ 
& \mu_{i} =\sum_{j=1}^{m}d_{ij} \mathcal{L} \phi_{j}+U_{i}(r, \Phi), \\ 
& \frac{\mathrm{d}\ln(r)}{\mathrm{d} t} =\sum_{i=1}^{m} \left(U_{i}(r, \Phi), \frac{\partial \phi_{i}}{\partial t} \right).
\end{aligned}\right. 
\end{equation}
Taking the inner product of the first two equations in \eqref{eq:gradient-flow-multiple-ESAV-1} with $\mu_{i}$ and $\frac{\partial \phi_{i}}{\partial t}$, respectively, 
combining them with the third equation in \eqref{eq:gradient-flow-multiple-ESAV-1}, and summing over $i$, we obtain the equivalent energy dissipation law
\begin{equation}
\frac{\mathrm{d}}{\mathrm{d} t}\left\{\frac{1}{2} \sum_{i, j=1}^{m} d_{i j}\left(\phi_{i}, \mathcal{L} \phi_{j}\right)+\ln(r)\right\}=\frac{\mathrm{d}}{\mathrm{d} t} E\left(\Phi\right)=-\sum_{i=1}^{m}\left(\mathcal{G} \mu_{i}, \mu_{i}\right) \leq 0.
\end{equation}

Next we can construct numerical scheme based on first version of R-ESAV approach and Crank-Nicolson formula (R-ESAV-1/CN) as follows. 
 
Given $r^{n-1}, r^{n}, \Phi^{n-1}, \Phi^{n}$, we determine $r^{n+1}, \Phi^{n+1}$ via two steps as follows:

\textbf{Step 1:} Compute an intermediate solution $(\tilde{r}^{n+1}, \Phi^{n+1})$ by using  the ESAV-1 approach:
\begin{eqnarray}
\label{eq:multiple-R-ESAV1-CN-1}
& & \frac{\phi_{i}^{n+1}-\phi_{i}^{n}}{\delta t} =-\mathcal{G} \mu_{i}^{n+\frac{1}{2}}, \\
\label{eq:multiple-R-ESAV1-CN-2}
& & \mu_{i}^{n+\frac{1}{2}} =\sum_{j=1}^{m}d_{ij} \mathcal{L} \phi_{j}^{n+\frac{1}{2}}+U_{i}\left(r^{*, n+\frac{1}{2}}, \Phi^{*, n+\frac{1}{2}}\right), \\ 
\label{eq:multiple-R-ESAV1-CN-3}
& & \frac{\ln(\tilde{r}^{n+1})-\ln(r^{n})}{\delta t} =\sum_{i=1}^{m} \left(U_{i}\left(r^{*, n+\frac{1}{2}}, \Phi^{*, n+\frac{1}{2}}\right), \frac{\phi_{i}^{n+1}-\phi_{i}^{n}}{\delta t}\right), 
\end{eqnarray}
where $g^{n+\frac{1}{2}}=\frac{g^{n+1}+g^{n}}{2}$ and $g^{*, n+\frac{1}{2}}$ can be any second-order explicit approximation of $g(t^{n+1/2})$, such as $g^{*, n+\frac{1}{2}}=\frac{3}{2}g^{n}-\frac{1}{2}g^{n-1}$.

\textbf{Step 2:} Update $r^{n+1}$ via relaxation factor as follows: 
\begin{equation}\label{eq:multiple-R-ESAV1-CN-4}
	r^{n+1} = \theta_{0}^{n+1} \tilde{r}^{n+1} + (1-\theta_{0}^{n+1})\exp\left(E_{1}(\Phi^{n+1})\right), \quad \theta_{0}^{n+1} \in \mathcal{W},
\end{equation}
where, $\mathcal{W}$ is a set defined as follows: 
\begin{equation}\label{eq:multiple-set-condition-ESAV1-CN}
	\mathcal{W}=\left\lbrace \theta\in [0,1] \; s.t. \; \ln(r^{n+1}) - \ln(\tilde r^{n+1}) \leq \delta t \gamma \sum_{i=1}^{m}\left(\mathcal{G} \mu_{i}^{n+\frac{1}{2}}, \mu_{i}^{n+\frac{1}{2}}\right) \right\rbrace,
\end{equation}
with $\gamma \in [0, 1]$ is a adjustable parameter. 

Inserting \eqref{eq:multiple-R-ESAV1-CN-4} into the inequality of \eqref{eq:multiple-set-condition-ESAV1-CN}, we observe that if we choose  $\theta_0^{n+1}$ satiesfies following condition
\begin{equation}\label{eq:multiple-ESAV1-cond_zeta-CN}
\left(\tilde{r}^{n+1}-\exp\left(E_{1}\left(\Phi^{n+1}\right)\right)\right)\theta_0^{n+1} \leq \exp \left(\delta t \gamma \sum_{i=1}^{m}\left(\mathcal{G} \mu_{i}^{n+\frac{1}{2}}, \mu_{i}^{n+\frac{1}{2}}\right) + \ln(\tilde{r}^{n+1})\right)-\exp(E_{1}\left(\Phi^{n+1})\right),
\end{equation}
then $\theta_0^{n+1}\in \mathcal{W}$. 
Denote $S=\exp \left(\delta t \gamma \sum_{i=1}^{m}\left(\mathcal{G} \mu_{i}^{n+\frac{1}{2}}, \mu_{i}^{n+\frac{1}{2}}\right) + \ln(\tilde{r}^{n+1})\right)$, then we can choose $\theta_{0}^{n+1}$ according to Theorem \ref{Th:ESAV1-stability} and the scheme \eqref{eq:multiple-R-ESAV1-CN-1}-\eqref{eq:multiple-R-ESAV1-CN-4} satisfies the unconditional energy stability in the sense that
\begin{equation}\label{eq:multiple-ESAV1-stability-CN}
	R_{R-ESAV-CN}^{n+1} - R_{R-ESAV-CN}^{n} \leq -\delta t (1-\gamma) \sum_{i=1}^{m}\left(\mathcal{G} \mu_{i}^{n+\frac{1}{2}}, \mu_{i}^{n+\frac{1}{2}}\right) \leq 0,
\end{equation}
where $R_{R-ESAV-CN}^{n+1}=\frac{1}{2} \sum_{i, j=1}^{m} d_{i j}\left(\phi_{i}^{n+1}, \mathcal{L} \phi_{j}^{n+1}\right) +\ln(r^{n+1})$.

\subsection{Extension to the multiple ESAV-1 approach}
 \label{sec:R-ESAV-multiple}
 In this subsection, we present how to construct relaxed multiple ESAV-1 (R-MESAV-1) schemes for gradient flow, where the model may include disparate terms such that original schemes with only one SAV has limitation on describing the different disparate evolution processes and may require overly small time steps to obtain accurate numerical solution \cite{cheng2018multiple}.
 
Without losing generality, we study gradient flow with  two disparate  nonlinear terms as follows and it can be easily extended to more than two disparate nonlinear terms 
\begin{equation} \label{eq:model-problem-MESAV}
	\left\{\begin{aligned}
		& \frac{\partial \phi}{\partial t}=-\mathcal{G} \mu, \\
		& \mu=\mathcal{L} \phi+F_{1}^{\prime}(\phi) + F_{2}^{\prime}(\phi),
	\end{aligned}\right.
\end{equation}
where $\mathcal{L}$  is a self-adjoint  linear  elliptic operator, $F_{1}(\phi), F_{2}(\phi)$ are nonlinear potential function,  $\mathcal{G}$ is a linear positive definite operator.  
The system \eqref{eq:model-problem-MESAV} satisfies the following energy dissipation law
\begin{align}
	\frac{\mathrm{d} E(\phi)}{\mathrm{d} t} =-\left(\mathcal{G}\mu, \mu\right),
\end{align} 
where the total energy is 
\begin{equation}
	E(\phi)=\frac{1}{2}(\mathcal{L} \phi, \phi) + \int_{\Omega} F_{1}(\phi) \mathrm{d} \boldsymbol{x}+\int_{\Omega} F_{2}(\phi) \mathrm{d} \boldsymbol{x}.
\end{equation}

 We first consider relaxed MESAV approach based on first ESAV approach. 
 Setting $E_{1}(\phi)=\int_{\Omega} F_{1}(\phi) \mathrm{d} \boldsymbol{x},\, E_{2}(\phi)=\int_{\Omega} F_{2}(\phi) \mathrm{d} \boldsymbol{x}$ and introducing two SAVs $r_{1}(t)=\exp\left(E_{1}(\phi)\right),\, r_{2}(t)=\exp\left(E_{2}(\phi)\right)$, we can rewrite the system \eqref{eq:model-problem-MESAV} as
 \begin{equation} \label{eq:model-problem-MESAV1}
	\left\{\begin{aligned}
		& \frac{\partial \phi}{\partial t}=-\mathcal{G} \mu, \\
		& \mu=\mathcal{L} \phi+ \frac{r_{1}(t)}{\exp\left(E_{1}(\phi)\right)}F_{1}^{\prime}(\phi) + \frac{r_{2}(t)}{\exp\left(E_{2}(\phi)\right)}F_{2}^{\prime}(\phi),\\
		& \frac{\mathrm{d} r_{1}(t)}{\mathrm{d} t}=\frac{r_{1}^{2}}{\exp \left(E_{1}(\phi)\right)} \left( F_{1}^{\prime}(\phi), \phi_{t} \right), \\
		& \frac{\mathrm{d} r_{2}(t)}{\mathrm{d} t}=\frac{r_{2}^{2}}{\exp \left(E_{2}(\phi)\right)} \left( F_{2}^{\prime}(\phi), \phi_{t} \right).
	\end{aligned}\right.
\end{equation}
Denote $U_{1}(r_{1}, \phi)=\frac{r_{1}}{\exp \left(E_{1}(\phi)\right)} F_{1}^{\prime}(\phi)$ and $U_{2}(r_{2}, \phi)=\frac{r_{2}}{\exp \left(E_{2}(\phi)\right)} F_{2}^{\prime}(\phi)$, then \eqref{eq:model-problem-MESAV1} can be transformed as 
\begin{equation}
\left\{\begin{aligned}\label{eq:gradient-flow-MESAV} 
&\frac{\partial \phi}{\partial t} =-\mathcal{G} \mu, \\ 
&\mu =\mathcal{L} \phi+U_{1}(r_{1}, \phi)+U_{2}(r_{2}, \phi), \\ 
&\frac{\mathrm{d} \ln (r_{1})}{\mathrm{d} t}=\left(U_{1}(r_{1}, \phi), \phi_{t}\right), \\
&\frac{\mathrm{d} \ln (r_{2})}{\mathrm{d} t}=\left(U_{2}(r_{2}, \phi), \phi_{t}\right).
\end{aligned}\right.
\end{equation}

Then we can construct R-MESAV-1/CN schemes inspired by the idea of relaxation factor.

Given $\phi^{n-1}, \phi^{n}, r_{1}^{n-1}, r_{1}^{n}, r_{2}^{n-1}, r_{2}^{n}$, we determine $\phi^{n+1}, r_{1}^{n+1}, r_{2}^{n+1}$ via two steps as follows:

\textbf{Step 1:} Compute an intermediate solution $(\phi^{n+1}, \tilde{r}_{1}^{n+1}, \tilde{r}_{2}^{n+1})$ by using  the MESAV approach:
\begin{eqnarray}
\label{eq:R-MESAV1-CN-1}
& & \frac{\phi^{n+1}-\phi^{n}}{\delta t} =-\mathcal{G} \mu^{n+\frac{1}{2}}, \\
\label{eq:R-MESAV1-CN-2}
& & \mu^{n+\frac{1}{2}} =\mathcal{L} \phi^{n+\frac{1}{2}}+U_{1}\left(r_{1}^{*, n+\frac{1}{2}}, \phi^{*, n+\frac{1}{2}}\right)+U_{2}\left(r_{2}^{*, n+\frac{1}{2}}, \phi^{*, n+\frac{1}{2}}\right), \\ 
\label{eq:R-MESAV1-CN-3}
& & \frac{\ln(\tilde{r}_{1}^{n+1})-\ln(r_{1}^{n})}{\delta t} =\left(U_{1}\left(r_{1}^{*, n+\frac{1}{2}}, \phi^{*, n+\frac{1}{2}}\right), \frac{\phi^{n+1}-\phi^{n}}{\delta t}\right), \\
\label{eq:R-MESAV1-CN-4}
& & \frac{\ln(\tilde{r}_{2}^{n+1})-\ln(r_{2}^{n})}{\delta t} =\left(U_{2}\left(r_{2}^{*, n+\frac{1}{2}}, \phi^{*, n+\frac{1}{2}}\right), \frac{\phi^{n+1}-\phi^{n}}{\delta t}\right).
\end{eqnarray}

\textbf{Step 2:} Update the SAVs $r_{1}^{n+1}, r_{2}^{n+1}$ via relaxation factor as follows: 
\begin{equation}\label{eq:R-MESAV1-CN-5}
\begin{aligned}
& r_{1}^{n+1} = \theta_{0}^{n+1} \tilde{r}_{1}^{n+1} + (1-\theta_{0}^{n+1})\exp\left(E_{1}(\phi^{n+1})\right), \quad \theta_{0}^{n+1} \in \mathcal{W}, \\
& r_{2}^{n+1} = \theta_{0}^{n+1} \tilde{r}_{2}^{n+1} + (1-\theta_{0}^{n+1})\exp\left(E_{2}(\phi^{n+1})\right), \quad \theta_{0}^{n+1} \in \mathcal{W},
\end{aligned}
\end{equation}
where $\mathcal{W}$ is a set defined as follows: 
\begin{equation}\label{eq:set-condition-MESAV1-CN}
	\mathcal{W}=\left\lbrace \theta\in [0,1] \; s.t. \; \ln(r_{1}^{n+1})+\ln(r_{2}^{n+1}) - \ln(\tilde r_{1}^{n+1}) - \ln(\tilde r_{2}^{n+1}) \leq \delta t \gamma \left(\mathcal{G} \mu^{n+\frac{1}{2}}, \mu^{n+\frac{1}{2}}\right) \right\rbrace
\end{equation}
with $\gamma \in [0, 1]$ is a adjustable parameter. 

We explain below how to choose $\theta_{0}^{n+1}$. 
Plugging \eqref{eq:R-MESAV1-CN-5} into  \eqref{eq:set-condition-MESAV1-CN}, we derive that if we choose  $\theta_0^{n+1}$ such that 
\begin{equation}\label{eq:MESAV1-cond_zeta-CN}
\begin{aligned}
& \left[\left(\tilde{r}_{1}^{n+1}-\exp\left(E_{1}\left(\phi^{n+1}\right)\right)\right)\theta_0^{n+1}+\exp\left(E_{1}(\phi^{n+1})\right)\right] \left[\left(\tilde{r}_{2}^{n+1}-\exp\left(E_{2}\left(\phi^{n+1}\right)\right)\right)\theta_0^{n+1}+\exp\left(E_{2}(\phi^{n+1})\right)\right]\\
& \leq \exp \left(\delta t \gamma \left(\mathcal{G} \mu^{n+\frac{1}{2}}, \mu^{n+\frac{1}{2}}\right) + \ln(\tilde{r}_{1}^{n+1}) + \ln(\tilde{r}_{2}^{n+1})\right),
\end{aligned}
\end{equation}
then $\theta_0^{n+1}\in \mathcal{W}$. 
And $\theta_{0}^{n+1}$ can be regarded as a solution of the optimization problem as follows
\begin{equation}\label{eq: R-ESAV-2-quadratic}
\theta_{0}^{n+1}=\min _{\theta \in[0,1]} \theta, \quad \text { s.t. } f(\theta)=a_{1}a_{2} \theta^{2}+(a_1b_2+a_2b_1) \theta+ b_1 b_2-c \leq 0,
\end{equation}
where the coefficients are 
\begin{equation*}
\begin{array}{l}
a_{1}=\tilde{r}_{1}^{n+1}-\exp\left(E_{1}\left(\phi^{n+1}\right)\right), \quad b_{1}=\exp\left(E_{1}(\phi^{n+1})\right), \\ 
a_{2}=\tilde{r}_{2}^{n+1}-\exp\left(E_{2}\left(\phi^{n+1}\right)\right), \quad b_{2}=\exp\left(E_{2}(\phi^{n+1})\right), \\ 
c=\exp \left(\delta t \gamma \left(\mathcal{G} \mu^{n+\frac{1}{2}}, \mu^{n+\frac{1}{2}}\right) + \ln(\tilde{r}_{1}^{n+1}) + \ln(\tilde{r}_{2}^{n+1})\right).\end{array}
\end{equation*}
Denote 
\begin{eqnarray*}
	& \theta_{1}=\frac{-(a_1b_2+a_2b_1)-\sqrt{(a_1b_2+a_2b_1)^2-4a_1a_2(b_1 b_2-c)}}{2a_1a_2}, \\
	& \theta_{2}=\frac{-(a_1b_2+a_2b_1)+\sqrt{(a_1b_2+a_2b_1)^2-4a_1a_2(b_1 b_2-c)}}{2a_1a_2},
\end{eqnarray*}
and the next theorem summarizes the choice of $\theta_0^{n+1}$. 

\begin{thm}\label{Th:MESAV1-stability}
We choose $\theta_{0}^{n+1}$ in \eqref{eq:R-MESAV1-CN-5} as follows:
\begin{enumerate}
\item If $a_{1}=0$ and $a_{2}=0$, we set $\theta_{0}^{n+1}=0$.
\item If $a_1 > 0$ and $a_2 > 0$, we set $\theta_{0}^{n+1}=0$.
\item If $a_1 < 0$ and $a_2 < 0$, we set $\theta_{0}^{n+1}=\max \left\{0, \theta_{1}\right\}$.
\item If $a_1 a_2 < 0$ and $b_1b_2-c \leq 0$, we set $\theta_{0}^{n+1}=0$. 
\item If $a_1 a_2 < 0$ and $b_1b_2-c > 0$, we set $\theta_{0}^{n+1}=\theta_{1}$. 
\item If $a_{1}a_{2}=0$ and $a_1b_2+a_2b_1>0$, we set $\theta_{0}^{n+1}=0$.  
\item If $a_{1}a_{2}=0$ and $a_1b_2+a_2b_1<0$, we set $\theta_{0}^{n+1}=\max \left\{0, \frac{c-b_1b_2}{a_1b_2+a_2b_1}\right\}$.
	\end{enumerate}
Then, \eqref{eq:MESAV1-cond_zeta-CN} for scheme \eqref{eq:R-MESAV1-CN-1}-\eqref{eq:R-MESAV1-CN-5} is satisfied in all cases above and  $\theta_0^{n+1}\in \mathcal{W}$.
Moreover, we have $r_{1}^{n+1}>0, r_{2}^{n+1}>0$, and the scheme \eqref{eq:R-MESAV1-CN-1}-\eqref{eq:R-MESAV1-CN-5} with the above choice of  $\theta_0^{n+1}$ is unconditionally energy stable in the sense that:
\begin{equation}\label{eq:MESAV1-stability-CN}
	R_{R-MESAV-CN}^{n+1} - R_{R-MESAV-CN}^{n} \leq -\delta t (1-\gamma) \left(\mathcal{G} \mu^{n+1}, \mu^{n+1}\right) \leq 0,
\end{equation}
where $R_{R-MESAV-CN}^{n+1}=\frac{1}{2}\left(\mathcal{L} \phi^{n+1}, \phi^{n+1}\right) +\ln(r_{1}^{n+1})+\ln(r_{2}^{n+1})$. 
\end{thm}
\begin{proof} 
	We find optimal relaxation $\theta_{0}^{n+1}$ by discussing the coefficient of \eqref{eq: R-ESAV-2-quadratic} case by case. 
\begin{itemize}
  	\item If $a_{1}=0$ and $a_{2}=0$, we have $\theta_{0}^{n+1}=0$ obviously. 
 	\item If $a_{1}a_{2} \neq 0$, notice that 
$$f(1)=\tilde{r}_{1}^{n+1}\tilde{r}_{2}^{n+1}-\exp \left(\delta t \gamma \left(\mathcal{G} \mu^{n+\frac{1}{2}}, \mu^{n+\frac{1}{2}}\right) + \ln(\tilde{r}_{1}^{n+1}) + \ln(\tilde{r}_{2}^{n+1})\right)\leq 0,$$
 we have $1\in \mathcal{W}$ which means $\mathcal{W} \neq \emptyset. $ 
  \begin{itemize}
	\item If $a_1 a_2>0$, and thanks to $f(1)\leq 0$, we have $\theta_{0}^{n+1}=\max \left\{0, \theta_{1}\right\}$, then we can easily derive Case 2, 3. 
	\item If $a_1 a_2 < 0$, and $f(0)=b_1b_2-c \leq 0$, we have $\theta_{0}^{n+1}=0$. 
	\item If $a_1 a_2 < 0$, and $f(0)=b_1b_2-c > 0, f(1) \leq 0$, we have $\theta_{0}^{n+1}=\max\left\{\theta_{1}, \theta_{2}\right\}$. And since $\theta_{1} \geq \theta_{2}$, then $\theta_{0}^{n+1}=\theta_{1}$. 
\end{itemize}
	\item If $a_{1}a_{2}=0$, and $a_1b_2+a_2b_1>0$, thanks to $f(1) \leq 0$, we have $\theta_{0}^{n+1}=0$. \item If $a_{1}a_{2}=0$, and $a_1b_2+a_2b_1<0$, since $f(1) \leq 0$, then we have $\theta_{0}^{n+1}=\max \left\{0, \frac{c-b_1b_2}{a_1b_2+a_2b_1}\right\}$.
\end{itemize}


We derive from \eqref{eq:R-MESAV1-CN-3}-\eqref{eq:R-MESAV1-CN-4} that $\tilde r_{1}^{n+1} > 0, \tilde r_{2}^{n+1} > 0$, and thanks to $\exp\left(E_{1}(\phi^{n+1})\right)>0, \exp\left(E_{2}(\phi^{n+1})\right)>0$, we have $r_{1}^{n+1}>0, r_{2}^{n+1}>0$.

Taking the inner product of \eqref{eq:R-MESAV1-CN-1}-\eqref{eq:R-MESAV1-CN-2} with $\mu^{n+1}$ and $\frac{\phi^{n+1}-\phi^{n}}{\delta t}$, respectively, combining them with the equation \eqref{eq:R-MESAV1-CN-3}-\eqref{eq:R-MESAV1-CN-4} and \eqref{eq:set-condition-MESAV1-CN}, we can obtain the desired result \eqref{eq:MESAV1-stability-CN}. 
\end{proof}

\section{The second version of relaxed exponential SAV approach}
\label{sec:R-ESAV-2}
Inspired by the total energy based on exponential SAV approach in \cite{liu2021highly}, we construct the second version of relaxed exponential SAV approach(R-ESAV-2) for dissipative system in this section.

\subsection{The R-ESAV-2 approach for dissipative system}
More generally, we consider dissipative system
\begin{align}\label{eq:dissipative-system}
 \frac{\partial \phi}{\partial t}+\mathcal{A} \phi +g(\phi)=0, 
\end{align} 
where $\mathcal{A}$ is a positive operator and $g(\phi)$ is a semi-linear or quasi-linear operator. Assume it  satisfies the following energy dissipation law
\begin{align}
	\frac{\mathrm{d} E(\phi)}{\mathrm{d} t} =-\mathcal{K}(\phi),\label{eq: new2}
\end{align} 
where $E(\phi)>-C_{0}$ is a free energy and $\mathcal{K}(\phi)>0$ for all $\phi$.
Introducing a SAV $R(t)=\exp\left(E(\phi)\right)$, we transform the equation \eqref{eq:dissipative-system} into the equivalent system as follows
\begin{equation}\label{eq:dissipative-system-ESAV}
\left\{\begin{aligned}
& \frac{\partial \phi}{\partial t}+\mathcal{A} (\phi)+V(\xi)g(\phi)=0, \\
& \frac{\mathrm{d} R(t)}{\mathrm{d} t}=-R(t)\mathcal{K}(\phi),\\
& \xi = \frac{R(t)}{\exp\left(E(\phi)\right)},
\end{aligned}\right.
\end{equation}
where $V(\xi)$ is a function related to $\xi$ and $V(\xi) \equiv 1$ at a continuous level. 

Then we can construct R-ESAV-2/BDF$k$ ($1 \leq k \leq 6$) schemes inspired by the idea of relaxation factor.

Given $R^{n-k}, ..., R^{n}, \phi^{n-k}, ..., \phi^{n}$, we determine $R^{n+1}, \phi^{n+1}$ via two steps as follows:

 \textbf{Step 1:} Compute an intermediate solution $(\tilde{R}^{n+1}, \phi^{n+1})$ by using the ESAV approach:
\begin{eqnarray}
\label{eq:R-ESAV2-BDFk-1}
& & \frac{\alpha_{k} \phi^{n+1}-A_{k}\left(\phi^{n}\right)}{\delta t}+\mathcal{A} \phi^{n+1}+V_{k}(\xi^{n+1})g\left(B_{k}\left(\phi^{n}\right)\right)=0, \\
\label{eq:R-ESAV2-BDFk-2}
& & \frac{1}{\delta t}\left(\tilde{R}^{n+1}-R^{n}\right)=-\tilde{R}^{n+1} \mathcal{K}\left(B_{k}\left(\phi^{n}\right)\right), \\
\label{eq:R-ESAV2-BDFk-3}
& & \xi^{n+1}=\frac{\tilde{R}^{n+1}}{\exp\left(E\left(B_{k}\left(\phi^{n}\right)\right)\right)}, 	
\end{eqnarray}
where $\alpha_{k}$, $A_{k}$ and $B_{k}$ are defined as above and $V_{k}(\xi^{n+1})$ for $k$th-order ($1\leq k\leq 4$) scheme are following:\\
First-order:
$V_{1}(\xi^{n+1})=\xi^{n+1}$;\\
Second-order:
$V_{2}(\xi^{n+1})=\xi^{n+1}(2-\xi^{n+1})$;\\
Third-order:
$V_{3}\left(\xi^{n+1}\right)=\xi^{n+1}\left(3-3 \xi^{n+1}+\left(\xi^{n+1}\right)^{2}\right)$;\\
Forth-order:
$V_{4}\left(\xi^{n+1}\right)=\xi^{n+1}\left(2-\xi^{n+1}\right)\left(2-2 \xi^{n+1}+\left(\xi^{n+1}\right)^{2}\right)$.

\textbf{Step 2:} Update the SAV $R^{n+1}$ via relaxation factor as follows: 
\begin{equation} \label{eq:R-ESAV2-BDFk-4}
	R^{n+1} = \theta_{0}^{n+1} \tilde{R}^{n+1} + (1-\theta_{0}^{n+1})\exp\left(E\left(\phi^{n+1}\right)\right), \quad \theta_{0}^{n+1} \in \mathcal{W},
\end{equation}
where 
\begin{eqnarray}
\label{eq:set-condition-ESAV2}
\mathcal{W}=\left\lbrace \theta\in [0,1] \; s.t. \;   \frac{R^{n+1}-\tilde{R}^{n+1}}{\delta t} = -\gamma^{n+1} \tilde{R}^{n+1} \mathcal{K}( {\phi}^{n+1})+ \tilde{R}^{n+1}\mathcal{K}\left(B_{k}\left(\phi^{n}\right)\right)\right\rbrace,
\end{eqnarray}
and  $\gamma^{n+1} \geq 0$ is to be determined such that the set $\mathcal{W}$ is not empty. 

 Then we describe how to choose   $\theta_{0}^{n+1}$ and  $\gamma^{n+1}$. 
Insertting \eqref{eq:R-ESAV2-BDFk-4} into the equality of the set  \eqref{eq:set-condition-ESAV2}, we derive that if we choose  $\theta_0^{n+1}$ and $\gamma^{n+1}$ such that 
\begin{equation}\label{eq:ESAV2-cond_zeta}
\left(\tilde{R}^{n+1}-\exp\left(E(\phi^{n+1})\right)\right)\theta_0^{n+1} = \tilde{R}^{n+1}-\exp\left(E(\phi^{n+1})\right)-\delta t \gamma^{n+1} \tilde{R}^{n+1}\mathcal{K}( {\phi}^{n+1})+ \delta t \tilde{R}^{n+1} \mathcal{K}\left(B_{k}(\phi^{n})\right),
\end{equation}
then, $\theta_0^{n+1}\in \mathcal{W}$. The summation of the choice of $\theta_0^{n+1}$ and $\gamma^{n+1}$ are provided in  next theorem.

\begin{thm}\label{Th:ESAV2-stability}
The choice of $\theta_{0}^{n+1}$ in \eqref{eq:R-ESAV2-BDFk-4} and $\gamma^{n+1}$ in \eqref{eq:set-condition-ESAV2} are shown as follows:
\begin{enumerate}
\item If $\tilde{R}^{n+1} = \exp\left(E(\phi^{n+1})\right)$, we set $\theta_{0}^{n+1}=0$ and $\gamma^{n+1}=\frac{\mathcal{K}\left( B_{k}(\phi^{n})\right)}{\mathcal{K}( {\phi}^{n+1})}$.
\item If $\tilde{R}^{n+1} > \exp\left(E(\phi^{n+1})\right)$, we set $\theta_{0}^{n+1}=0$ and
\begin{equation}\label{eq:ESAV2-gamma}
\gamma^{n+1} = \frac{\tilde{R}^{n+1}-\exp\left(E(\phi^{n+1})\right)}{\delta t \tilde{R}^{n+1}\mathcal{K}(\phi^{n+1})} + \frac{\mathcal{K}\left( B_{k}(\phi^{n})\right)}{\mathcal{K}( {\phi}^{n+1})}.
\end{equation}
\item If $\tilde{R}^{n+1} < \exp\left(E(\phi^{n+1})\right)$ and $\tilde{R}^{n+1}-\exp\left(E(\phi^{n+1})\right)+ \delta t \tilde{R}^{n+1} \mathcal{K}\left( B_{k}(\phi^{n})\right) \geq 0$, we set $\theta_{0}^{n+1}=0$ and 	
$\gamma^{n+1}$ the same as \eqref{eq:ESAV2-gamma}. 	
\item  If $\tilde{R}^{n+1} < \exp\left(E(\phi^{n+1})\right)$ and $\tilde{R}^{n+1}-\exp\left(E(\phi^{n+1})\right)+ \delta t \tilde{R}^{n+1}\mathcal{K}\left( B_{k}(\phi^{n})\right)< 0$, we set $\theta_{0}^{n+1}=1-\frac{\delta t \tilde{R}^{n+1} \mathcal{K}\left( B_{k}(\phi^{n})\right)}{\exp\left(E(\phi^{n+1})\right)-\tilde{R}^{n+1}}$ and 
$\gamma^{n+1}=0$.
	\end{enumerate}
Then,  \eqref{eq:ESAV2-cond_zeta} is satisfied in all cases above and  $\theta_0^{n+1}\in \mathcal{W}$. 
Moreover, we have $R^{n+1}> 0,\, \xi_{k}^{n+1} > 0$, and the scheme \eqref{eq:R-ESAV2-BDFk-1}-\eqref{eq:R-ESAV2-BDFk-4} with the above choice of  $\theta_0^{n+1}$ and  $\gamma^{n+1}$  satiesfies unconditionally energy stability in the sense that
\begin{equation}\label{eq:ESAV2-stability}
	R^{n+1} - R^{n} = -\delta t \gamma^{n+1} \tilde{R}^{n+1} \mathcal{K}( {\phi}^{n+1})\leq 0,
\end{equation}
and more importantly we have 
\begin{equation}\label{eq:ESAV2-stability-1}
	ln\left(R^{n+1}\right) - \ln\left(R^{n}\right) \leq 0.
\end{equation}
Furthermore, we have 
\begin{equation}\label{eq:stability2}
	R^{n+1}\le \exp\left(E(\phi^{n+1})\right), \quad\forall n\ge 0.
\end{equation}


\end{thm}

\begin{proof}
	It can be verified easily that the above choice of $\theta_{0}^{n+1}$ and $\gamma^{n+1}$ satiesfies \eqref{eq:ESAV2-cond_zeta} in all cases such that $\theta_0^{n+1}\in \mathcal{W}$.
	
Since $R^{0} > 0$. 
It follows from \eqref{eq:R-ESAV2-BDFk-2} that 
\begin{equation}\label{tilder}
	\tilde{R}^{1}=\frac{R^{0}}{1+\delta t \mathcal{K}\left(B_{k}(\phi^{0})\right)} > 0.
\end{equation}
Then we derive from \eqref{eq:R-ESAV2-BDFk-3} that $\xi_{k}^{1}> 0$, and we  derive from \eqref{eq:R-ESAV2-BDFk-4} that $R^{1} > 0$. 
Therefore, it is easy to obtain $\xi_{k}^{n+1}> 0$ and  $R^{n+1} > 0$ by induction method.

Then we obtain \eqref{eq:ESAV2-stability} by combining \eqref{eq:R-ESAV2-BDFk-2} and \eqref{eq:set-condition-ESAV2}. 

For Cases 1-3,  it can obtain that $\theta_0^{n+1}=0$, then we have $R^{n+1}= \exp\left(E(\phi^{n+1})\right)$. 
For Case 4, thanks to 
$\theta_{0}^{n+1}=1-\frac{\delta t \tilde{R}^{n+1} \mathcal{K}\left( B_{k}(\phi^{n})\right)}{\exp\left(E(\phi^{n+1})\right)-\tilde{R}^{n+1}}\in (0,1]$ and $\tilde R^{n+1}<  \exp\left(E(\phi^{n+1})\right)$, we derive that $R^{n+1}\le \exp\left(E(\phi^{n+1})\right)$ from \eqref{eq:R-ESAV2-BDFk-4} .
\end{proof}

 
\begin{remark}
As a further extension, firstly, we can also construct numerical schemes for gradient flows of multiple functions \eqref{eq:gradient-flow-multiple} by using the R-ESAV-2 approach, and $\theta_0^{n+1}, \gamma^{n+1}$ can be chosen  similarly  according to Theorem \ref{Th:ESAV2-stability}. 
Secondly, the R-ESAV-2 approach can be also extended to multiple ESAV form based on gradient flow with two disparate nonlinear terms \eqref{eq:model-problem-MESAV}. 
 Setting $E_{1}(\phi)=\frac{1}{2}(\mathcal{L} \phi, \phi) + \int_{\Omega} F_{1}(\phi) \mathrm{d} \boldsymbol{x},\, E_{2}(\phi)=\int_{\Omega} F_{2}(\phi) \mathrm{d} \boldsymbol{x}$ and introducing two SAVs $R_{1}(t)=\exp\left(E_{1}(\phi)\right),\, R_{2}(t)=\exp\left(E_{2}(\phi)\right)$, we can rewrite the equation \eqref{eq:model-problem-MESAV} as
\begin{equation}\label{eq:model-problem-reformulation-MESAV2}
	\left\{\begin{aligned}
		& \frac{\partial \phi}{\partial t}=-\mathcal{G} \mu, \\
		& \mu=\mathcal{L} \phi+V(\xi_{1})F_{1}^{\prime}(\phi)+V(\xi_{2})F_{2}^{\prime}(\phi), \\
		& \frac{\mathrm{d}\ln\left(R_{1}(t)\right)}{\mathrm{d}t}=-\left(\mathcal{G}\frac{\delta E_{1}}{\delta \phi}, \mu\right),\\
		& \frac{\mathrm{d}\ln\left(R_{2}(t)\right)}{\mathrm{d}t}=-\left(\mathcal{G}\frac{\delta E_{2}}{\delta \phi}, \mu\right), \\
		& \xi_{1} = \frac{R_{1}(t)}{\exp\left(E_{1}(\Phi)\right)}, \, \xi_{2} = \frac{R_{2}(t)}{\exp\left(E_{2}(\Phi)\right)}.
	\end{aligned}\right.
\end{equation}
The forms of the third and fourth energy equations of \eqref{eq:model-problem-reformulation-MESAV2} are different from that in \eqref{eq:dissipative-system-ESAV}. 
Then we can construct BDF$k$ numerical schemes for above system. 
And we can choose $\theta_0^{n+1}$ according to Theorem \ref{Th:MESAV1-stability} if we set $0 \leq \gamma \leq \frac{\left(\mathcal{G}\mu\left(B_k(\phi^{n})\right), \mu\left(B_k(\phi^{n})\right)\right)}{\left(\mathcal{G} \mu^{n+1}, \mu^{n+1}\right)}$. 
\end{remark}

 \subsection{The fully decoupled R-ESAV-2 approach for Navier-Stokes equation} 
In this subsection, we consider the fully decoupled R-ESAV-2 approach for the following incompressible Navier-Stokes equation, which is a classic dissipative system
\begin{equation}\label{eq:Navier-Stokes}
\left\{\begin{aligned}
& \frac{\partial \mathbf{u}}{\partial t}-\nu \Delta \mathbf{u}+(\mathbf{u} \cdot \nabla) \mathbf{u}+\nabla p =\mathbf{f} & & \text { in } \Omega \times \mathcal{T} ,\\
& \nabla \cdot \mathbf{u} =0 & & \text { in } \Omega \times \mathcal{T}, \\ 
& \mathbf{u} =\mathbf{0} & & \text { on } \partial \Omega \times \mathcal{T}, 
\end{aligned}\right.
\end{equation}
where $\Omega$ is an open bounded domain in $\mathbb{R}^{d} (d=2,3)$ with a sufficiently smooth boundary $\partial \Omega, \mathcal{T}=(0, T]$, $\mathbf{u}, p$ are the unknown velocity and pressure respectively, $\mathbf{f}$ represents an external body force, $\nu>0$ is the viscosity coefficient and $\mathbf{n}$ is the unit outward normal of the domain $\Omega$. 
The system \eqref{eq:Navier-Stokes} satisfies the following law
\begin{equation}\label{eq:Navier-Stokes-energy-law}
	\frac{\mathrm{d}}{\mathrm{d} t} E(\mathbf{u})=-\nu\|\nabla\mathbf{u}\|^{2} + \left(\mathbf{f}, \mathbf{u}\right),
\end{equation}     
where $E(\mathbf{u})=\frac{1}{2}\|\mathbf{u}\|^{2}$ is the total energy. 

Then we can consider the R-ESAV-2 approach for Navier-Stokes equation \eqref{eq:Navier-Stokes}. 

Introduce an exponential SAV $R(t)=\exp\left(E\left(\mathbf{u}\right)\right)$, then we can rewrite the governing system \eqref{eq:Navier-Stokes} into the equivalent form as follows 
\begin{equation}\label{eq:Navier-Stokes-ESAV-2}
\left\{\begin{aligned}
& \frac{\partial \mathbf{u}}{\partial t}+V(\xi)(\mathbf{u} \cdot \nabla) \mathbf{u}-\nu \Delta \mathbf{u}+\nabla p =\mathbf{f},\\
& \nabla \cdot \mathbf{u} =0, \\ 
& \frac{\mathrm{d} R(t)}{\mathrm{d} t}=-\nu R(t) \|\nabla \mathbf{u}\|^{2} + R(t)(\mathbf{f}, \mathbf{u}), \\
& \xi = \frac{R(t)}{\exp\left(E(\mathbf{u})\right)}.
\end{aligned}\right.
\end{equation}

Next, we construct two BDF$k$ schemes for \eqref{eq:Navier-Stokes-ESAV-2}. 
First one is based on pressure correction approach. 

\textbf{Scheme \uppercase\expandafter{\romannumeral1}: }
Given $\mathbf{u}^{n-k}, \cdots, \mathbf{u}^{n},  p^{n-k}, \cdots, p^{n}, r^{n-k}, \cdots r^{n}$, we solve $\mathbf{u}^{n+1}, p^{n+1}, r^{n+1}$ via four steps as follows:

\textbf{Step 1:} Determine solution $\tilde R^{n+1}$ and compute $\xi^{n+1}$:
\begin{eqnarray}
\label{eq:NS-R-ESAV2-BDFk-1}
 && \frac{\tilde{R}^{n+1}-R^{n}}{\delta t} =-\nu \tilde{R}^{n+1} \|\nabla B_{k}\left(\mathbf{u}^{n}\right)\|^{2} + \tilde{R}^{n+1}(\mathbf{f}^{n+1}, B_{k}\left(\mathbf{u}^{n}\right)), \\
 \label{eq:NS-R-ESAV2-BDFk-2}
 && \xi^{n+1}=\frac{\tilde{R}^{n+1}}{\exp\left(E\left(B_{k}\left(\mathbf{u}^{n}\right)\right)\right)}. 	
\end{eqnarray}

\textbf{Step 2:} Compute an intermediate solution $\tilde{\mathbf{u}}^{n+1}$:
\begin{eqnarray}
\label{eq:NS-R-ESAV2-BDFk-3}
& & \frac{\alpha_{k}\tilde{\mathbf{u}}^{n+1}-A_{k}\left(\mathbf{u}^{n}\right)}{\delta t}+V(\xi^{n+1})(B_{k}\left(\mathbf{u}^{n}\right) \cdot \nabla) B_{k}\left(\mathbf{u}^{n}\right)-\nu \Delta \tilde{\mathbf{u}}^{n+1}+\nabla \hat B_{k}\left(p^{n}\right) =\mathbf{f}^{n+1}, \\
\label{eq:NS-R-ESAV2-BDFk-4}
& & \tilde{\mathbf{u}}^{n+1}|_{\partial \Omega} = 0. 
\end{eqnarray}

\textbf{Step 3:} Solve solution $\left(\mathbf{u}^{n+1}, p^{n+1}\right)$:
\begin{eqnarray}
\label{eq:NS-R-ESAV2-BDFk-5}
& & \frac{\alpha_{k} \mathbf{u}^{n+1}- \alpha_{k} \tilde{\mathbf{u}}^{n+1}}{\delta t}+\nabla \left(p^{n+1} - \hat B_{k}\left(p^{n}\right) \right) =0, \\
\label{eq:NS-R-ESAV2-BDFk-6}
& & \nabla \cdot \mathbf{u}^{n+1} =0, \\ 
\label{eq:NS-R-ESAV2-BDFk-7}
& & \mathbf{u}^{n+1}\cdot \mathbf{n}|_{\partial \Omega} = 0, 
\end{eqnarray}
where operator $\hat B_{1}=B_{1}$ for BDF$1$ scheme and $\hat B_{k}=B_{k-1}$ for BDF$k$ ($k \geq 2$) scheme.

\textbf{Step 4:} Update the SAV $R^{n+1}$ via relaxation factor as follows:
\begin{equation}\label{eq:NS-R-ESAV2-BDFk-8}
	R^{n+1} = \theta_{0}^{n+1} \tilde{R}^{n+1} + (1-\theta_{0}^{n+1})\exp\left(E(\mathbf{u}^{n+1})\right), \quad \theta_{0}^{n+1} \in \mathcal{W},
\end{equation}
where, $\mathcal{W}$ is a set defined as follows: 
\begin{equation}\label{eq:NS-set-condition-ESAV2-BDFk}
	\mathcal{W}=\left\lbrace \theta\in [0,1] \; s.t. \;   \frac{R^{n+1}-\tilde{R}^{n+1}}{\delta t} = -\gamma^{n+1}\nu \tilde{R}^{n+1} \|\nabla \mathbf{u}^{n+1}\|^{2} + \nu \tilde{R}^{n+1}\|\nabla B_{k}\left(\mathbf{u}^{n}\right)\|^{2}\right\rbrace,
\end{equation}
with  $\gamma^{n+1} \geq 0$ to be determined such that $\mathcal{W}$ is not empty.

Inspired by projection method in \cite{wu2022new, huang2021stability}, we can construct the following fully decoupled R-ESAV-2 scheme: 

\textbf{Scheme \uppercase\expandafter{\romannumeral2}: }
Given $\mathbf{u}^{n-k}, \cdots, \mathbf{u}^{n},  p^{n-k}, \cdots, p^{n}, r^{n-k}, \cdots r^{n}$, we determine $\mathbf{u}^{n+1}, p^{n+1}, r^{n+1}$ via two steps as follows:

\textbf{Step 1:} Solve solution $\left(\mathbf{u}^{n+1}, p^{n+1}, \tilde R^{n+1}\right)$:
\begin{eqnarray}
 \label{eq:NS-R-ESAV2-new-BDFk-1}
& & \frac{\alpha_{k}\mathbf{u}^{n+1}-A_{k}\left(\mathbf{u}^{n}\right)}{\delta t}-\nu \Delta\mathbf{u}^{n+1}+V(\xi^{n+1})(B_{k}\left(\mathbf{u}^{n}\right) \cdot \nabla) B_{k}\left(\mathbf{u}^{n}\right)+\nabla B_{k}\left(p^{n}\right) =\mathbf{f}^{n+1}, \\
 \label{eq:NS-R-ESAV2-new-BDFk-2}
& & \frac{\tilde{R}^{n+1}-R^{n}}{\delta t} =-\nu \tilde{R}^{n+1} \|\nabla B_{k}\left(\mathbf{u}^{n}\right)\|^{2} + \tilde{R}^{n+1}(\mathbf{f}^{n+1}, B_{k}\left(\mathbf{u}^{n}\right)), \\
 \label{eq:NS-R-ESAV2-new-BDFk-3}
& & \xi^{n+1}=\frac{\tilde{R}^{n+1}}{\exp\left(E\left(B_{k}\left(\mathbf{u}^{n}\right)\right)\right)}, \\
 \label{eq:NS-R-ESAV2-new-BDFk-4}
& & \left(\nabla p^{n+1}, \nabla q\right)=\left(\mathbf{f}^{n+1} - (\mathbf{u}^{n+1} \cdot \nabla) \mathbf{u}^{n+1}-\nu \nabla\times\nabla\times\mathbf{u}^{n+1}, \nabla q\right) \\
& & \qquad \qquad \qquad  =\left(\mathbf{f}^{n+1} - (\mathbf{u}^{n+1} \cdot \nabla) \mathbf{u}^{n+1}, \nabla q\right)-\nu \int_{\partial \Omega}\left(\left(\nabla\times\mathbf{u}^{n+1}\right)\times \nabla q\right) \cdot \boldsymbol{n} \mathrm{d} s, \notag
\end{eqnarray}
where $\mathbf{n}$ is the outward normal of $\partial \Omega$.

\textbf{Step 2:} Update the SAV $R^{n+1}$ via relaxation factor as follows: 
\begin{equation}
 \label{eq:NS-R-ESAV2-new-BDFk-5}
	R^{n+1} = \theta_{0}^{n+1} \tilde{R}^{n+1} + (1-\theta_{0}^{n+1})\exp\left(E(\mathbf{u}^{n+1})\right), \quad \theta_{0}^{n+1} \in \mathcal{W},
\end{equation}
where, $\mathcal{W}$ is a set defined as follows: 
\begin{equation}
 \label{eq:NS-R-ESAV2-new-BDFk-6}
	\mathcal{W}=\left\lbrace \theta\in [0,1] \; s.t. \;   \frac{R^{n+1}-\tilde{R}^{n+1}}{\delta t} = -\gamma^{n+1} \nu \tilde{R}^{n+1} \|\nabla \mathbf{u}^{n+1}\|^{2} + \nu \tilde{R}^{n+1}\|\nabla B_{k}\left(\mathbf{u}^{n}\right)\|^{2}\right\rbrace,
\end{equation}
and  $\gamma^{n+1} \geq 0$ is to be determined such that  the set $\mathcal{W}$ is not empty. 

\begin{remark}
	In the case of periodic boundary condition, the operators $\nabla , \nabla \cdot$ and $\Delta^{-1}$ can commute with each other by defining them in the Fourier space. 
	In the absence of $\mathbf{f}$, taking the divergence on both sides of first equation of \eqref{eq:Navier-Stokes}, we obtain
	\begin{equation}
	-\Delta p=\nabla \cdot(\mathbf{u} \cdot \nabla \mathbf{u}),
	\end{equation} 
	Then the first equation of \eqref{eq:Navier-Stokes} can be rewritten as 
	\begin{equation}
	\frac{\partial \boldsymbol{u}}{\partial t}-\nu \Delta \boldsymbol{u}-\mathbf{J}(\boldsymbol{u} \cdot \nabla \boldsymbol{u})=\mathbf{0},
	\end{equation}
	where $\mathbf{J}$ is defined by
	\begin{equation}
	\mathbf{J} \mathbf{v}:=\nabla \times \nabla \times \Delta^{-1} \mathbf{v} \quad \forall \mathbf{v} \in \mathbf{L}_{0}^{2}(\Omega).
	\end{equation}
	 Moreover, in addition to \eqref{eq:Navier-Stokes-energy-law}, the system \eqref{eq:Navier-Stokes}  satisfies  energy dissipation law as follows
	\begin{equation}
	\frac{1}{2} \frac{d}{d t}\|\nabla \mathbf{u}\|^{2}=-\nu\|\Delta \mathbf{u}\|^{2}.
	\end{equation}
	Thus, \eqref{eq:NS-R-ESAV2-new-BDFk-1}-\eqref{eq:NS-R-ESAV2-new-BDFk-4} can be replaced by
	\begin{eqnarray}
 \label{eq:NS-R-ESAV2-new-BDFk-periodic-1}
& & \frac{\alpha_{k}\mathbf{u}^{n+1}-A_{k}\left(\mathbf{u}^{n}\right)}{\delta t}-\nu \Delta\mathbf{u}^{n+1}-V(\xi^{n+1})\mathbf{J}\left(B_{k}\left(\mathbf{u}^{n}\right) \cdot \nabla B_{k}\left(\mathbf{u}^{n}\right)\right)=\mathbf{0}, \\
 \label{eq:NS-R-ESAV2-new-BDFk-periodic-2}
& & \frac{\tilde{R}^{n+1}-R^{n}}{\delta t} =-\nu \tilde{R}^{n+1} \|\Delta B_{k}\left(\mathbf{u}^{n}\right)\|^{2}, \\
 \label{eq:NS-R-ESAV2-new-BDFk-periodic-3}
& & \xi^{n+1}=\frac{\tilde{R}^{n+1}}{\exp\left(E\left(B_{k}\left(\mathbf{u}^{n}\right)\right)\right)}, \\
 \label{eq:NS-R-ESAV2-new-BDFk-periodic-4}
& & \Delta p^{n+1}=-\nabla \cdot\left(\mathbf{u}^{n+1} \cdot \nabla \mathbf{u}^{n+1}\right).
\end{eqnarray}
\end{remark}

 Setting $K(\mathbf{u})=\nu\|\nabla \mathbf{u}\|^{2}$ or $K(\mathbf{u})=\nu\|\Delta \mathbf{u}\|^{2}$,  we can choose $\theta_{0}^{n+1}$ and $\gamma^{n+1}$ in  Scheme \uppercase\expandafter{\romannumeral1} and Scheme \uppercase\expandafter{\romannumeral2} accordding to Theorem \ref{Th:ESAV2-stability}. 
 Similarly, above schemes satiesfy energy dissipation law \eqref{eq:ESAV2-stability} and \eqref{eq:ESAV2-stability-1}.

%
%
%

 \section{Numerical simulations}
 \label{sec:numerical-simulations}
In this section, we demonstrate ample numerical results to verify that the constructed  R-ESAV-1 and R-ESAV-2 approaches are accurate and 
efficient. Besides, we also give detailed comparisons between the original ESAV schemes with the constructed R-ESAV schemes. We consider the numerical examples with periodic boundary condition and use the Fourier spectral method for spatial discretization in what follows unless explicitly given, and the dissipation rate parameter $\gamma$ is set to $1$ for the R-ESAV-1 schemes by default.

\textbf{Example 1.} We first consider the Allen-Cahn equation 
\begin{equation}\label{eq:Allen-Cahn}
\frac{\partial \phi}{\partial t}=\sigma_0 \Delta \phi+\left(1-\phi^{2}\right) \phi.
\end{equation}

(\romannumeral1) We give an exact solution
\begin{equation} \label{eq:AC-CH-exact-solution-example}
\phi(x, y, t)=\exp (\sin (\pi x) \sin (\pi y)) \sin (t),
\end{equation}
and $f$ is the external force satiesfying \eqref{eq:Allen-Cahn}. 
We set computational domain $\Omega=[0, 2]^2$ and the model parameter $\sigma_0=0.01^2$. For spatial discretization, we use Fourier mode $N^2=64^2$, so that compared with the time discretization error, the spatial discretization error is negligible.

The convergence rates of the $L^2$ error at $T = 1$ obtained by various schemes are presented in Fig.\,\ref{Fig:AC-order-test}, where we can observe that 

(a) Numerical results are all consistent with the expected convergence rates; 

(b) The errors of of  R-ESAV-1 (resp. R-ESAV-2) schemes for BDF$1$ scheme are  obviously smaller than that of ESAV-1 (resp. ESAV-2) schemes; 

(c) The improvement in the accuracy for the ESAV schemes with relaxation for higher-order schemes is not as notable as for first-order scheme. 

We also demonstrate  the evolution of relaxation factor $\theta_0^{n+1}$ obtained by  R-ESAV-1/BDF$2$ and R-ESAV-2/BDF$2$ scheme with time step $\delta t=1e-3$ in Fig.\,\ref{Fig:AC-exact-solution-zeta}. It can be obviously show that $\theta_0^{n+1}$ always takes the value zeros except at an initial time interval for R-ESAV-2/BDF$2$ scheme.


\begin{figure}[htbp]
	\centering
	\includegraphics[width=5.3cm]{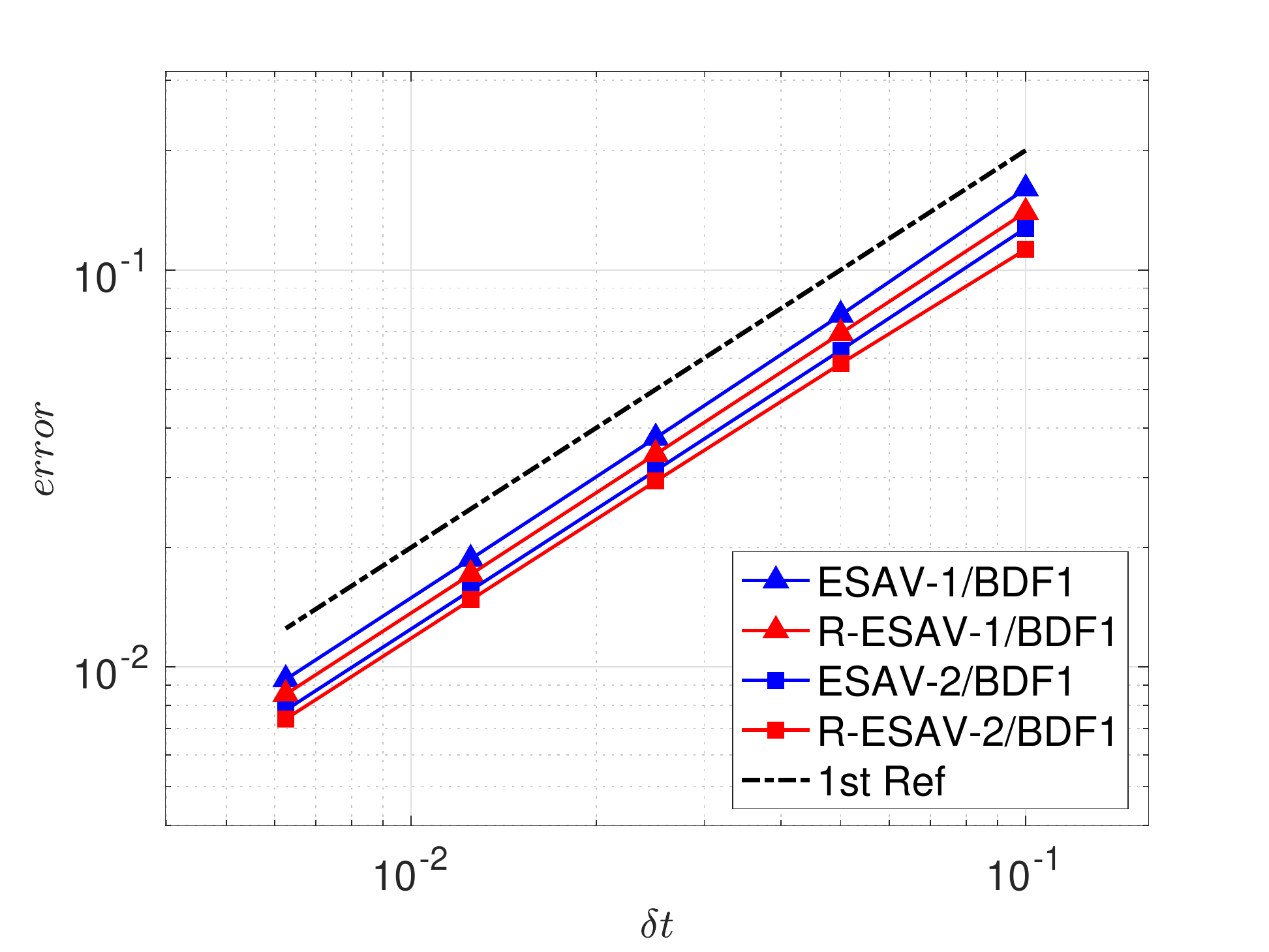}\hspace{-6mm}
	\includegraphics[width=5.3cm]{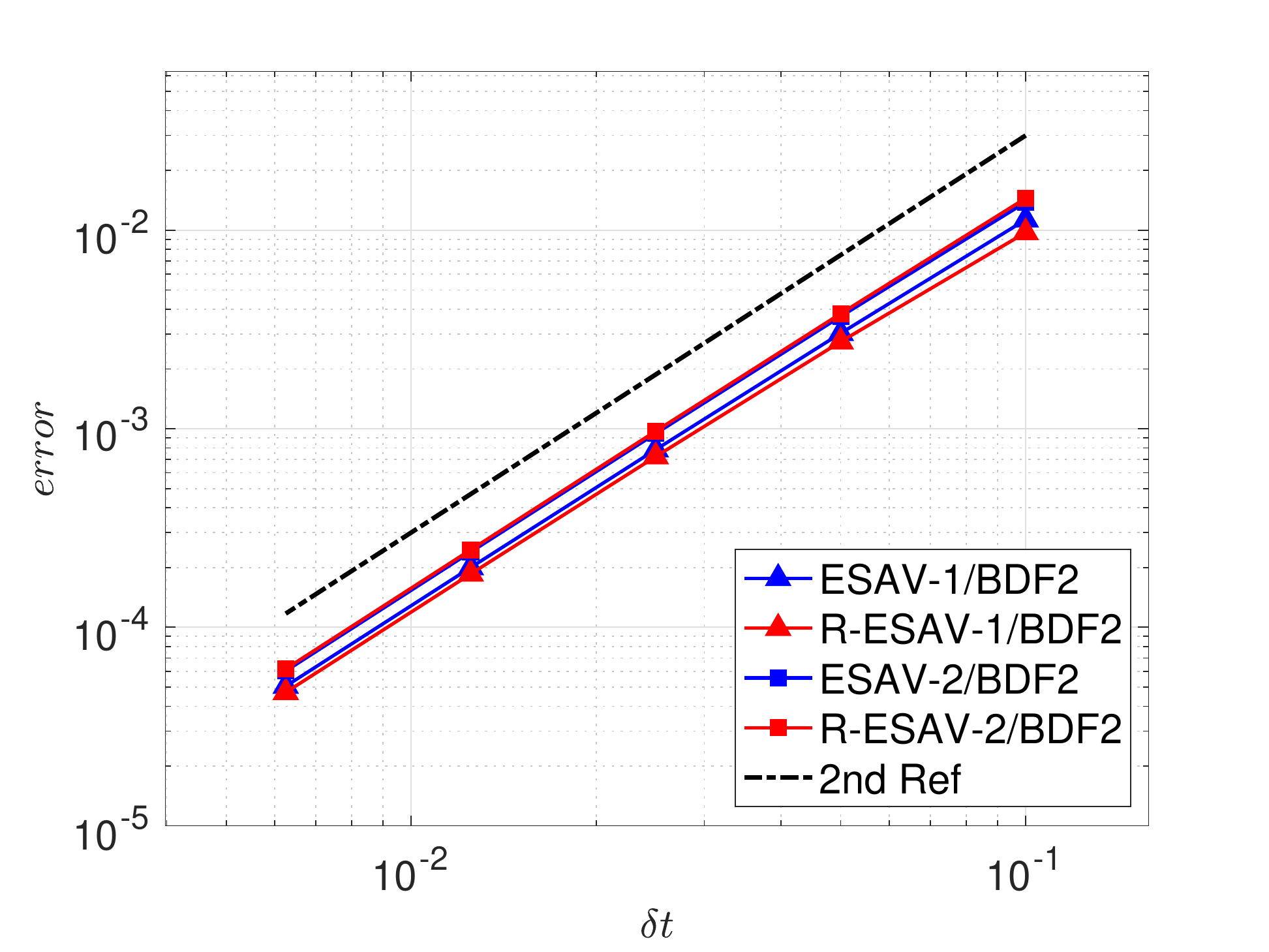}\hspace{-6mm}
	\includegraphics[width=5.3cm]{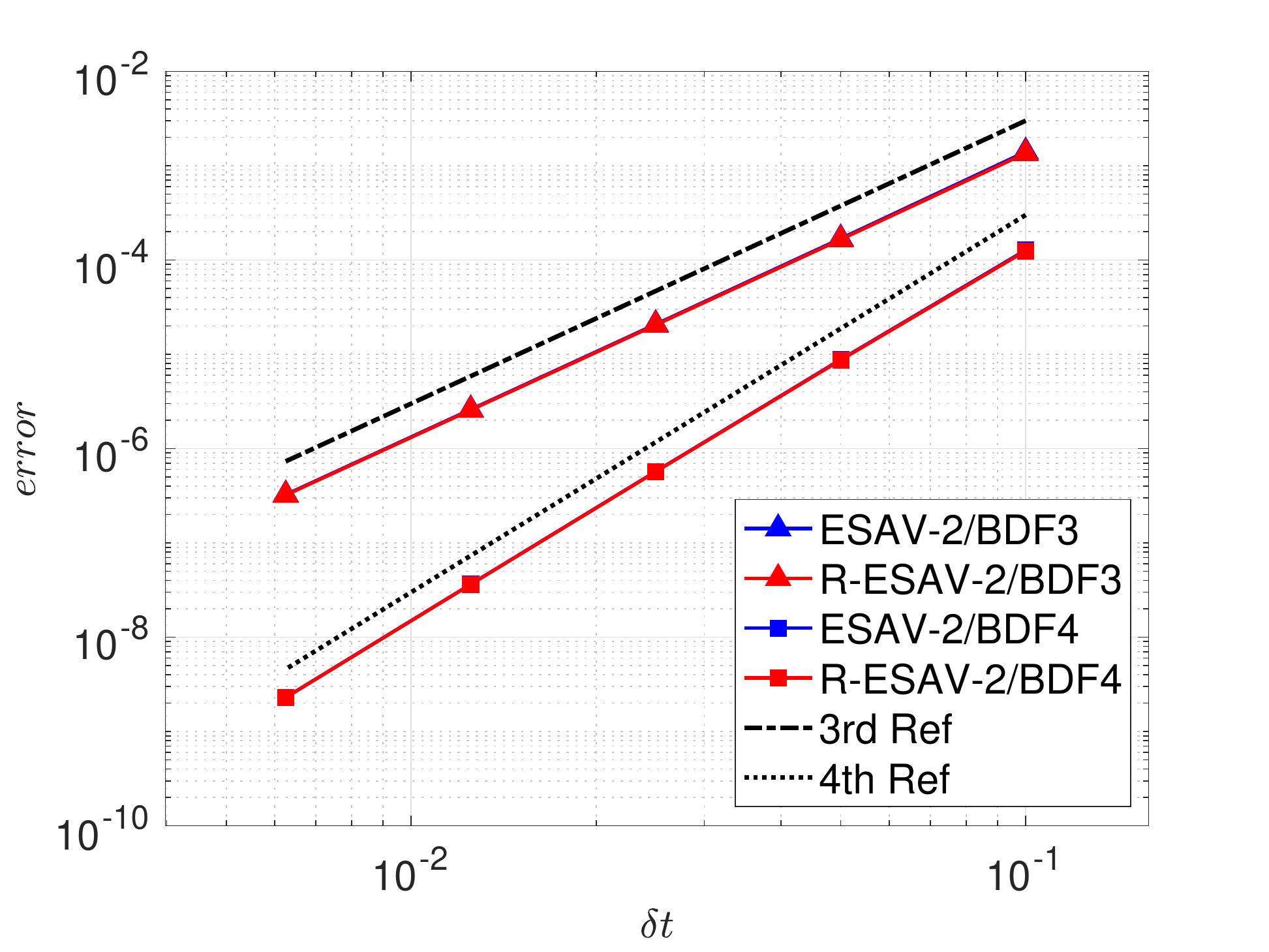}
\caption{Example 1 (\romannumeral1). Convergence test for Allen-Cahn equation obtained by ESAV-1/BDF$k$, R-ESAV-1/BDF$k$ ($k=1, 2$), ESAV-2/BDF$k$ and R-ESAV-2/BDF$k$ ($k=1, 2, 3, 4$) schemes.}
	\label{Fig:AC-order-test}
\end{figure} 

\begin{figure}[htbp]
	\centering
	\includegraphics[width=5.3cm]{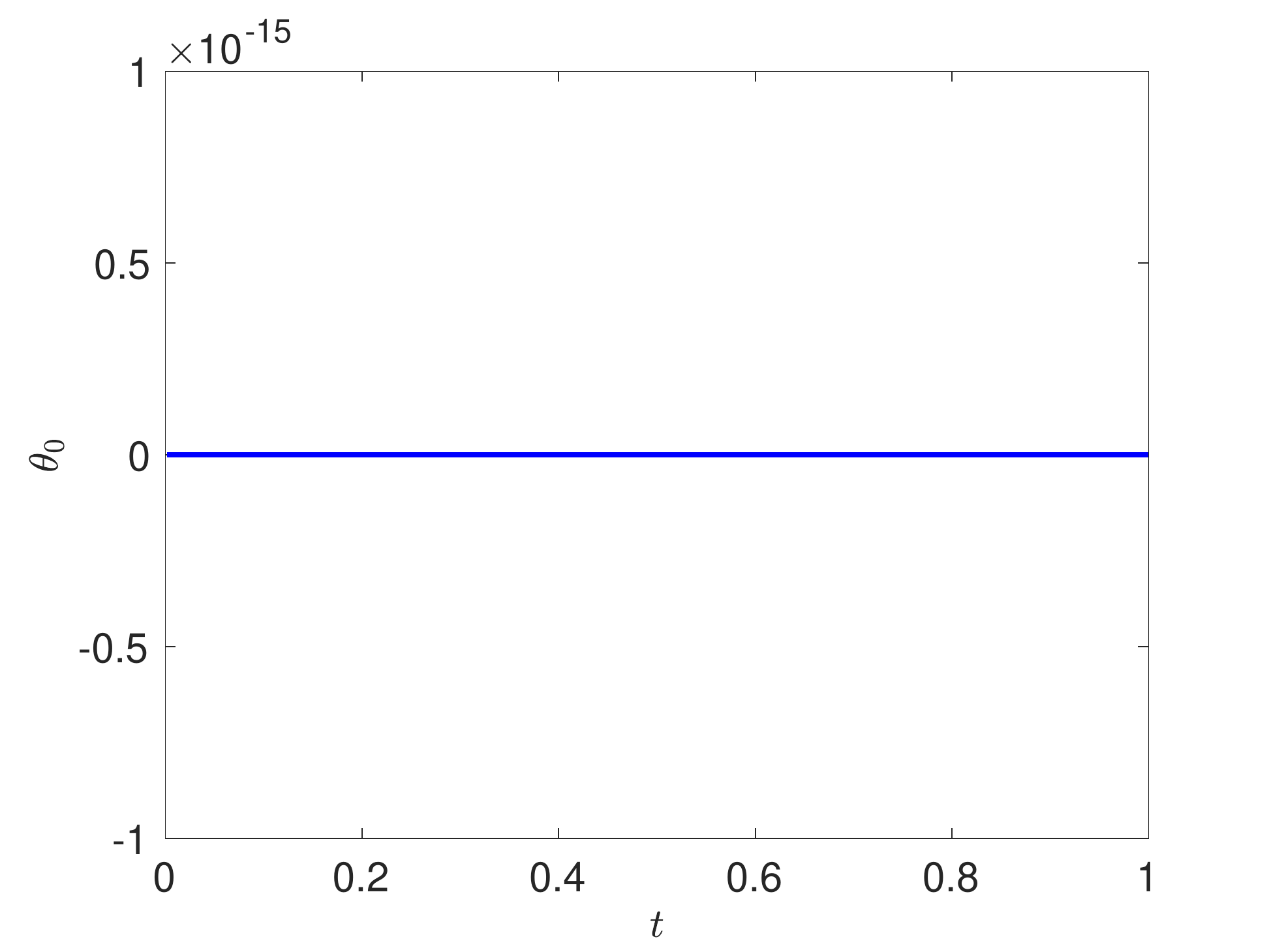}
	\includegraphics[width=5.3cm]{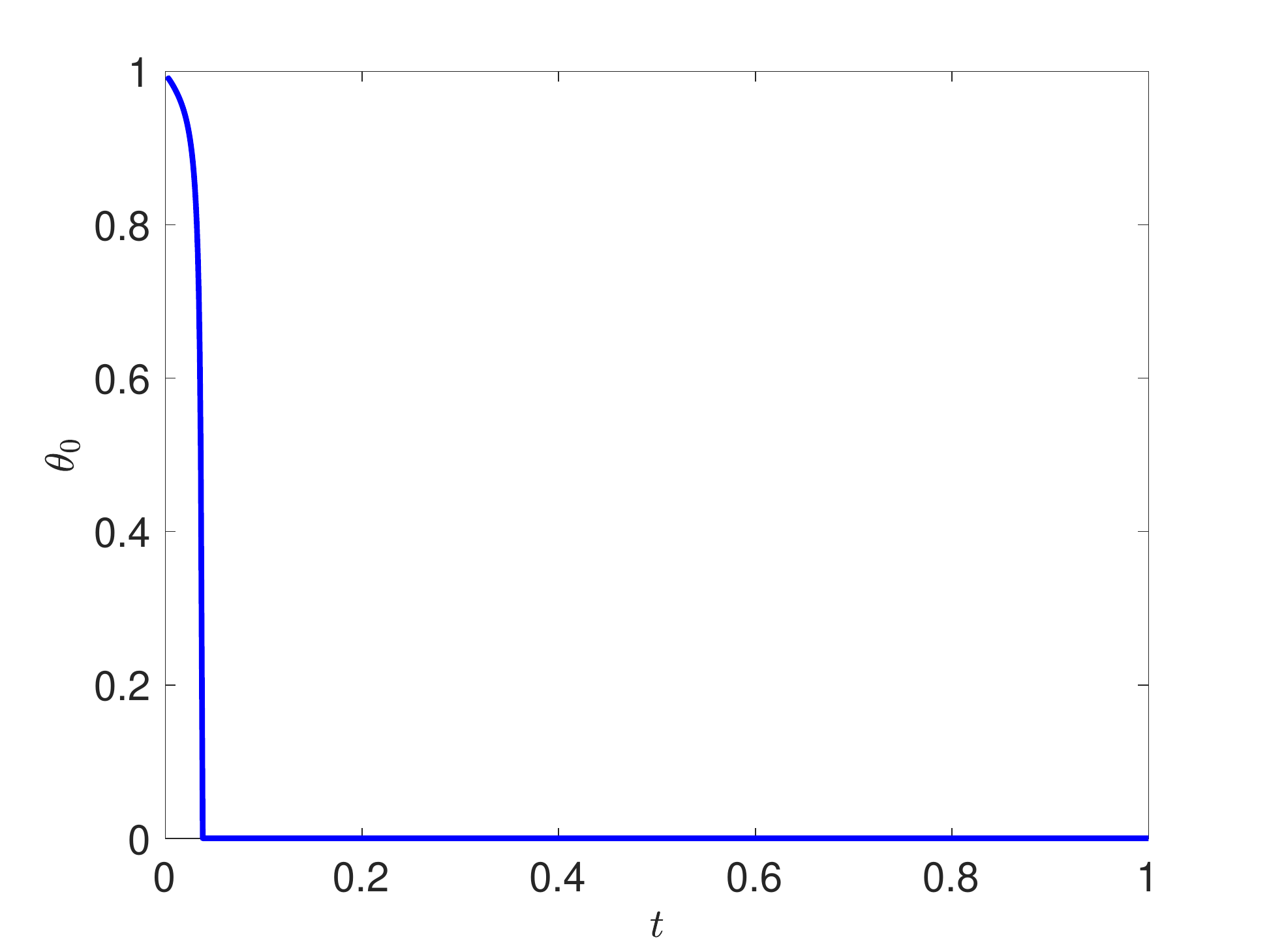}\hspace{-6mm}
	\caption{Example 1 (\romannumeral1). The evolution of relaxation factor $\theta_{0}^{n+1}$ with $\delta t=1e-3$. Left: R-ESAV-1/BDF$2$ scheme; right: R-ESAV-2/BDF$2$ scheme.}
	\label{Fig:AC-exact-solution-zeta}
\end{figure}


(\romannumeral2)  We choose the initial condition as            
\begin{equation}\label{eq:AC-CH-initial-condition-star-shape}
\begin{aligned}
	& \phi(x, y)=\tanh \frac{1.5+1.2 \cos (6 \lambda)-2 \pi \rho}{\sqrt{2\alpha}}, \\ 
	& \lambda=\arctan \frac{y-0.5}{x-0.5}, \quad \rho=\sqrt{\left(x-\frac{1}{2}\right)^{2}+\left(y-\frac{1}{2}\right)^{2}},
\end{aligned}
\end{equation}
where $(\lambda, \rho)$ are the polar coordinates of $(x, y)$. 
We set computational domain as $\Omega=[0, 1]^{2}$, the other parameters are $\sigma_0=0.01^2$, and Fourier modes are $N^2=128^2$. 
The computational solution of the semi-implicit/BDF2 scheme obtained by time step $ \delta t = 1e-5$  is regarded as the reference solution. 
It represents the $L^{2}$-norm error of four numerical schemes  we constructed above at $T=200$ with various time steps in Table \ref{table:comparison-dt-schemes}. 
It can be observed that, compared with ESAV-1 (resp. ESAV-2) schemes, R-ESAV-1 (resp. R-ESAV-2) schemes can noticeably reduce the error of the solution. 
The error of solution is large when the time step of ESAV-2 scheme is not sufficiently small, while R-ESAV-2 scheme can  improve accuracy obviously. 
It also shows a comparison of energy (first), errors of energy (second) and the evolution of errors of energy at different time steps (third) for the constructed schemes in Fig.\,\ref{Fig:AC-star-shape-energy}. 
Moreover the evolution of error of $\xi^{n+1}$ is presented in Fig.\,\ref{Fig:AC-star-shape-xi},  which indicates that the R-ESAV-1 (resp. R-ESAV-2) scheme can improve the accuracy compared with ESAV-1 (resp. ESAV-2) scheme and the error of  $\xi^{n+1}$ for R-ESAV-2 scheme will reach the machine accuracy after the simulation reaching the steady state. 

\linespread{1.2}
\begin{table}[htbp] 
	\centering
	\caption{Example 1 (\romannumeral2). A comparison of  $L^2$-error obtained by four approaches based on BDF$2$ scheme for Allen-Cahn equation at $T=200$ with various time steps.}
	 \label{table:comparison-dt-schemes}
	\begin{tabular}{||c|c|c|c|c||}
		\hline
		 & ESAV-1 & R-ESAV-1 & ESAV-2 & R-ESAV-2 \\
		\hline
	1E-1  & 9.49E-04 &  2.76E-04  & 0.98       & 1.01E-04  \\
	5E-2  & 8.43E-04 &  7.22E-05  & 1.07E-03   & 2.74E-05  \\
	1E-2  & 1.18E-04 &  2.94E-06  & 4.55E-05   & 1.01E-06  \\
	5E-3  & 3.53E-05 &  7.36E-07  & 2.29E-05   & 2.49E-07  \\
	1E-3  & 1.63E-06 &  2.96E-08  & 1.47E-06   & 9.88E-09  \\
	\hline
	\end{tabular}
\end{table}

\begin{figure}[htbp]
	\centering
	\includegraphics[width=5.3cm]{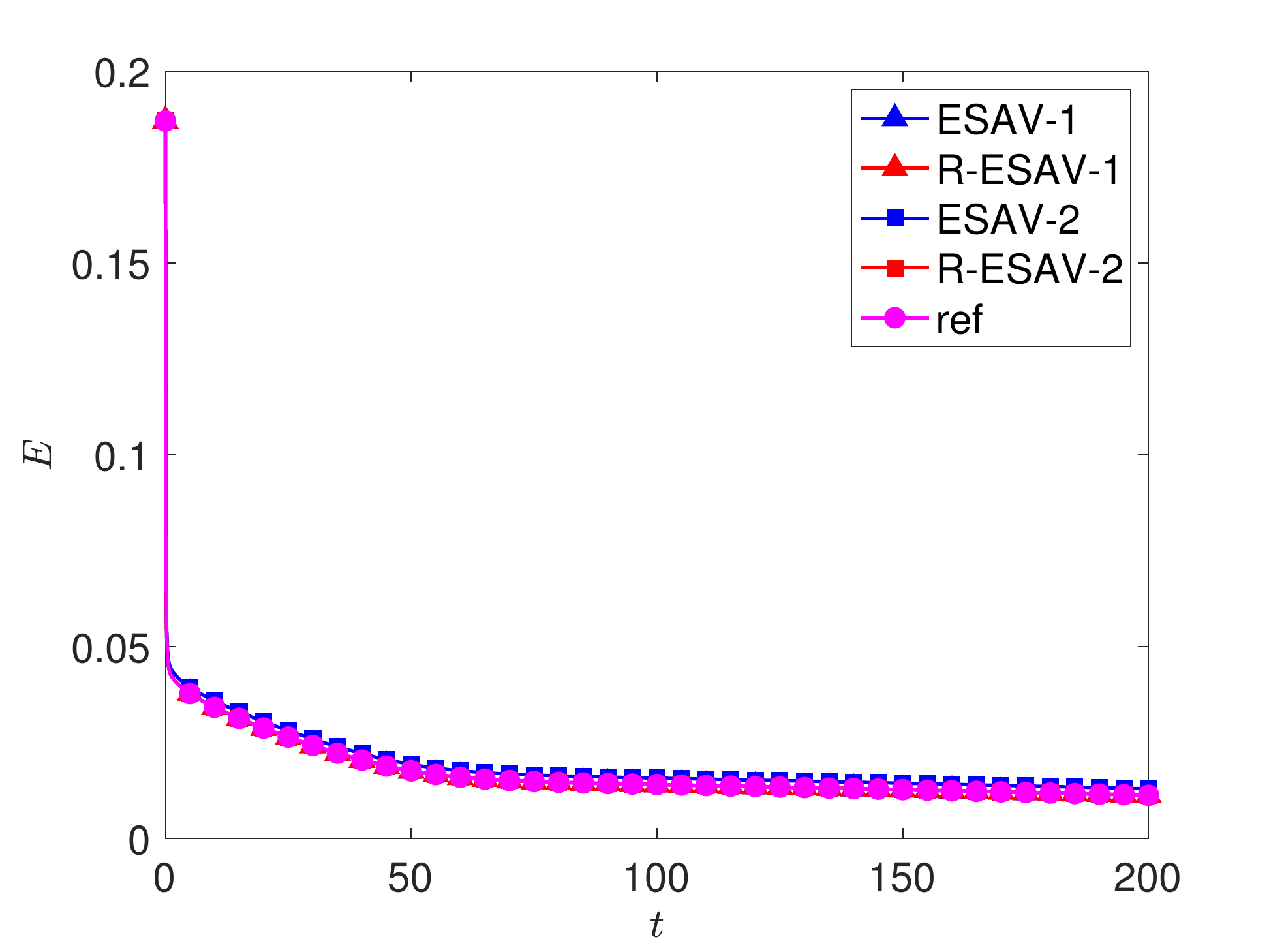}\hspace{-6mm} 
	\includegraphics[width=5.3cm]{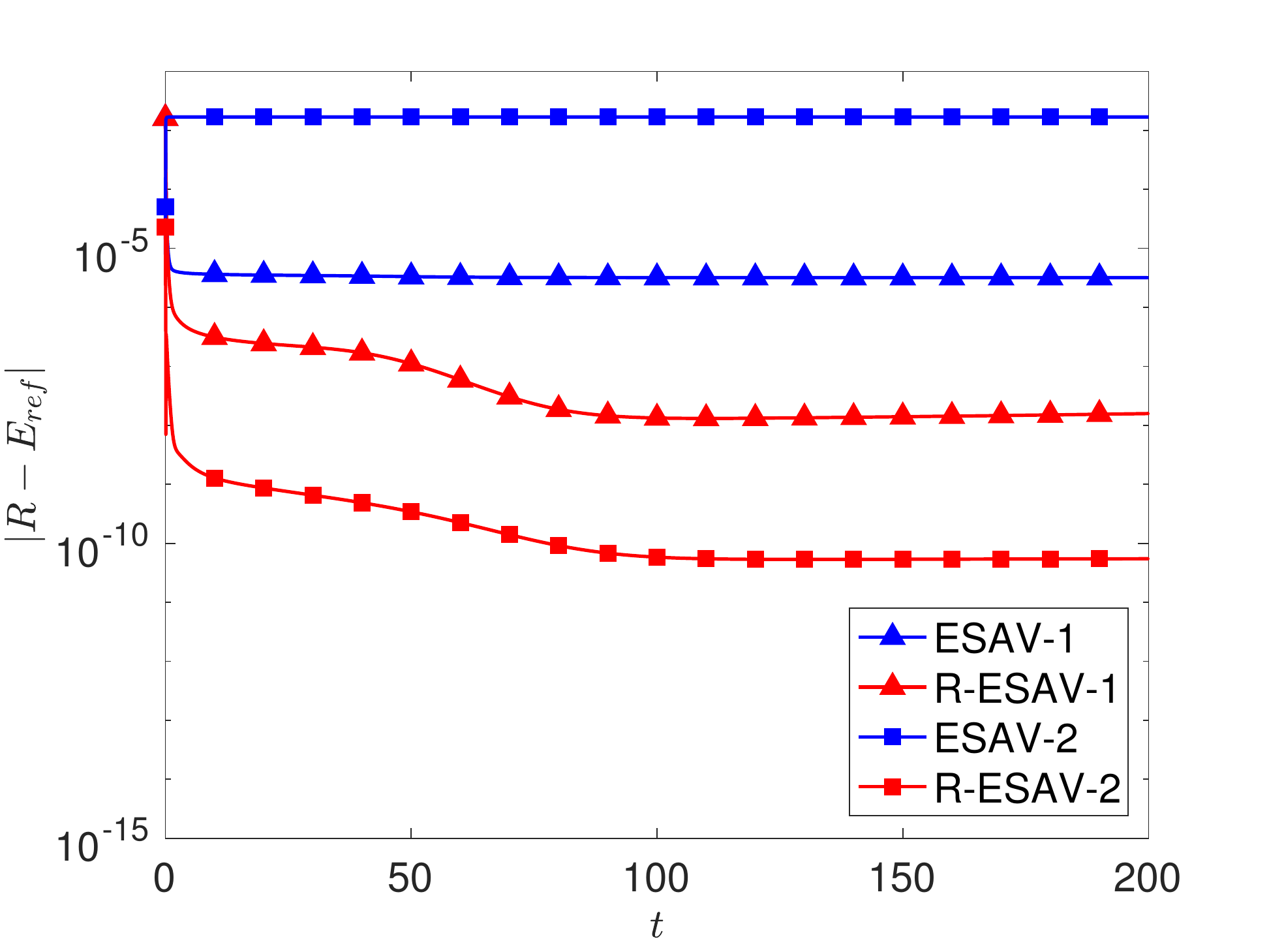} 
	\hspace{-6mm}
	\includegraphics[width=5.3cm]{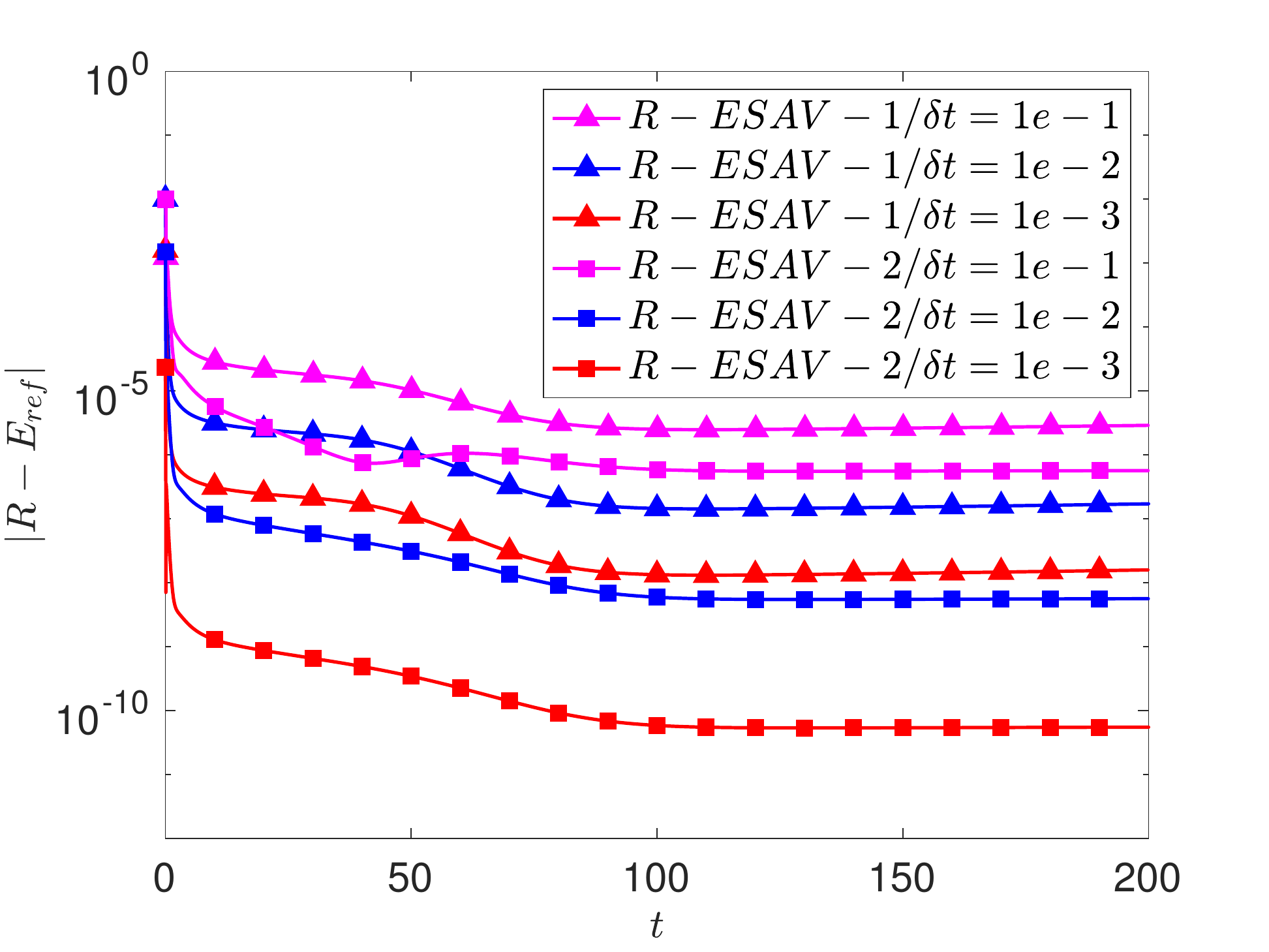}
	\hspace{-1cm}
	\caption{Example 1 (\romannumeral2). Allen-Cahn equation: a comparison of energy (first) and errors of energy (second) of four approaches based on BDF$2$ scheme; and a comparison of errors of energy of R-ESAV-1/BDF$2$  and R-ESAV-2/BDF$2$ schemes with various time steps (third).}
		\label{Fig:AC-star-shape-energy}
\end{figure} 

\begin{figure}[htbp]
	\centering
	\includegraphics[width=5.3cm]{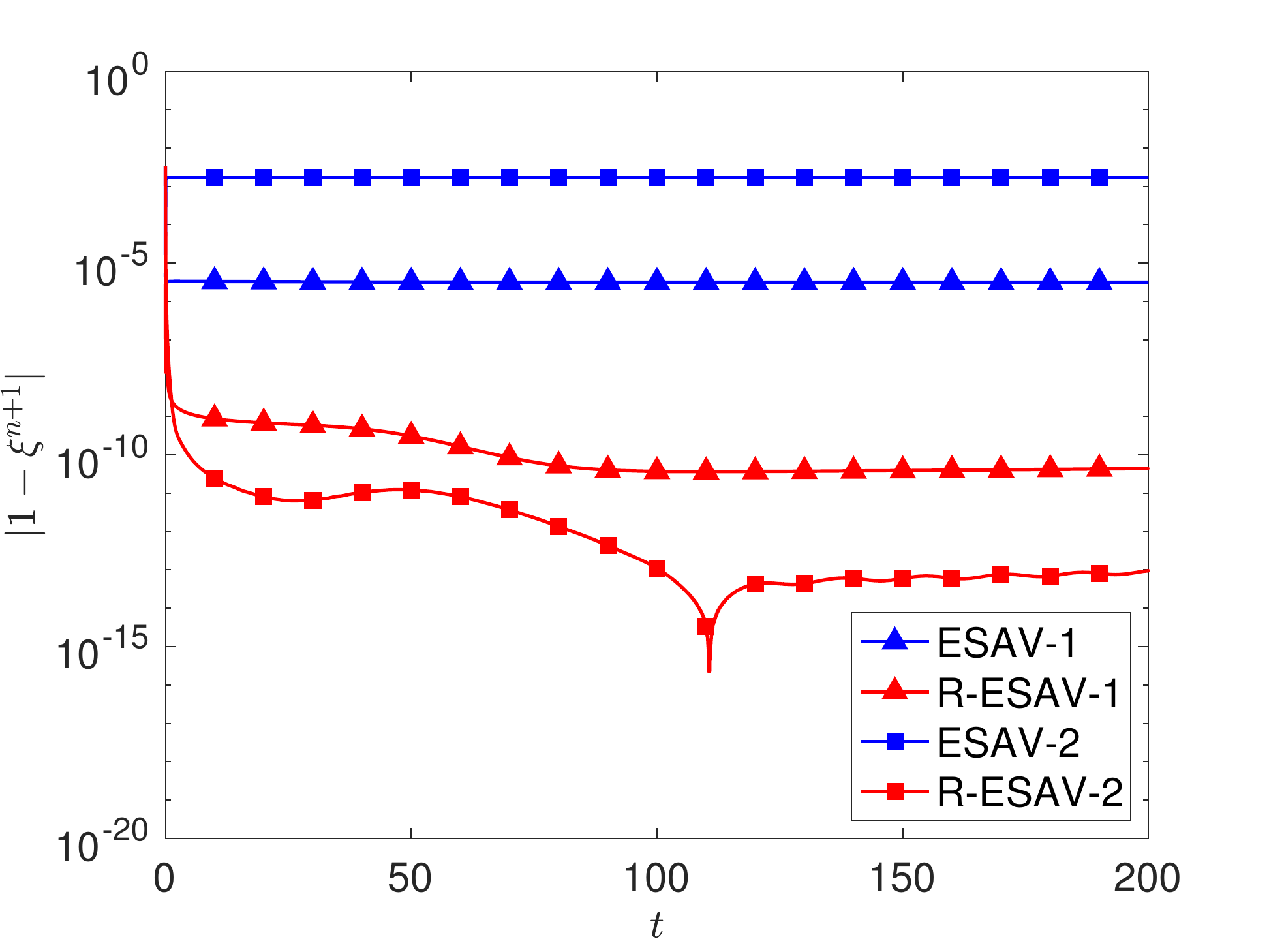}
	\caption{Example 1 (\romannumeral2). A comparison of the evolution of errors of $\xi^{n+1}$ obtained by four approaches based on BDF$2$ scheme for Allen-Cahn equation.}
		\label{Fig:AC-star-shape-xi}
\end{figure}

\textbf{Example 2.} The Cahn-Hilliard equation
\begin{equation}
\frac{\partial \phi}{\partial t}=-M \Delta\left(\sigma_0 \Delta \phi+\frac{1}{\epsilon^2}\left(1-\phi^{2}\right) \phi\right).
\end{equation}

(\romannumeral1)
We also choose \eqref{eq:AC-CH-exact-solution-example} as the exact solution, and
set model parameter to be $\sigma_0=0.04, M=0.005, \epsilon=1$. 
Fig.\,\ref{Fig:CH-order-test-three-method} shows the convergence rates of different schemes. 
We can observe similar results as those of the Allen-Cahn equation. 

\begin{figure}[htbp]
	\centering
	\includegraphics[width=5.3cm]{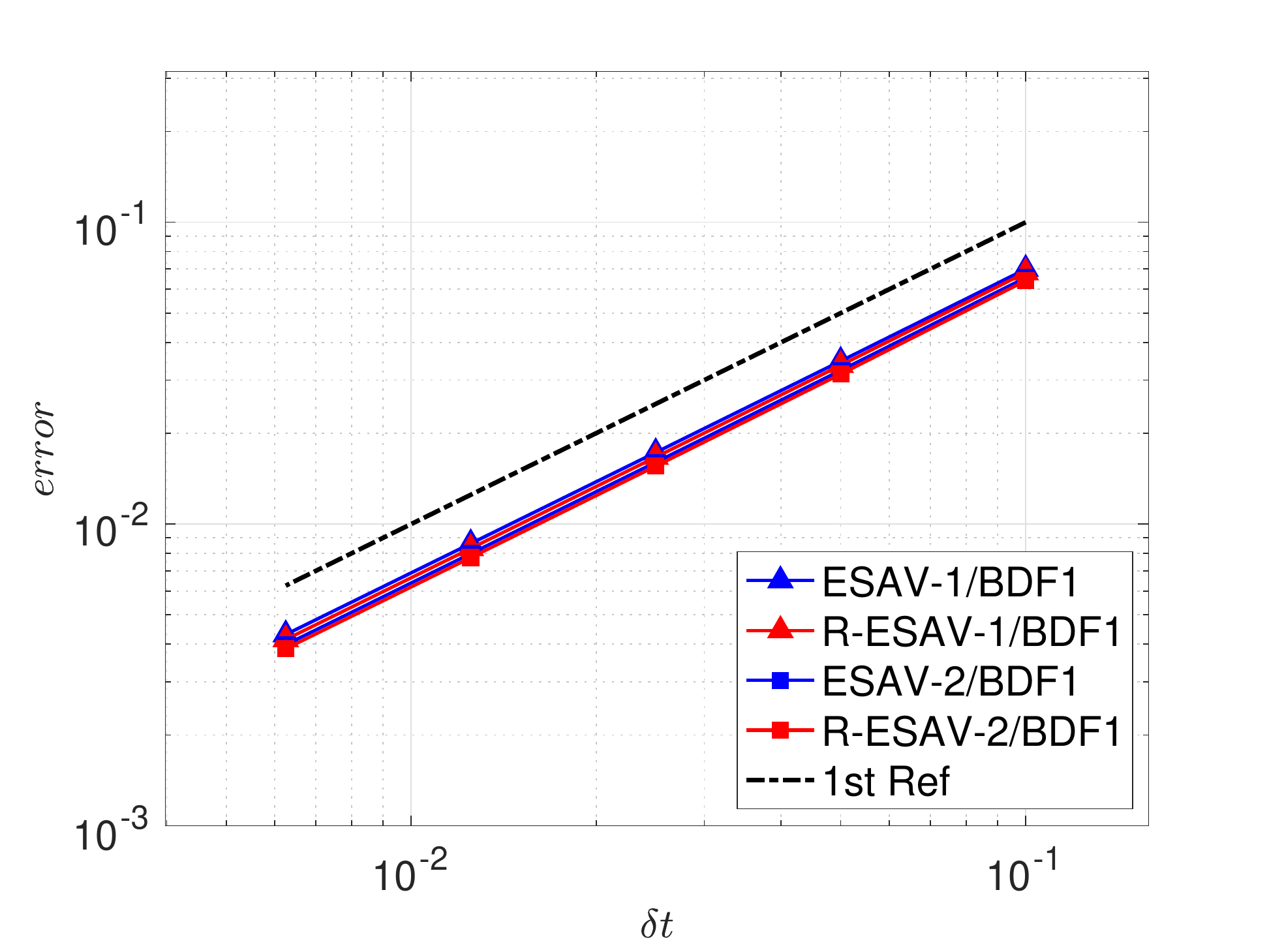}\hspace{-6mm}
	\includegraphics[width=5.3cm]{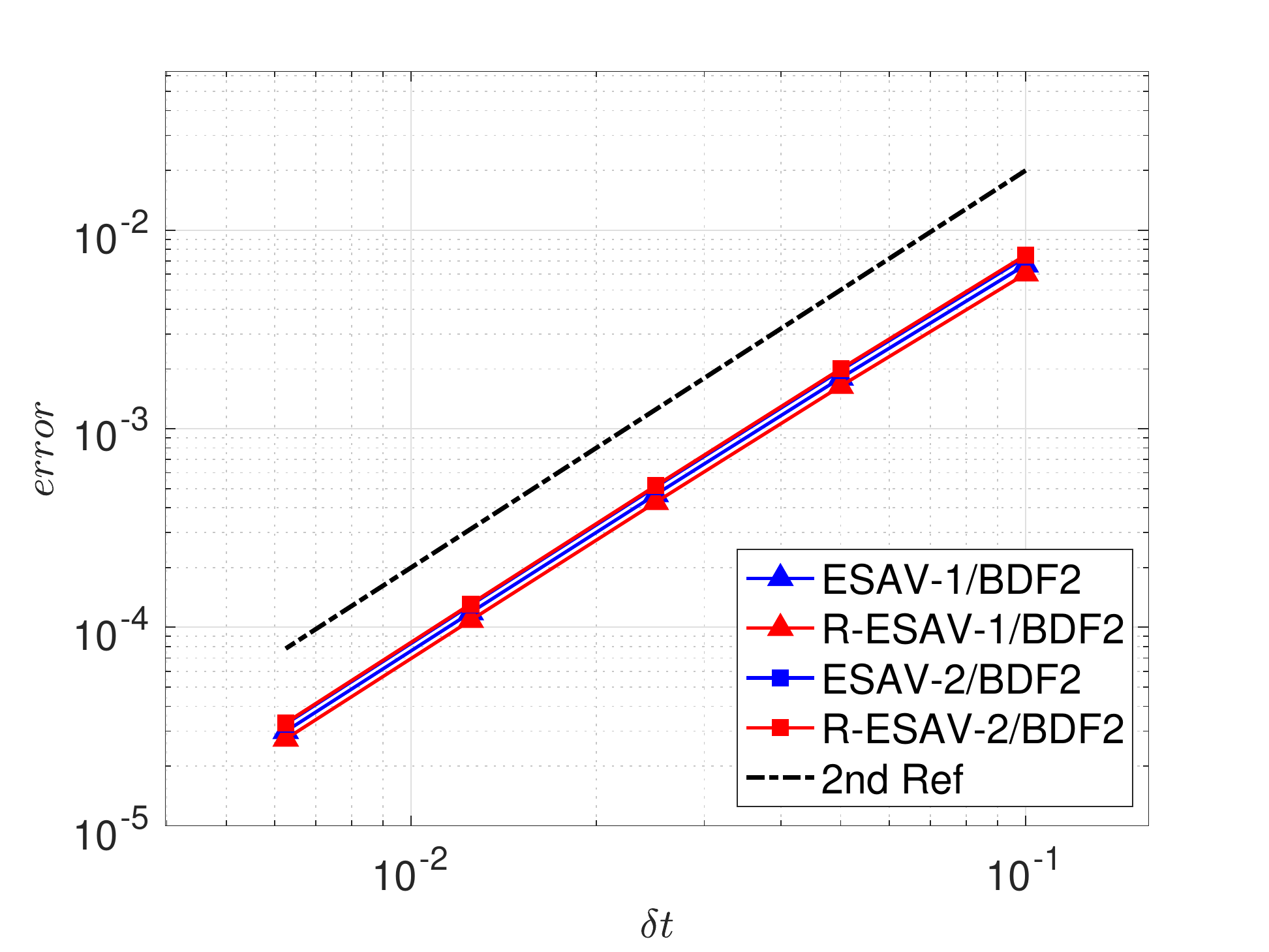}\hspace{-6mm}
	\includegraphics[width=5.3cm]{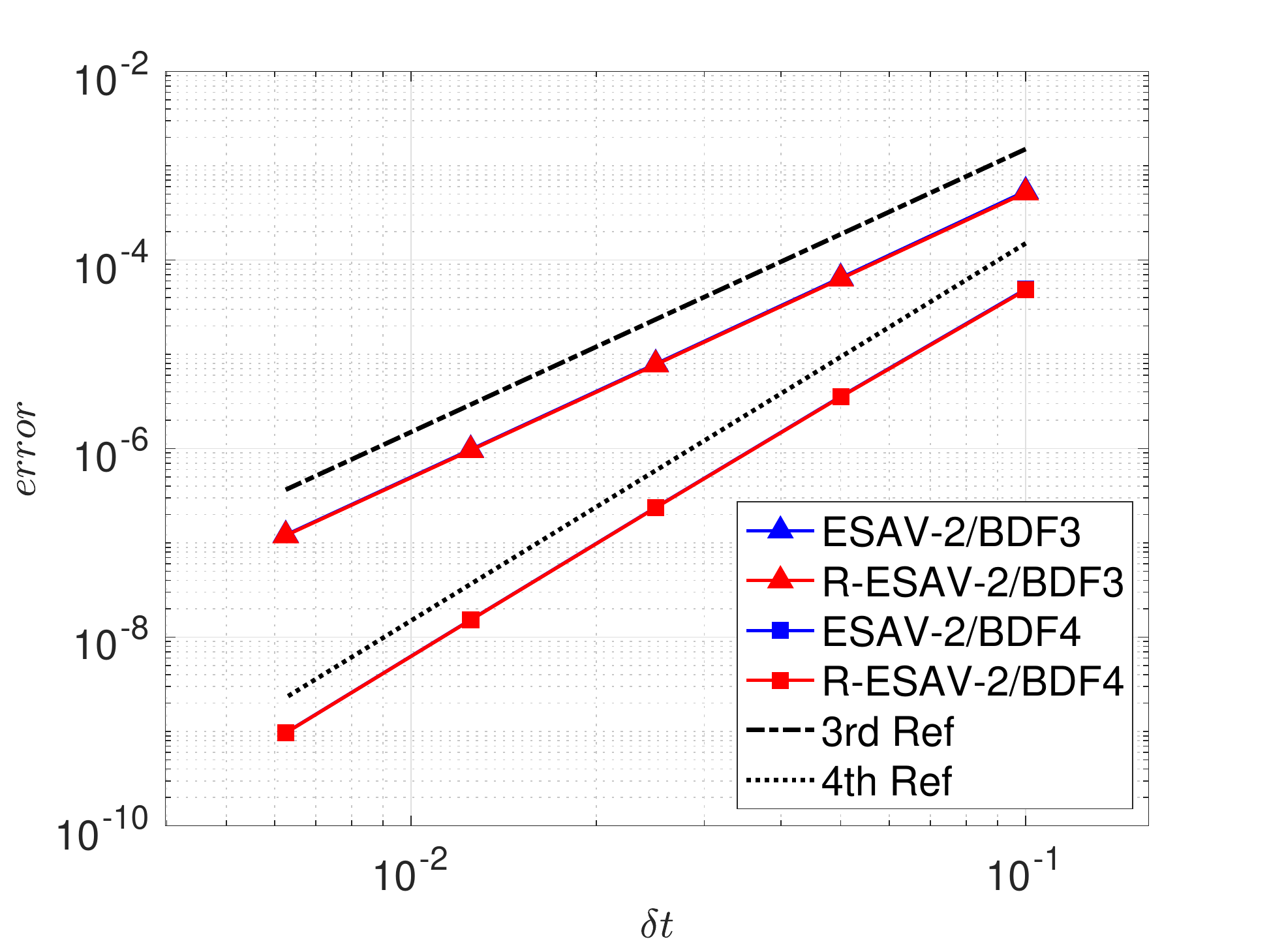}
	\caption{Example 2 (\romannumeral1). Convergence test for Cahn-Hilliard equation obtained by ESAV-1/BDF$k$, R-ESAV-1/BDF$k$ ($k=1, 2$), ESAV-2/BDF$k$ and R-ESAV-2/BDF$k$ ($k=1, 2, 3, 4$) schemes.}
	\label{Fig:CH-order-test-three-method}
\end{figure} 

(\romannumeral2)  We consider a rectangular array of $9 \times  9$ circles as  the initial condition
\begin{equation}
\phi_{0}(\boldsymbol{x}, t)=80-\sum_{m=1}^{9} \sum_{n=1}^{9} \tanh\left(\frac{\sqrt{\left(x-x_{m}\right)^{2}+\left(y-y_{n}\right)^{2}}-r_{0}}{\sqrt{2} \epsilon}\right),
\end{equation}
where $r_0 = 0.085, x_{m} = 0.2\times m, y_{n} = 0.2\times n$ for $m, n = 1, 2, \cdots, 9$. 
We set computational domain as $[0, 2]^2$, the other parameters are $M=1e-6, \sigma_0=1, \epsilon=0.01$ and Fourier modes are $N^2=512^2$  in the simulations. 
The evolutions of an array of circles governed by Cahn-Hilliard equation using the R-ESAV-2/BDF$2$ scheme with $\delta t =1e-3$ are shown in Fig.\,\ref{Fig:CH-circles-RESAV2}.

\begin{figure}[htbp]
\centering
\subfigure[profiles of $\phi=0$ at $T=0, 5, 10$]{
	\includegraphics[width=5.3cm]{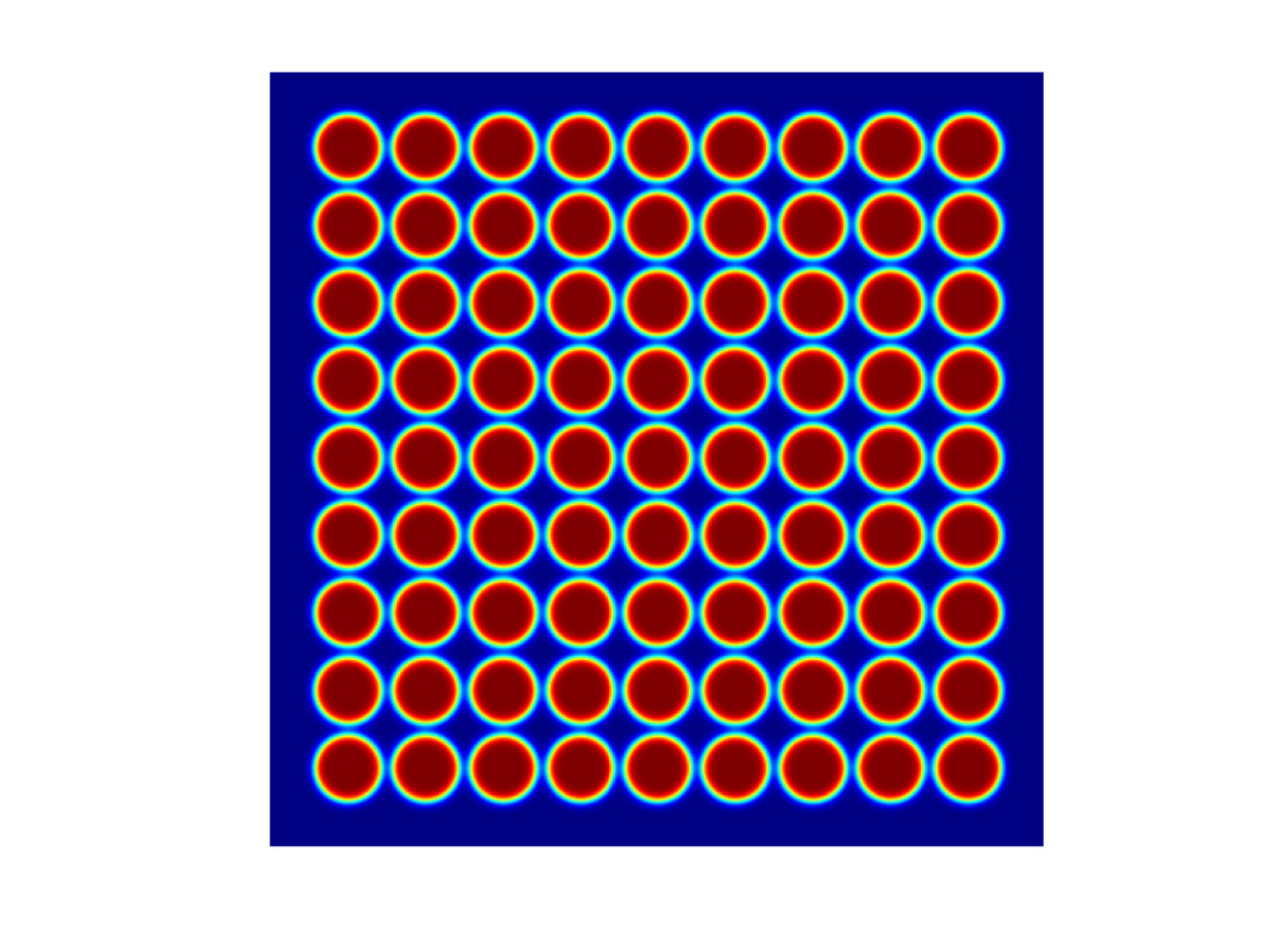}\hspace{-9mm}
	\includegraphics[width=5.3cm]{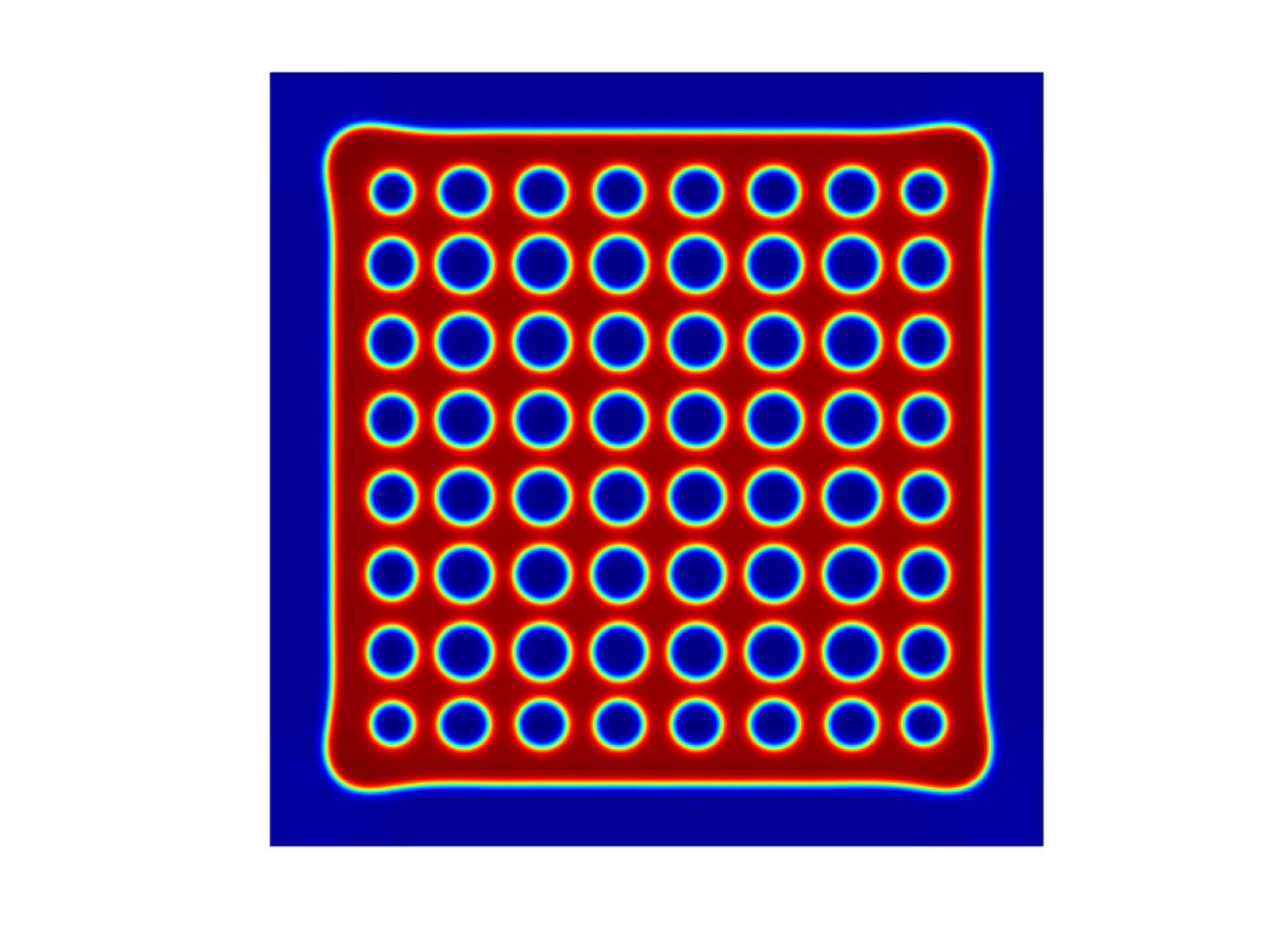}\hspace{-9mm}
	\includegraphics[width=5.3cm]{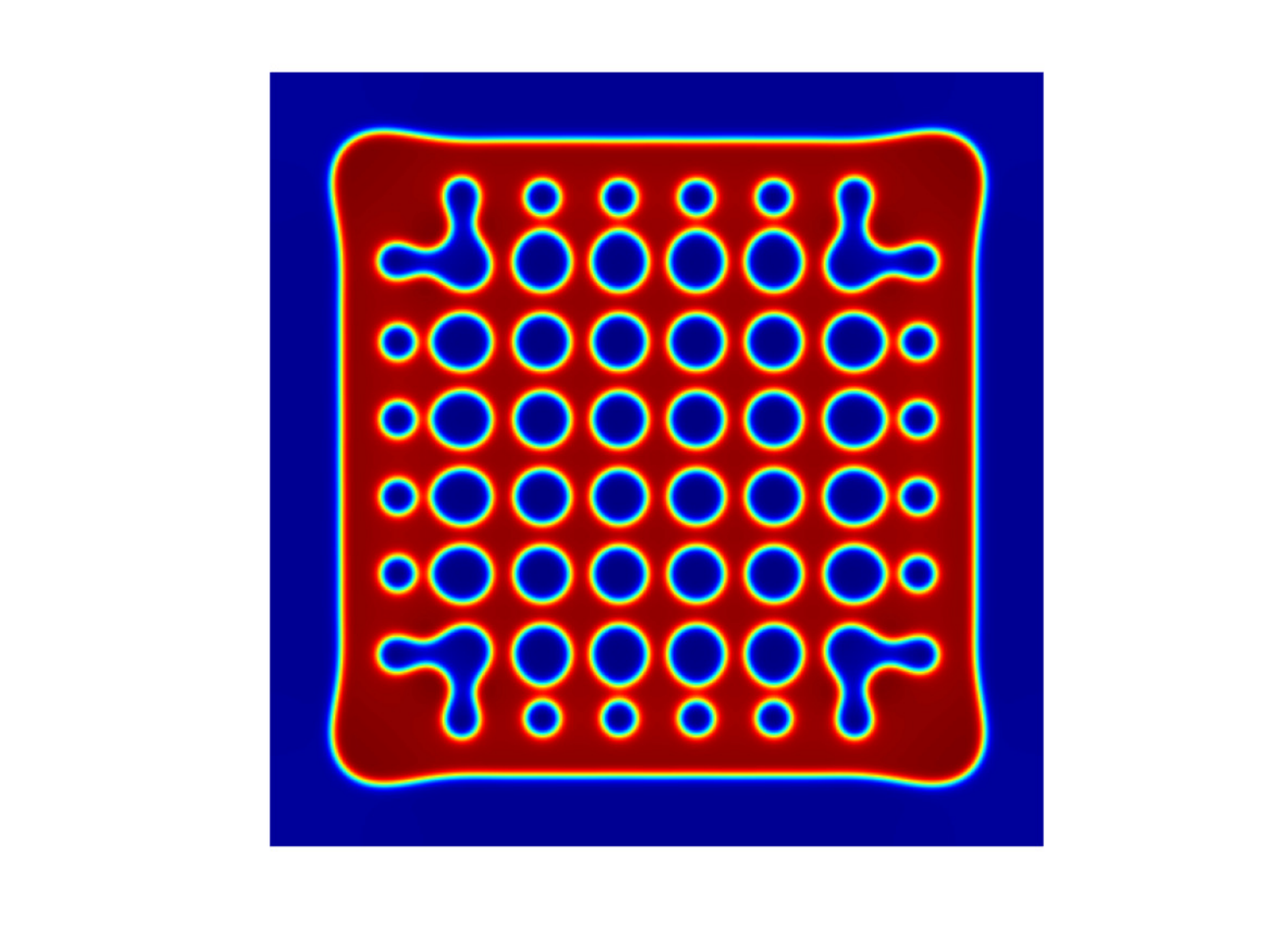}
} 
\subfigure[profiles of $\phi=0$ at $T=20, 30, 50$]{
	\includegraphics[width=5.3cm]{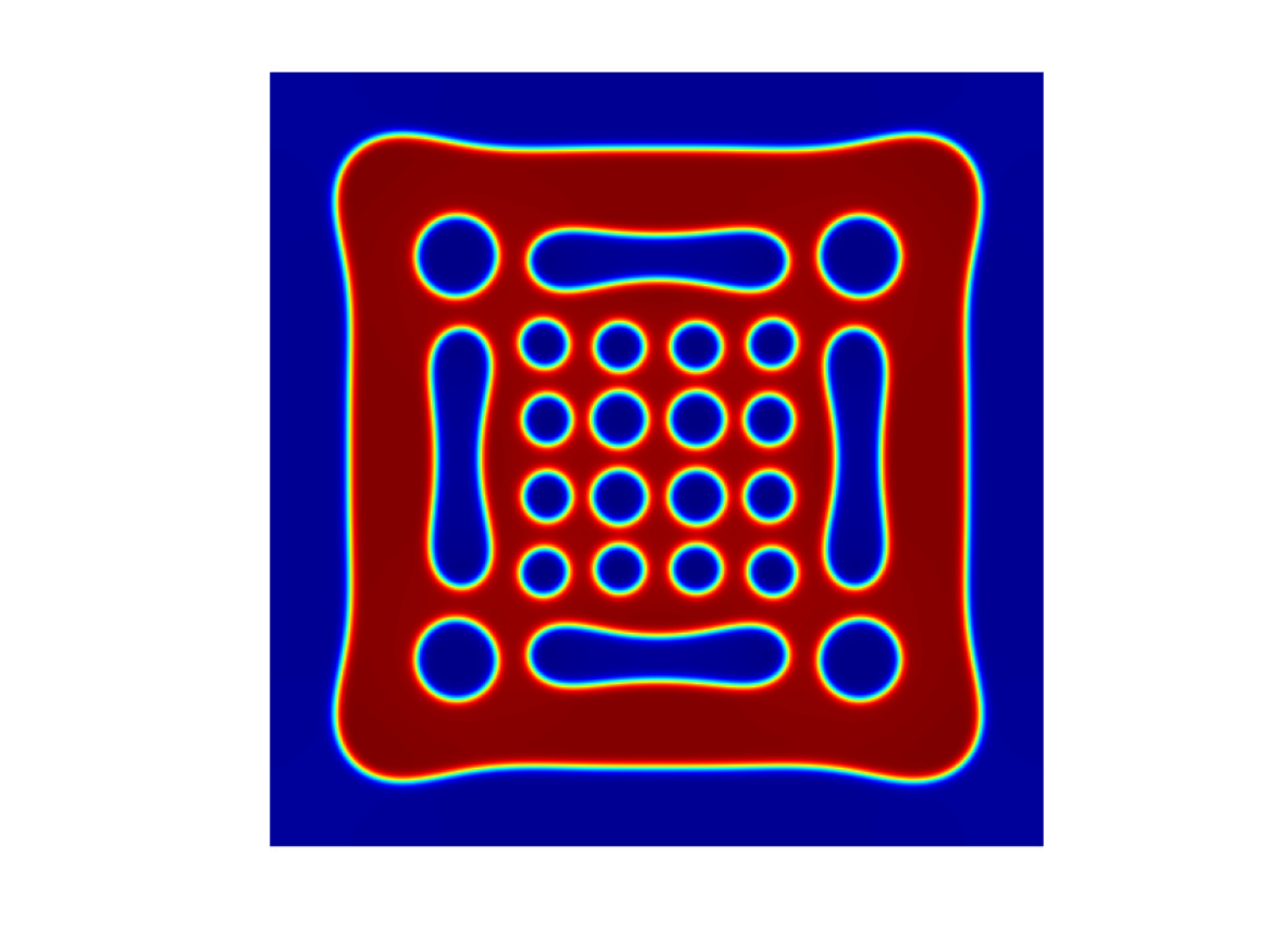}\hspace{-9mm}
	\includegraphics[width=5.3cm]{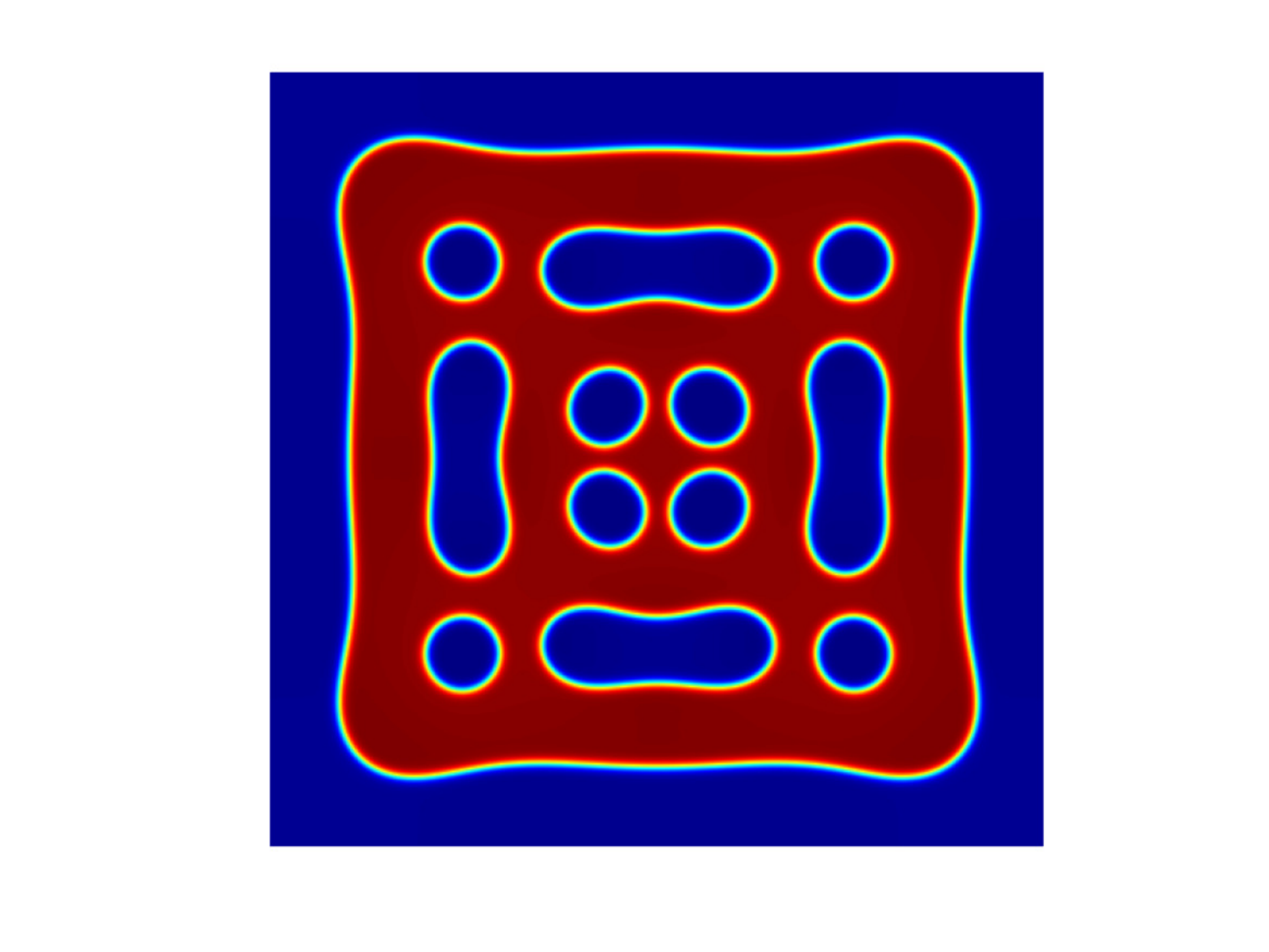}\hspace{-9mm}
	\includegraphics[width=5.3cm]{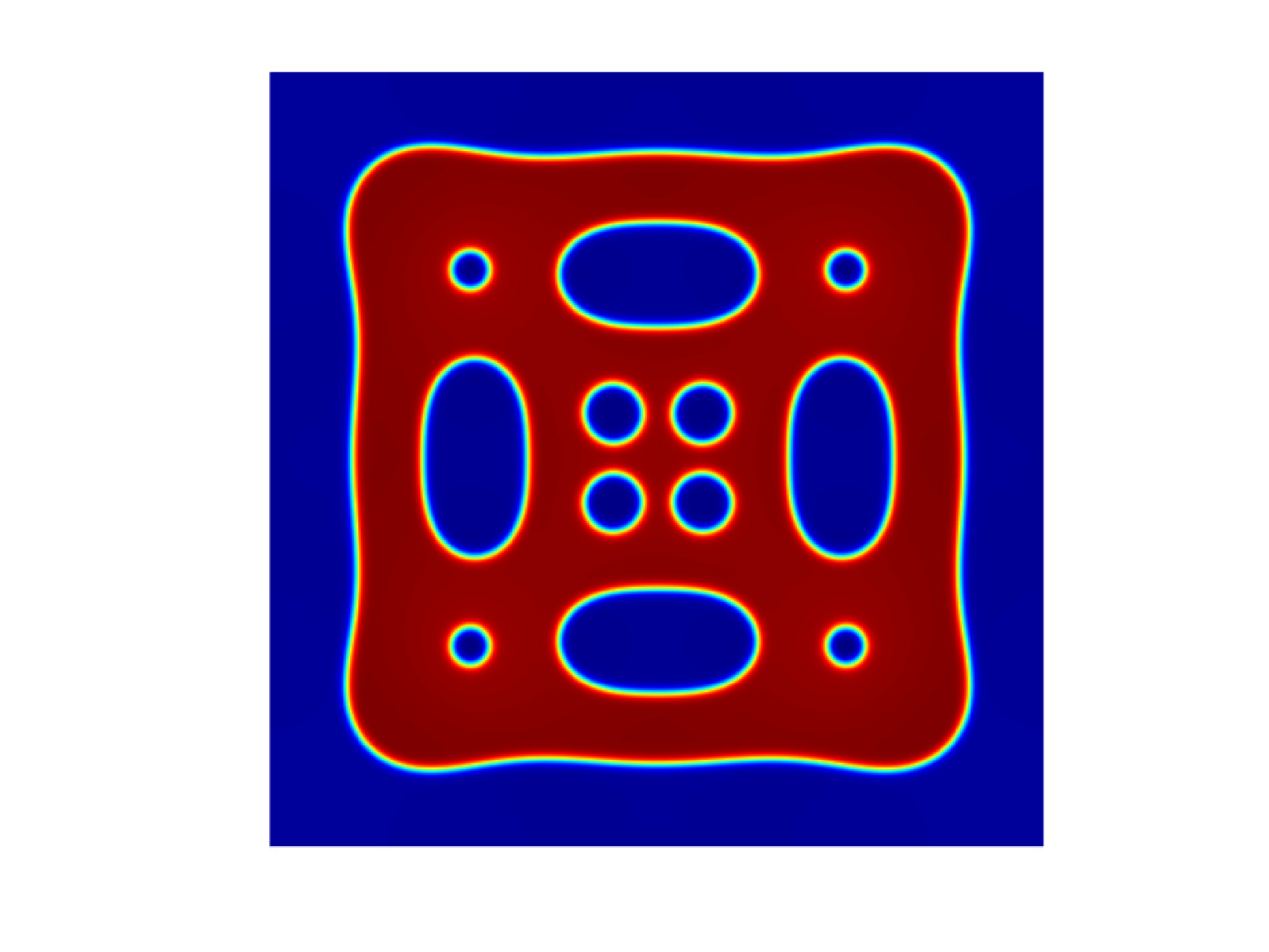}
}
\label{Fig:CH-circles-RESAV2}
\caption{Example 2 (\romannumeral2). The dynamic evolution of an array of circles governed by Cahn-Hilliard equation obtained by R-ESAV-2/BDF$2$ scheme.}
\end{figure}

\textbf{Example 3.} In this example, we consider the phase-field crystal model 
\begin{equation}
\left\{
\begin{array}{l}\frac{\partial \phi}{\partial t}=M \Delta \mu,  \quad \boldsymbol{x} \in \Omega, t>0, \\ 
\mu=(\Delta+\zeta)^{2} \phi+\phi^{3}-\epsilon \phi, \quad \boldsymbol{x} \in \Omega, t>0, \\ 
\phi(\boldsymbol{x}, 0)=\phi_{0}(\boldsymbol{x}),
\end{array}\right.
\end{equation} 
which is a $H^{-1}$-gradient flow associated with the total free energy as follows
\begin{equation}
	E(\phi)=\int_{\Omega}\left(\frac{1}{2} \phi(\Delta+\zeta)^{2} \phi+\frac{1}{4} \phi^{4}-\frac{\epsilon}{2} \phi^{2}\right) \mathrm{d} \boldsymbol{x},
\end{equation}
where $M>0$ is the mobility coefficient. 
We set $M=1, \zeta=1$ in the following simulations. 

(\romannumeral1)  We first simulate the crystal growth in a super-cooled liquid in $2D$. 
We adopt the following initial condition
\begin{equation}
\phi\left(x_{j}, y_{j}, 0\right)=\bar{\phi}+\alpha_{1}\left(\cos \left(\frac{\alpha_{2}}{\sqrt{3}} y_{j}\right) \cos \left(\alpha_{2} x_{j}\right)-0.5 \cos \left(\frac{2 \alpha_{2}}{\sqrt{3}} y_{j}\right)\right), \quad j=1,2, \cdots 5,
\end{equation}
where $x_j$ and $y_j$ define a local system of Cartesian coordinates, which is oriented with the crystallite lattice. 
We choose the constant parameters as $\bar{\phi} = 0.285, \alpha_{1} = 0.446, \alpha_{2} = 0.66$. 
Then, we define five crystallites in five small square patches, each  with a length of $40$, located at $(200, 200)$, $(150, 600)$, $(350, 400)$, $(600, 300)$ and $(550, 700)$ respectively. 
In order to produce crystallites with different orientations, we use the following affine transformation to generate rotation by five diffferent angles $\rho=-\frac{3 \pi}{4}, -\frac{\pi}{4}, 0, \frac{\pi}{4}, \frac{3 \pi}{4}$ respectively 
\begin{equation}
x_{j}(x, y)=x \sin (\rho)+y \cos (\rho), \quad y_{j}(x, y)=-x \cos (\rho)+y \sin (\rho).
\end{equation}
We choose Fourier modes $N^2=1024^2$ to discretize the space and $\delta t =0.02$ to discretize the time. 
And we set the other parameters as $\epsilon=0.25, T=800$. 
The relaxation parameter $\theta_0^{n+1}$ is also always zero in this example. 
We present the crystal growth in a super-cooled liquid by using the R-ESAV-1/BDF$2$ scheme for the PFC equation in Fig.\,\ref{Fig:PFC-2D-RESAV1}, which indicate that the different alinement of the crystallites leads to defects and dislocations, just similar with those in \cite{yang2017linearly, li2020stability}.

\begin{figure}[htbp]
\centering
\subfigure[profiles of $\phi$ at $T=0, 100, 200, 300$]{
	\hspace{-9mm}
	\includegraphics[width=4.7cm]{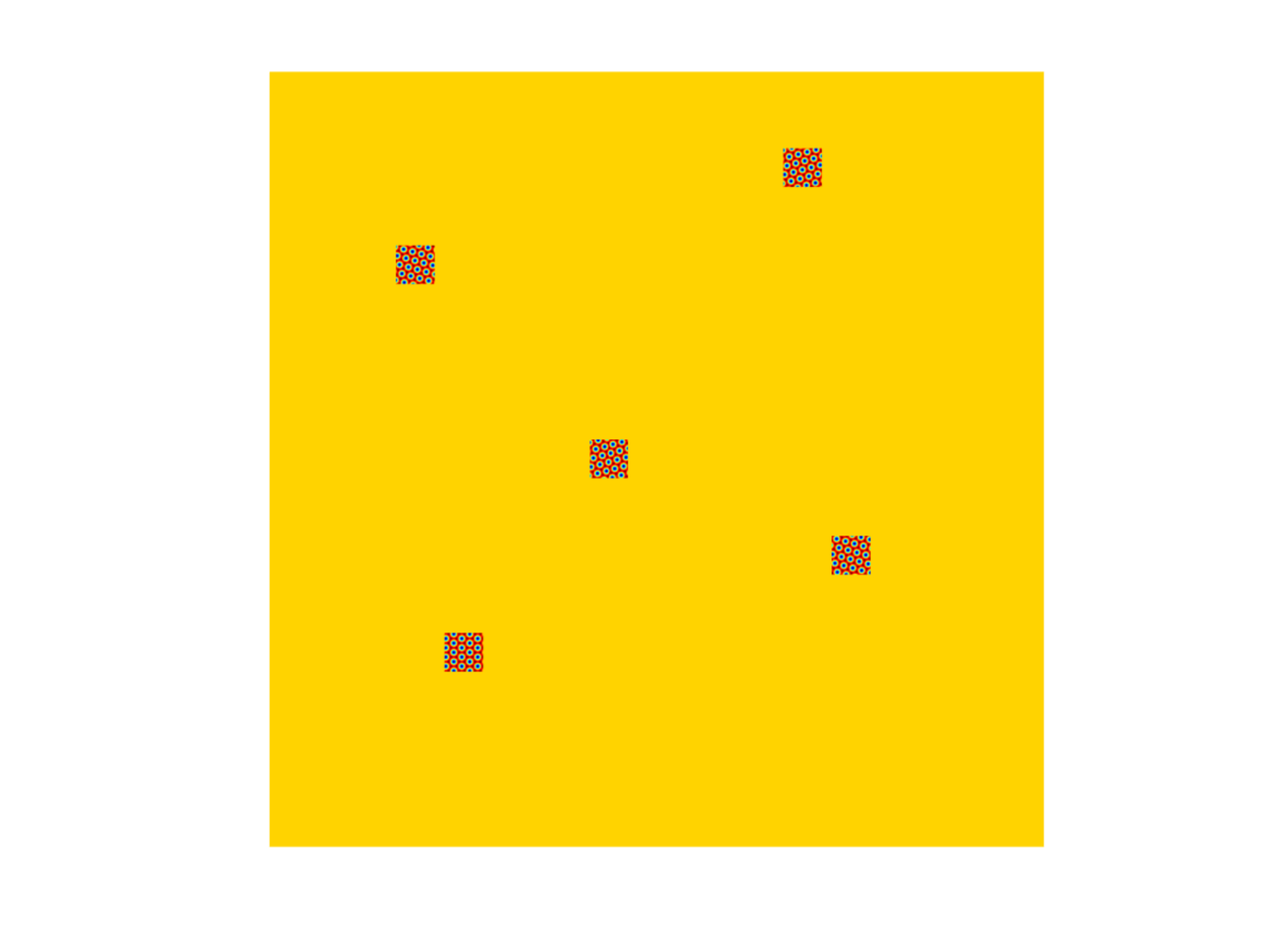}\hspace{-9mm}
	\includegraphics[width=4.7cm]{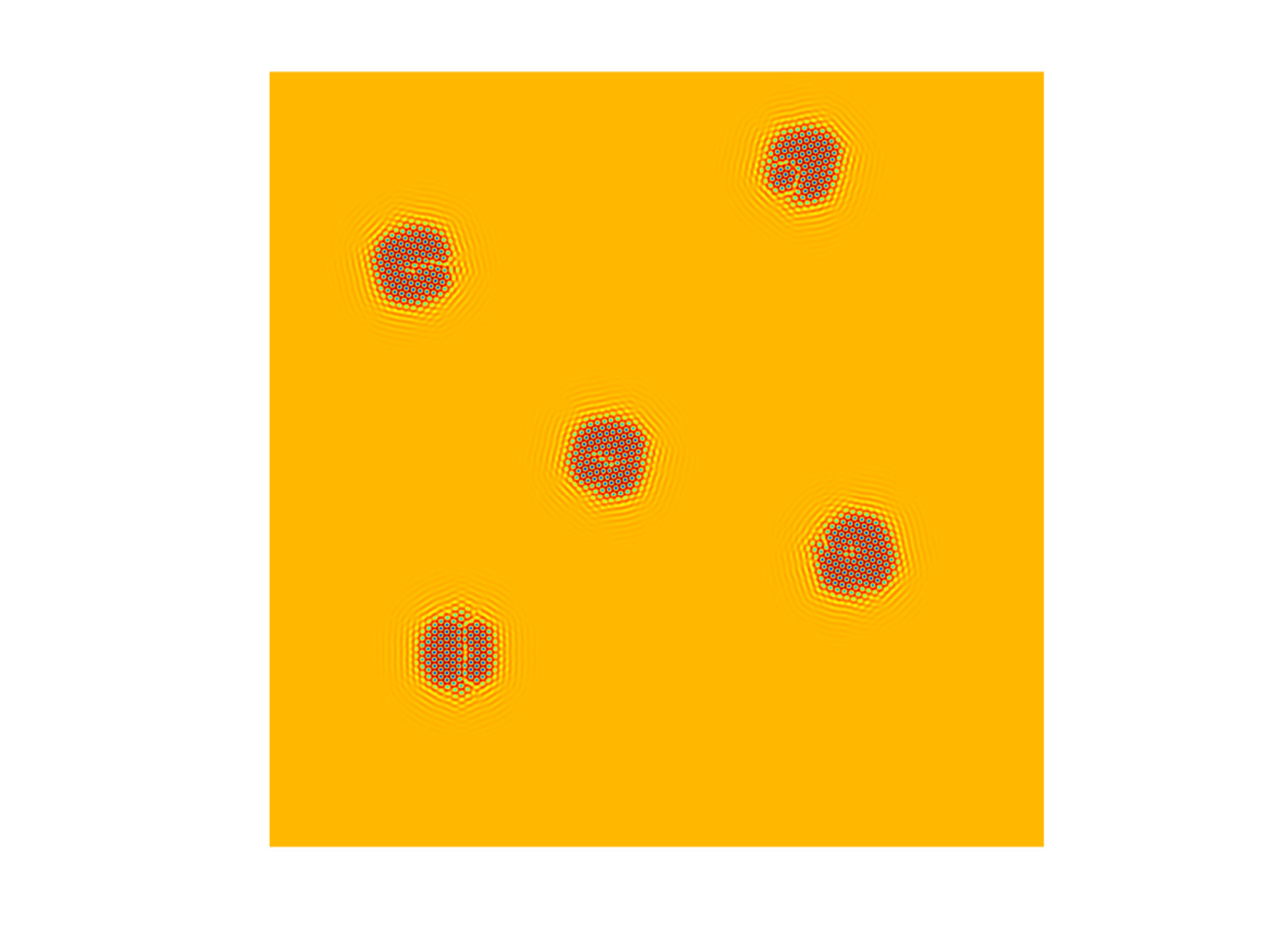}\hspace{-9mm}
	\includegraphics[width=4.7cm]{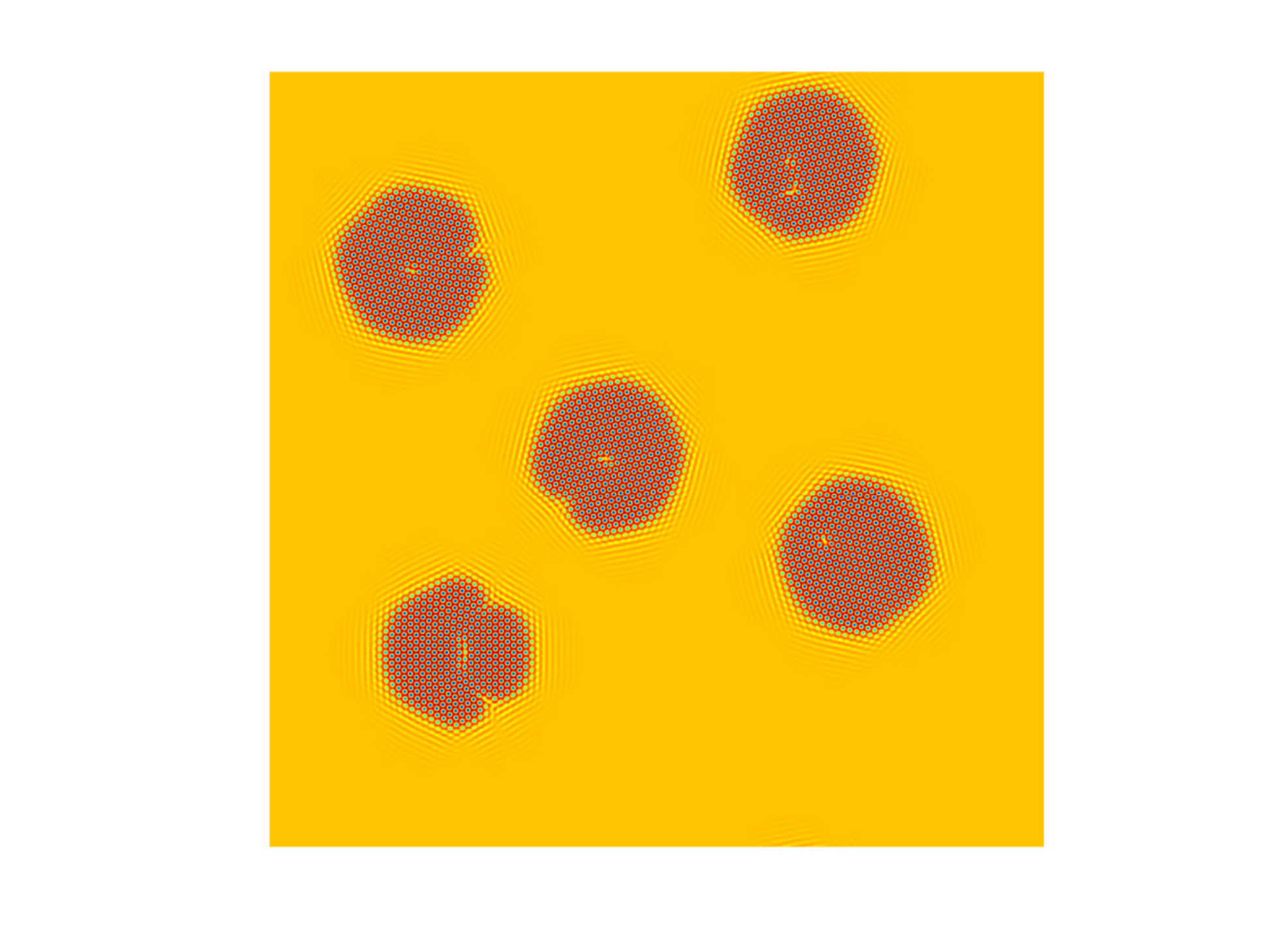}\hspace{-9mm}
	\includegraphics[width=4.7cm]{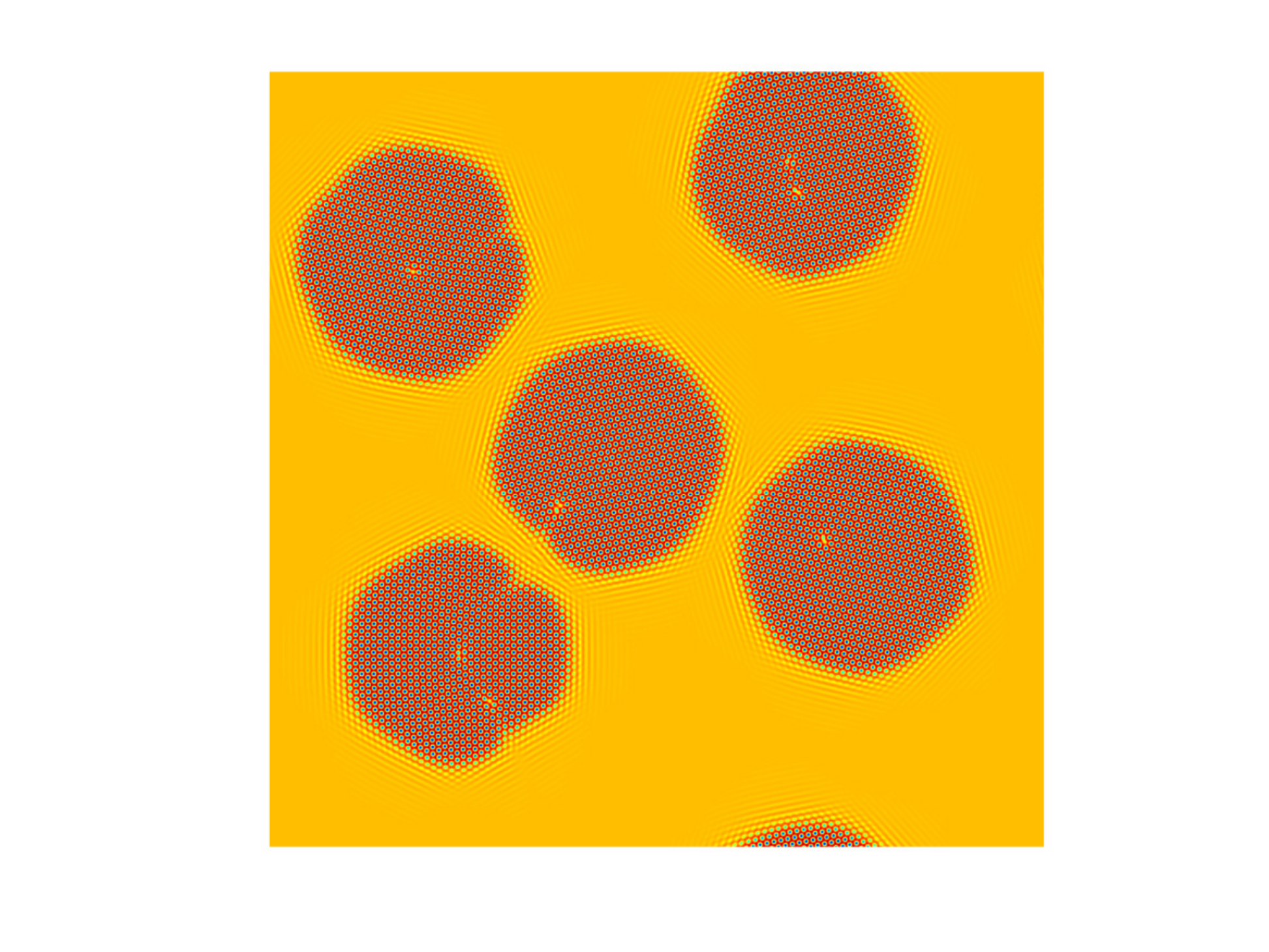}
} 
\subfigure[profiles of $\phi$ at $T=400, 500, 600, 800$]{
   \hspace{-9mm}
	\includegraphics[width=4.7cm]{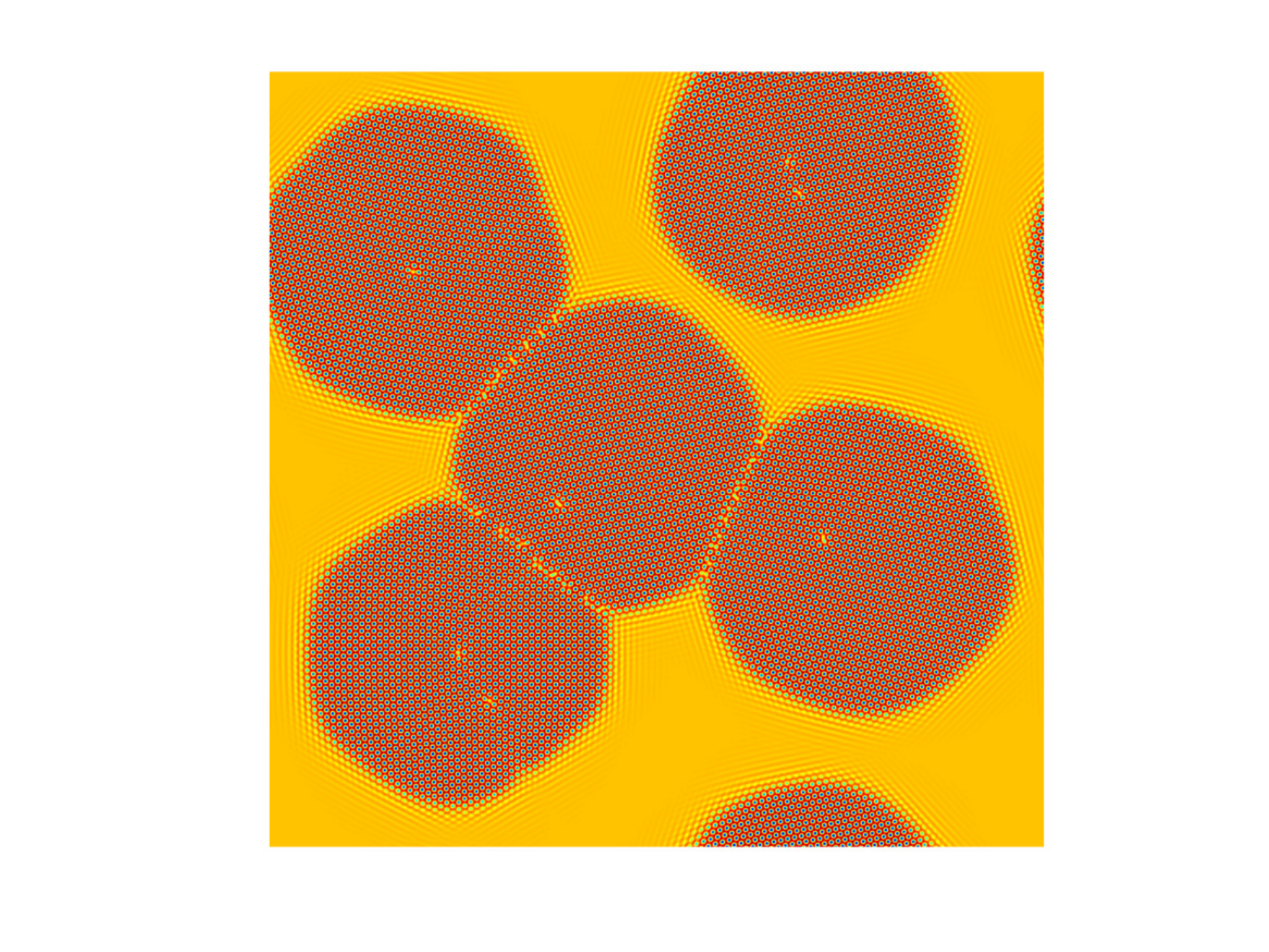}\hspace{-9mm}
	\includegraphics[width=4.7cm]{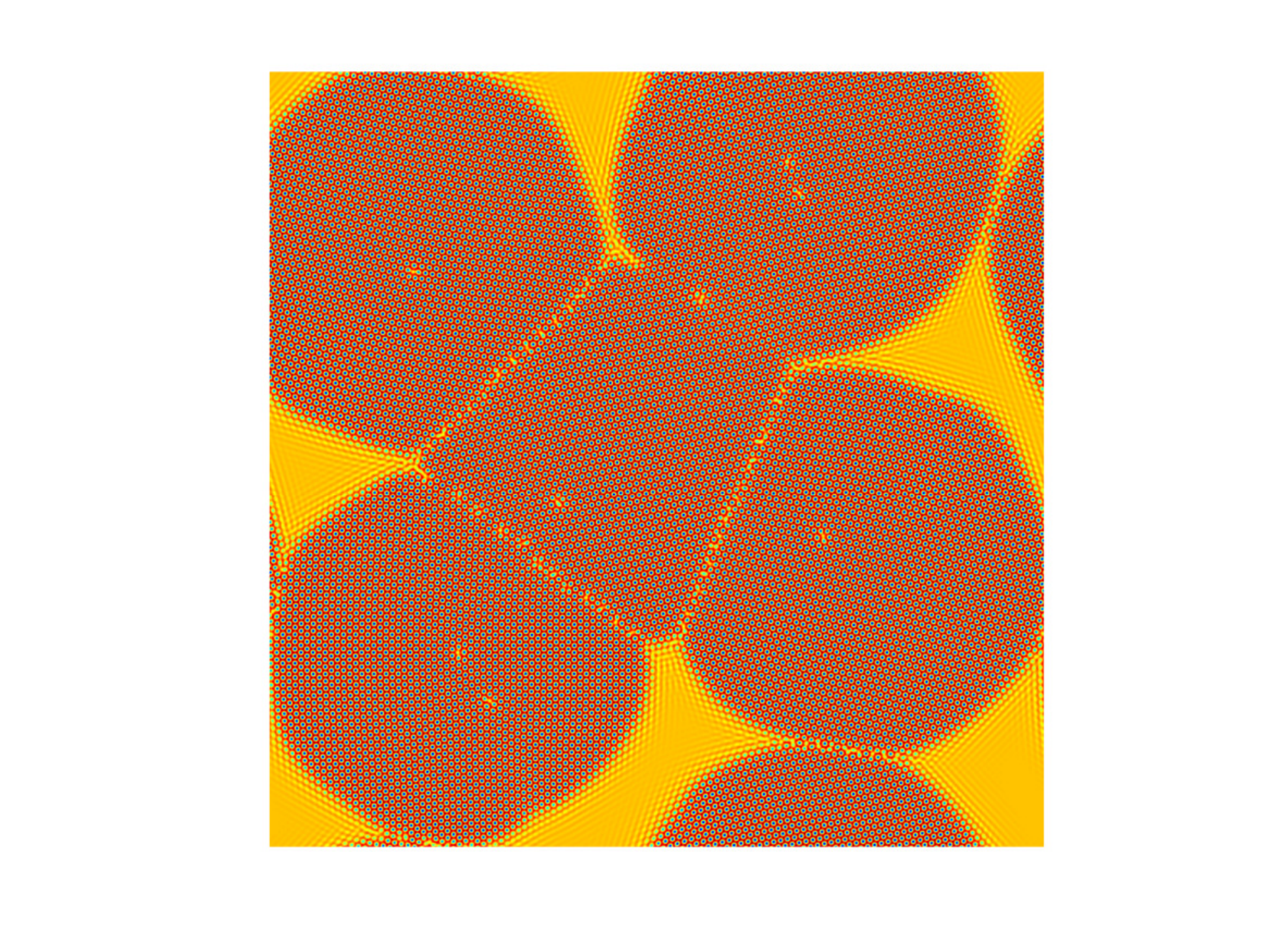}\hspace{-9mm}
	\includegraphics[width=4.7cm]{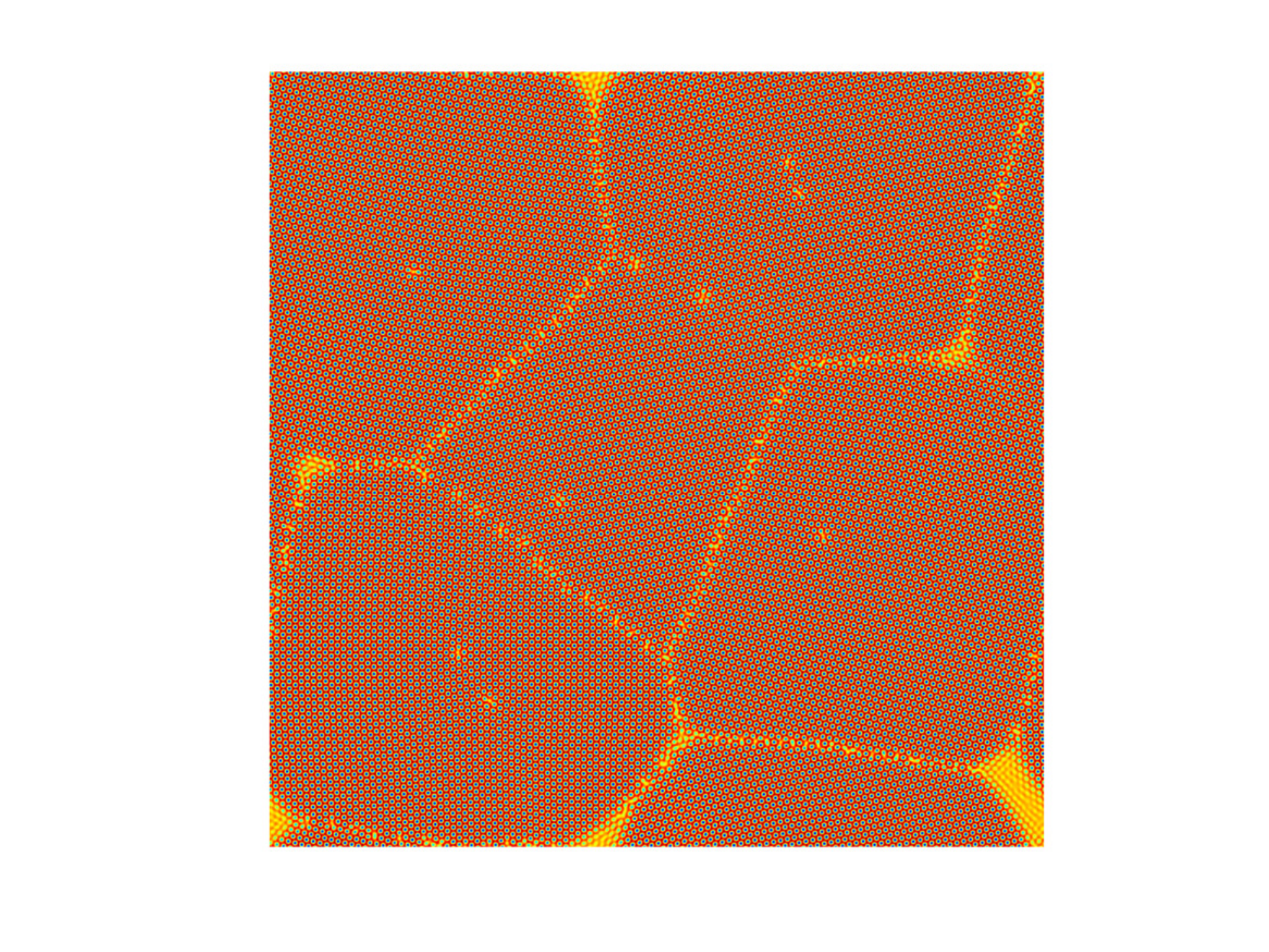}\hspace{-9mm}
	\includegraphics[width=4.7cm]{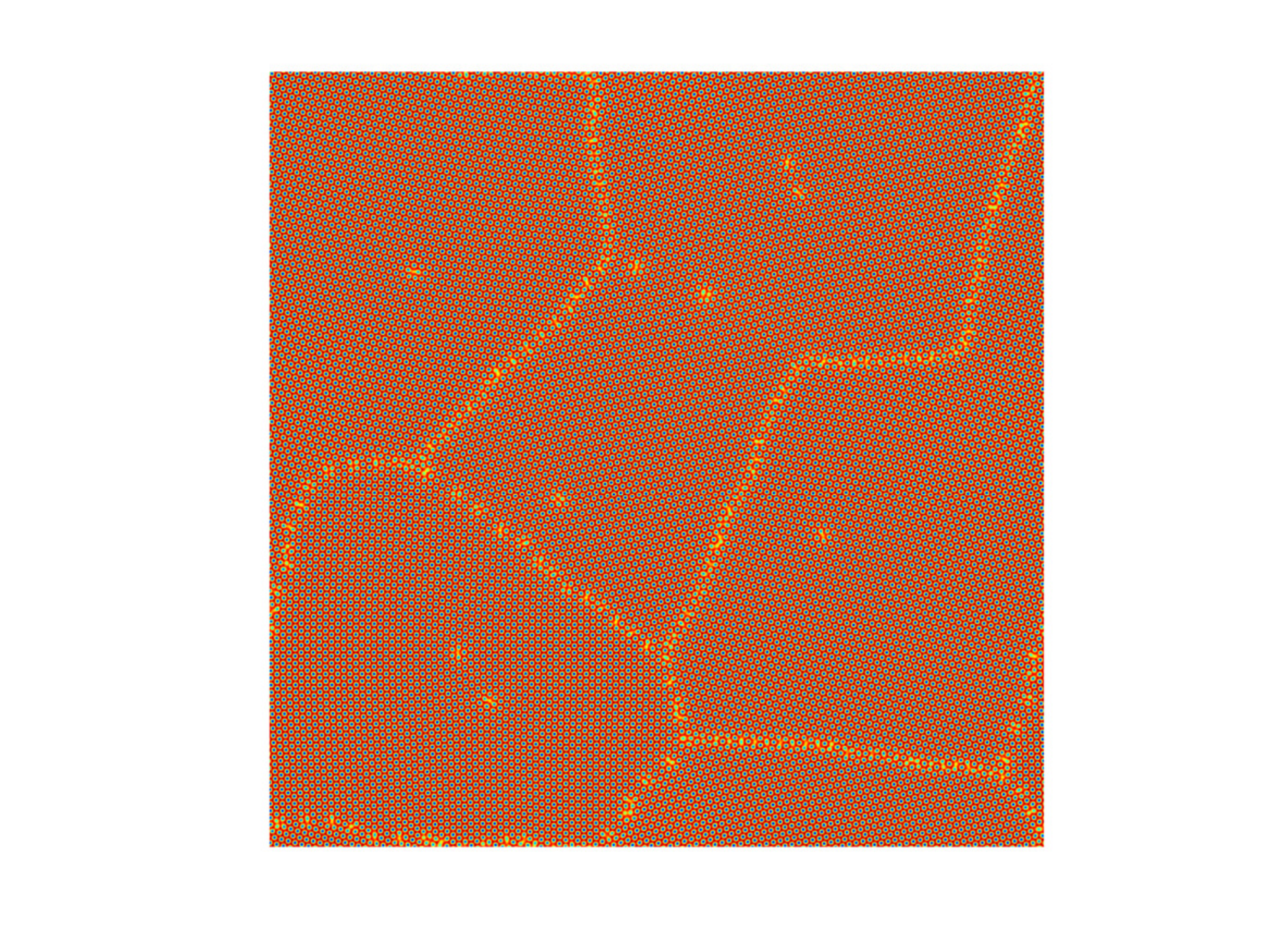}
}
\label{Fig:PFC-2D-RESAV1}
\caption{Example 3 (\romannumeral1). The $2$D dynamic evolution of crystal growth in a supercooled liquid driven by the PFC equation obtained by R-ESAV-1/BDF$2$ scheme. Snapshots of the numerical solution $\phi$ at $T=0,$ $100,$ $200,$ $300,$ $400,$ $500,$ $600,$ $800,$ respectively.}
\end{figure}

(\romannumeral2)  
The initial condition is set  to be
\begin{equation}
	\phi\left(x, y, 0\right)=0.07+0.07\delta,
\end{equation}
where $\delta=rand(x, y)$ is the uniformly distributed random number in $[-1, 1]$ with zeros mean.  Set the computational domain to $\Omega=[0, 128]^{2}$. 
The model parameter is chosen as $\epsilon=0.025$, time step is $\delta t=0.1$ and $N^{2}=256^{2}$ Fourier modes to discretize the space. 
It presents the configuration evolution in Fig.\,\ref{Fig:PFC-rand-2D-RESAV1}  and can observe from that uniform phase separation is formed finally. 
Similar computation results can be found in  \cite{liu2020exponential, liu2021highly}. 

\begin{figure}[htbp]
	\centering
		\hspace{-10mm}
		\includegraphics[width=4.7cm]{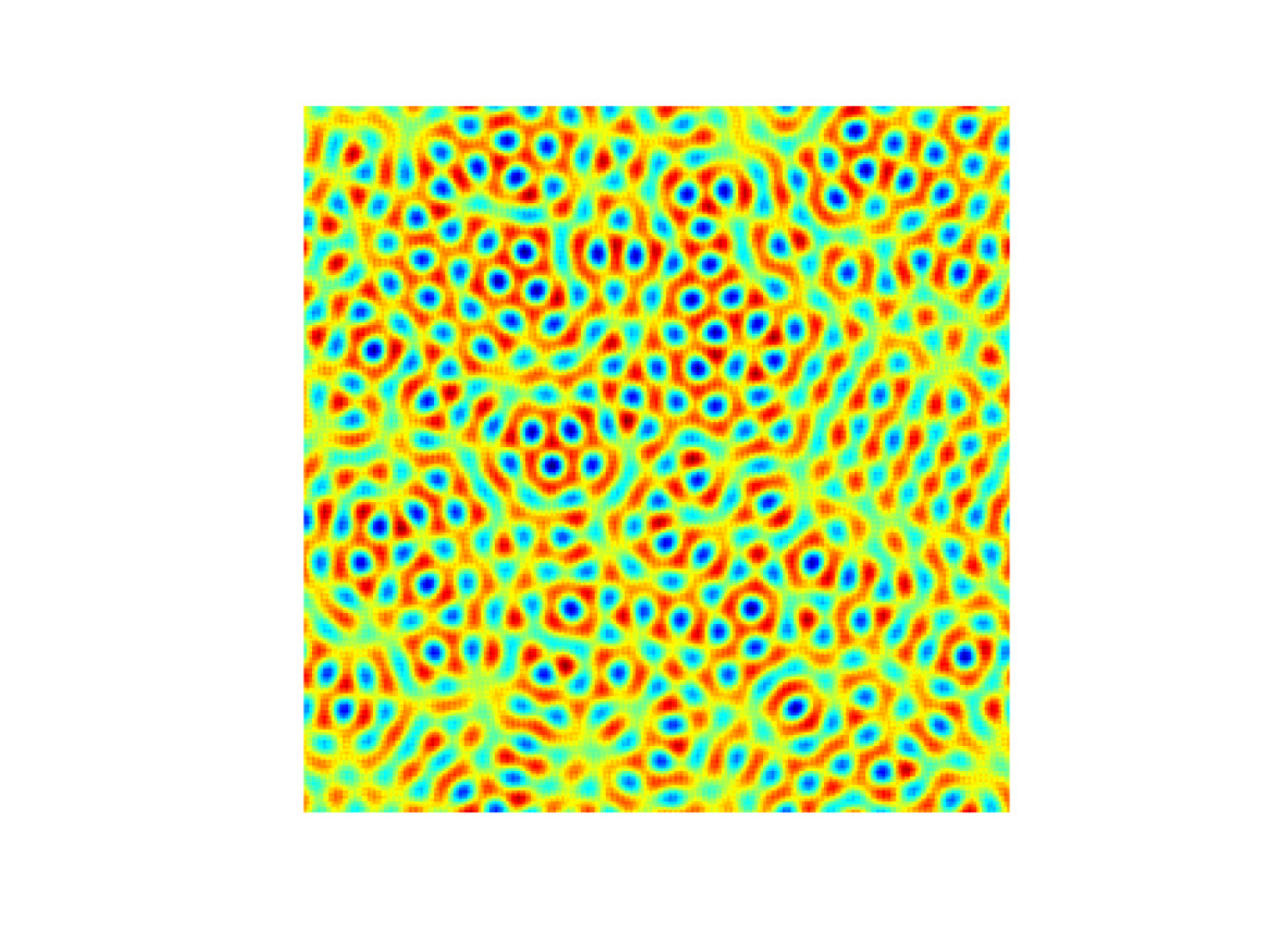}	\hspace{-12mm}
		\includegraphics[width=4.7cm]{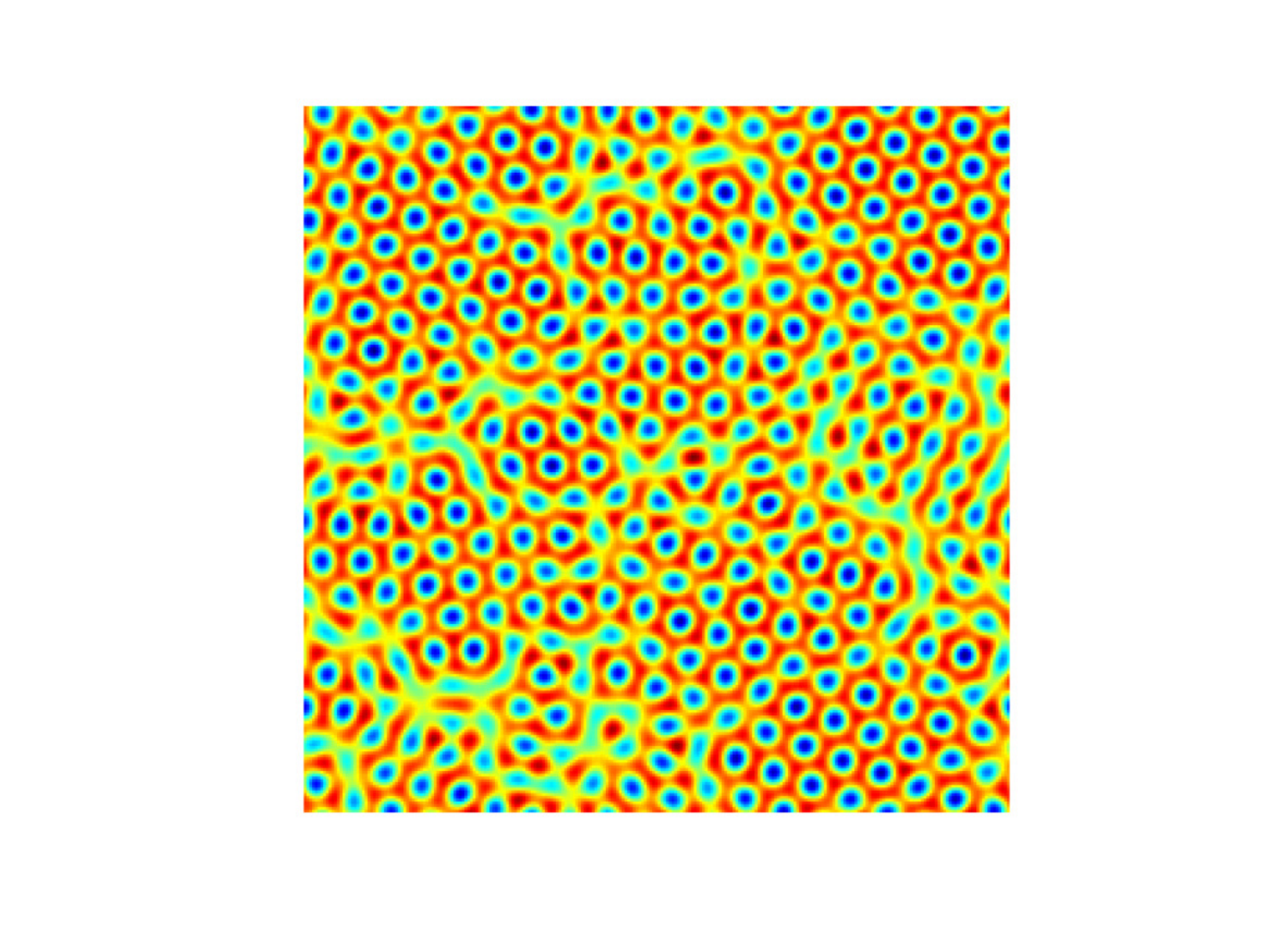}	\hspace{-12mm}
		\includegraphics[width=4.7cm]{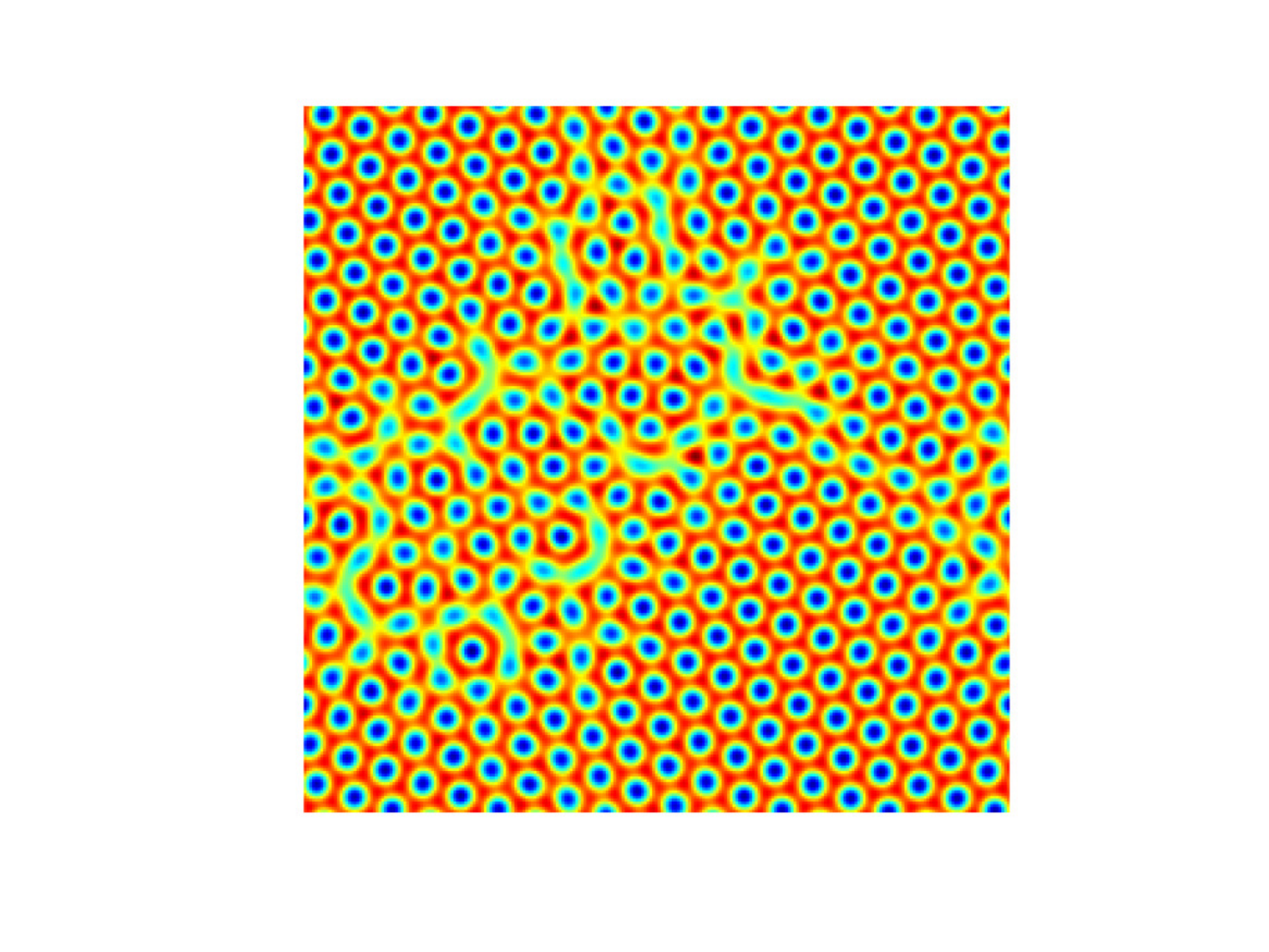}	\hspace{-12mm}
		\includegraphics[width=4.7cm]{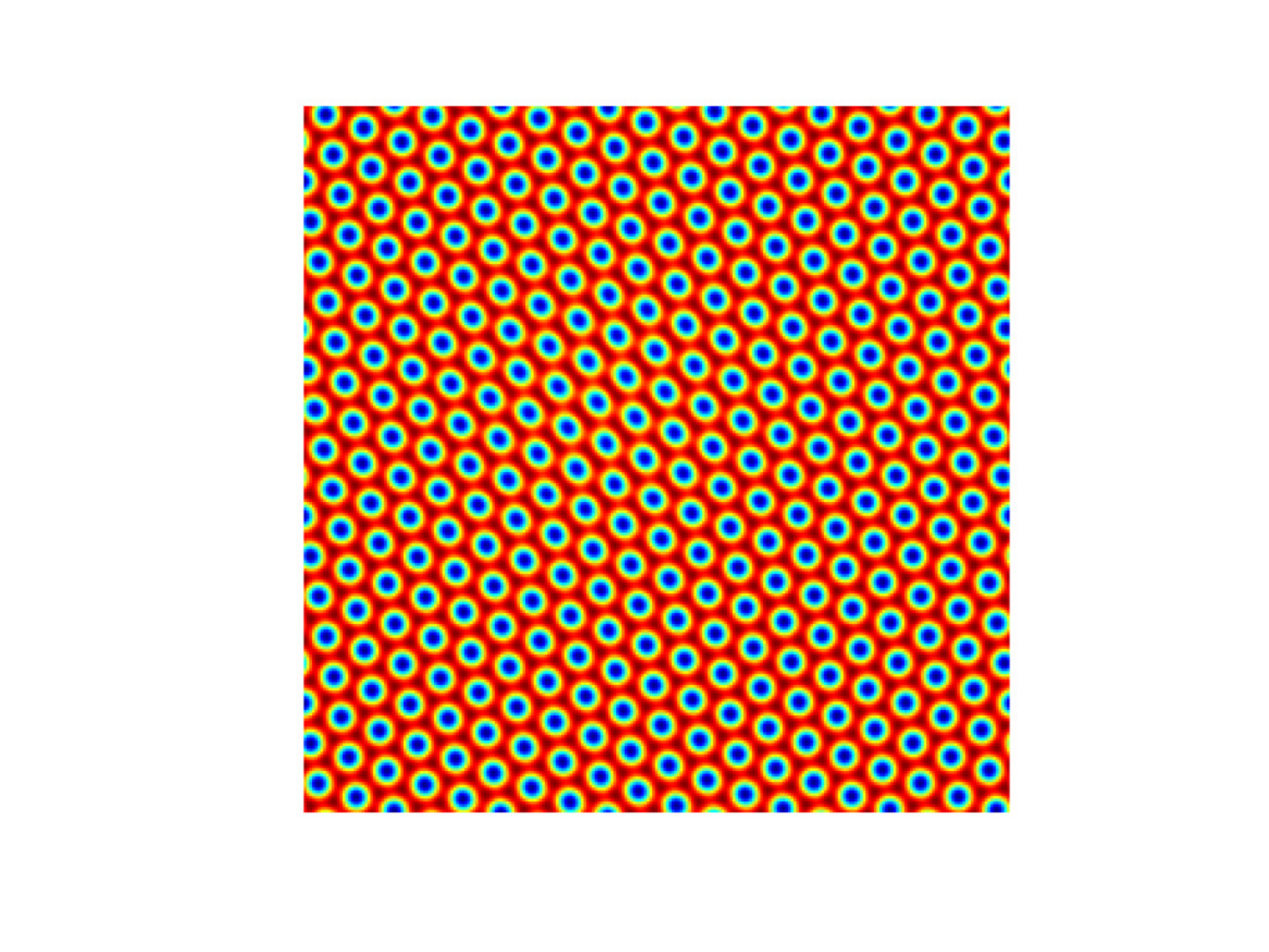}\hspace{-12mm}
	\caption{Example 3 (\romannumeral2). The $2$D configuration evolution driven by the PFC equation obtained by R-ESAV-1/BDF$2$ scheme. Snapshots of the numerical solution $\phi$ at $T=400,$ $600,$ $1000,$ $2500,$ respectively.}
	\label{Fig:PFC-rand-2D-RESAV1}
\end{figure}

(\romannumeral3) Then we simulate the crystal growth in a super-cooled liquid in $3D$. 
We choose the computational domain $[0, 100]^3$ and the initial condition are two crystallites generated by 
$$\phi(x, y, t=0)=0.285+0.01\delta.$$
The other parameters are chosen as $\epsilon=0.25, \delta t = 0.02, T=2000$. 
We adopt Fourier modes $N^3=128^{3}$  to discretize space. 
It can be observed that the effects of different arrangement of crystallites on the growth of the crystalline phase and the motion of crystal-liquid interfaces in Fig.\,\ref{Fig:PFC-3D-crystal-growth-ESAV1}, where these results are also consistent with those in \cite{li2019efficient}.

\begin{figure}[htbp]
\centering
\subfigure[profiles of $\phi=0$ at $T=0, 60, 900$]{
	\includegraphics[width=5.3cm]{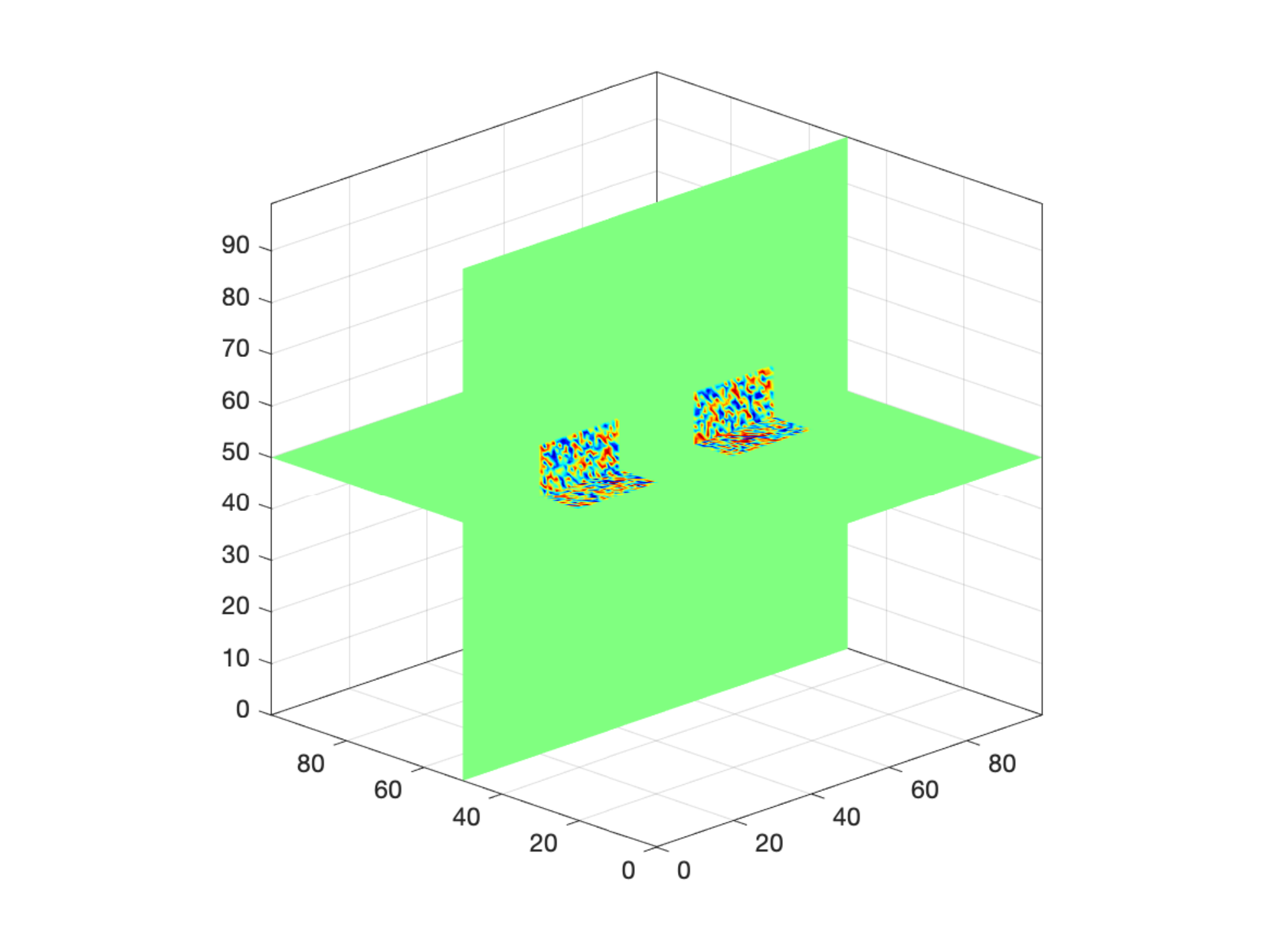}\hspace{-9mm}
	\includegraphics[width=5.3cm]{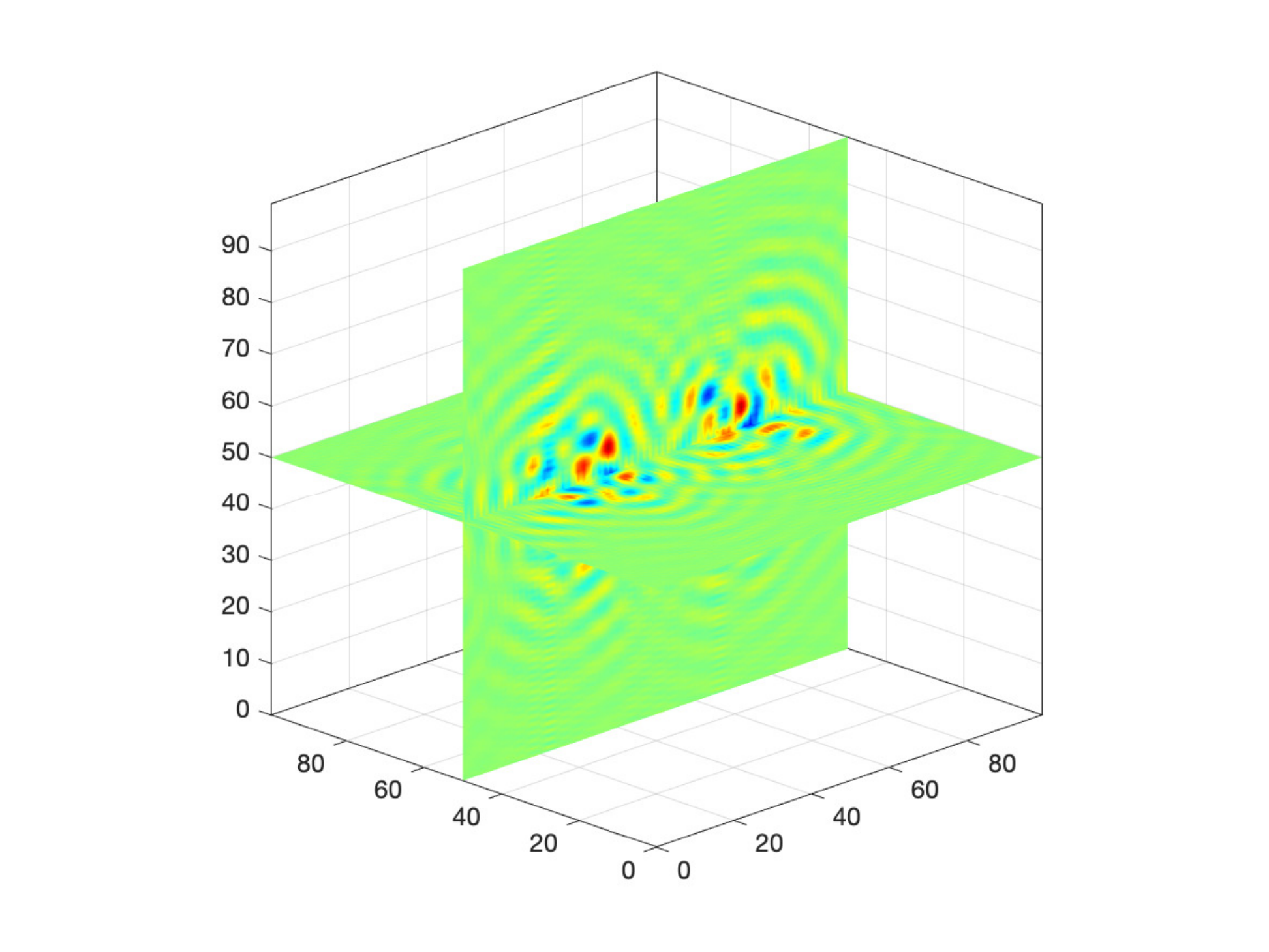}\hspace{-9mm}
	\includegraphics[width=5.3cm]{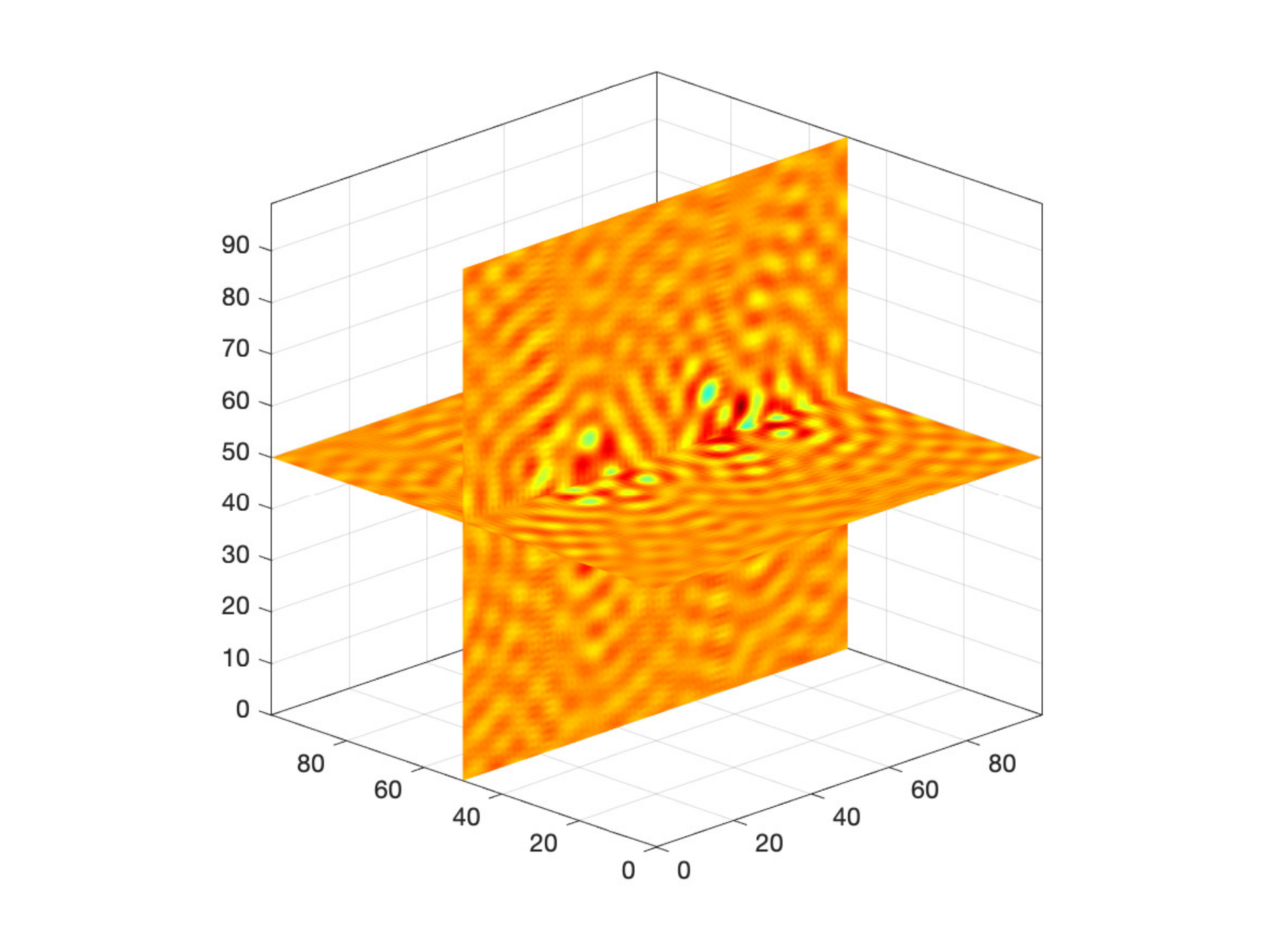}
} 
\subfigure[profiles of $\phi=0$ at $T=1000, 1100, 2000$]{
	\includegraphics[width=5.3cm]{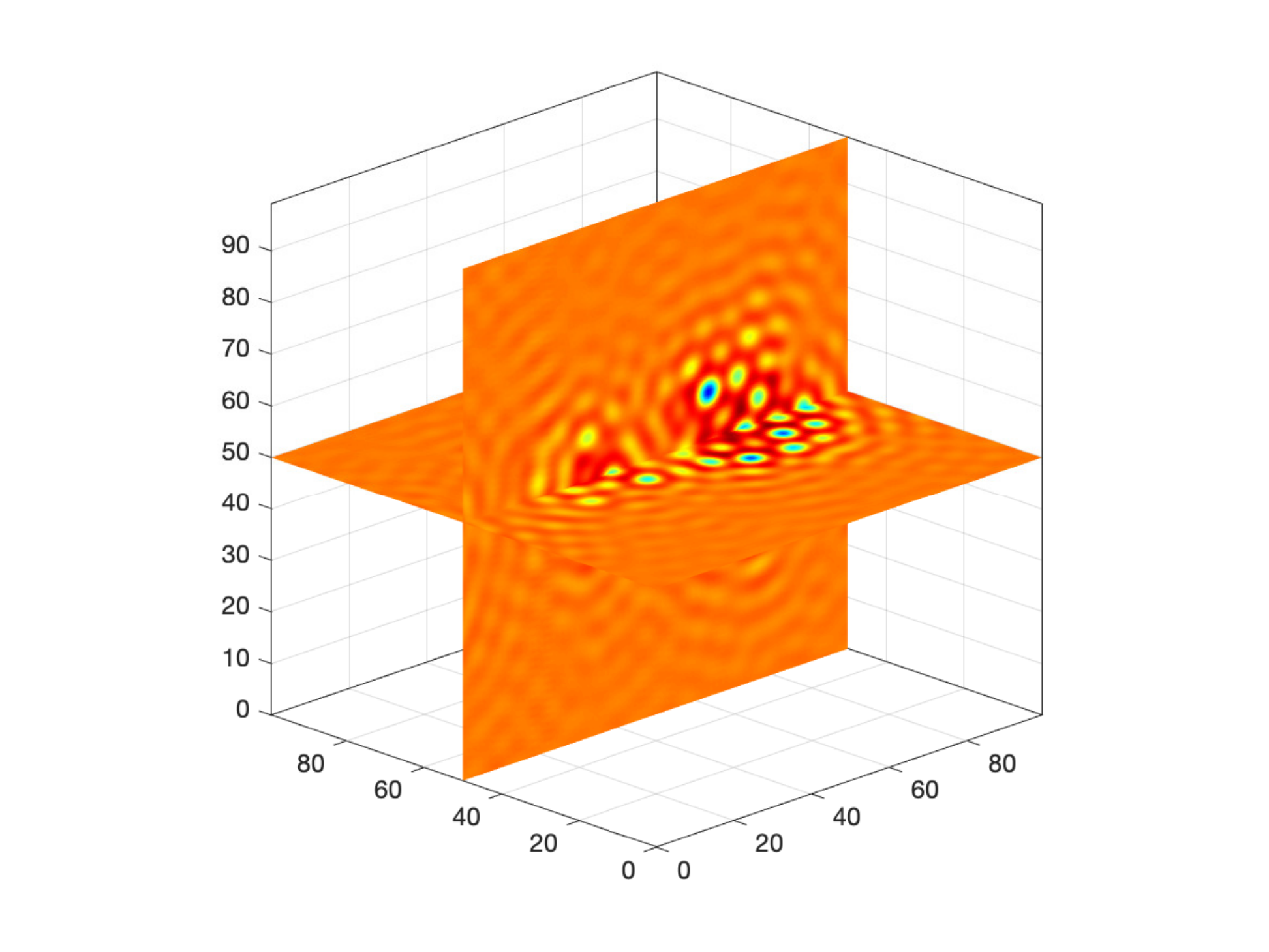}\hspace{-9mm}
	\includegraphics[width=5.3cm]{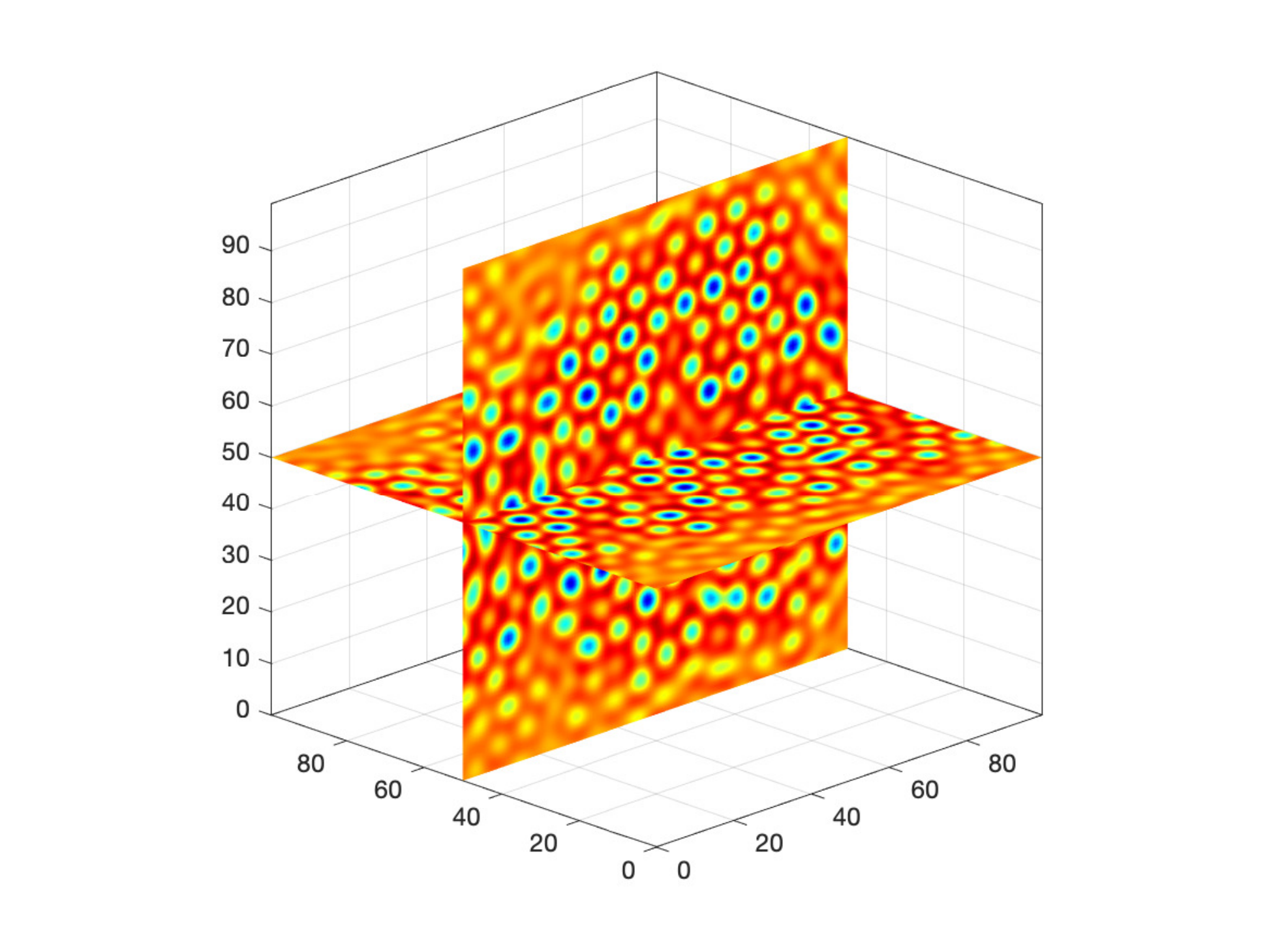}\hspace{-9mm}
	\includegraphics[width=5.3cm]{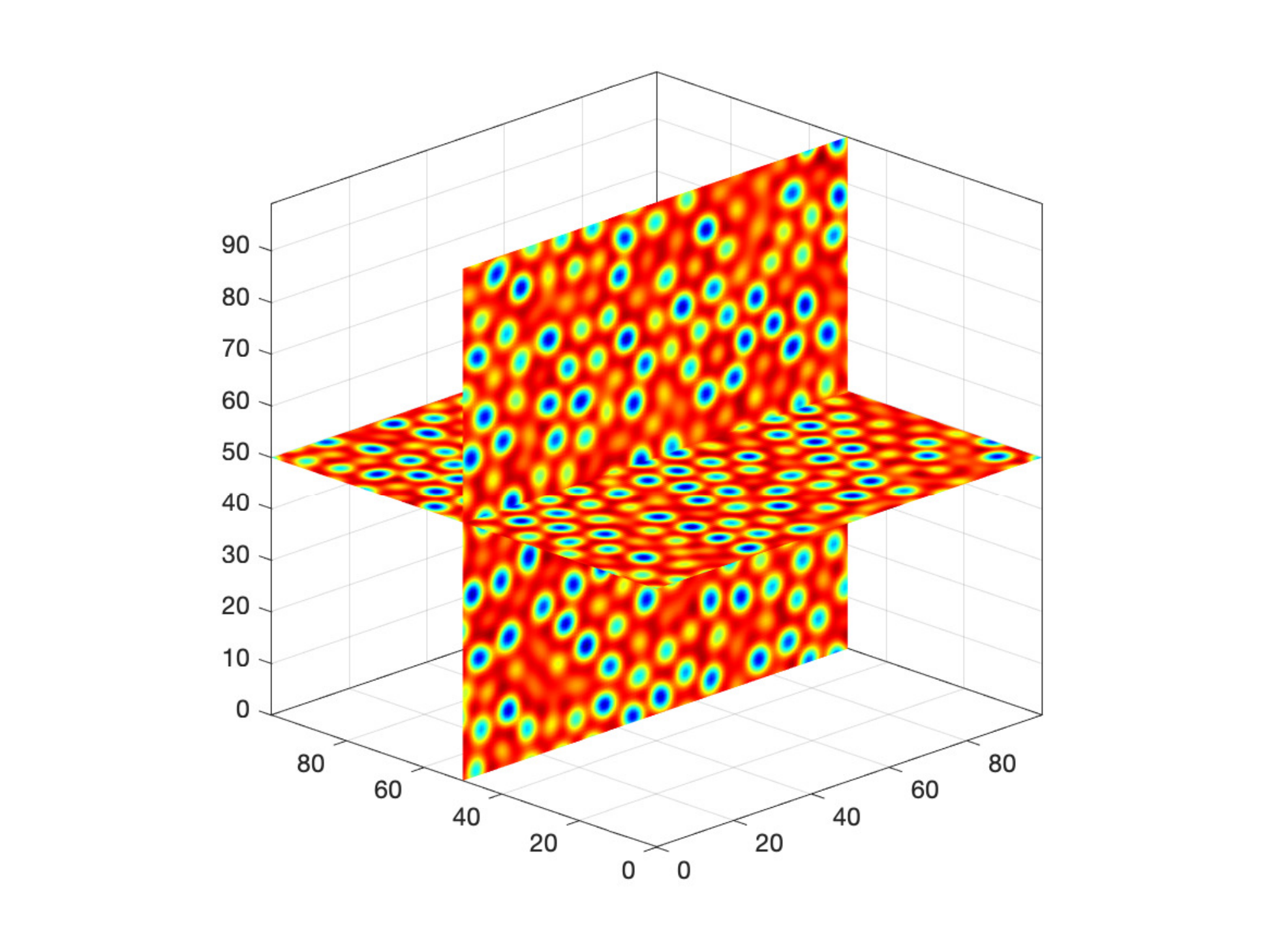}
}
\label{Fig:PFC-3D-crystal-growth-ESAV1}
\caption{Example 3 (\romannumeral3). The $3$D dynamic evolution of crystal growth in a supercooled liquid driven by the PFC equation obtained by R-ESAV-1/BDF$2$ scheme. Snapshots of the numerical solution $\phi$ at $T=0,$ $60,$ $900,$ $1000,$ $1100,$ $2000,$ respectively.}
\end{figure}

(\romannumeral4)  Finally we study phase transition behaviors in $3D$.  The initial data are chosen as
$$\phi(x, y, t=0)=\bar{\phi}+0.01\delta,$$  
and set computational domains $\left[0, 50\right]^{3}$. 
Other parameters are chosen as $\epsilon=0.56, \delta t=0.02, T=3000$ and Fourier modes $N^3=64^3$. We present the steady state
microstructure of the phase transition behavior for $\bar{\phi}=0.20, 0.35$ and $0.43$, respectively in Fig.\,\ref{Fig:PFC-3D-rand-ESAV1}. 


\begin{figure}[htbp]
\centering
	\includegraphics[width=5.3cm]{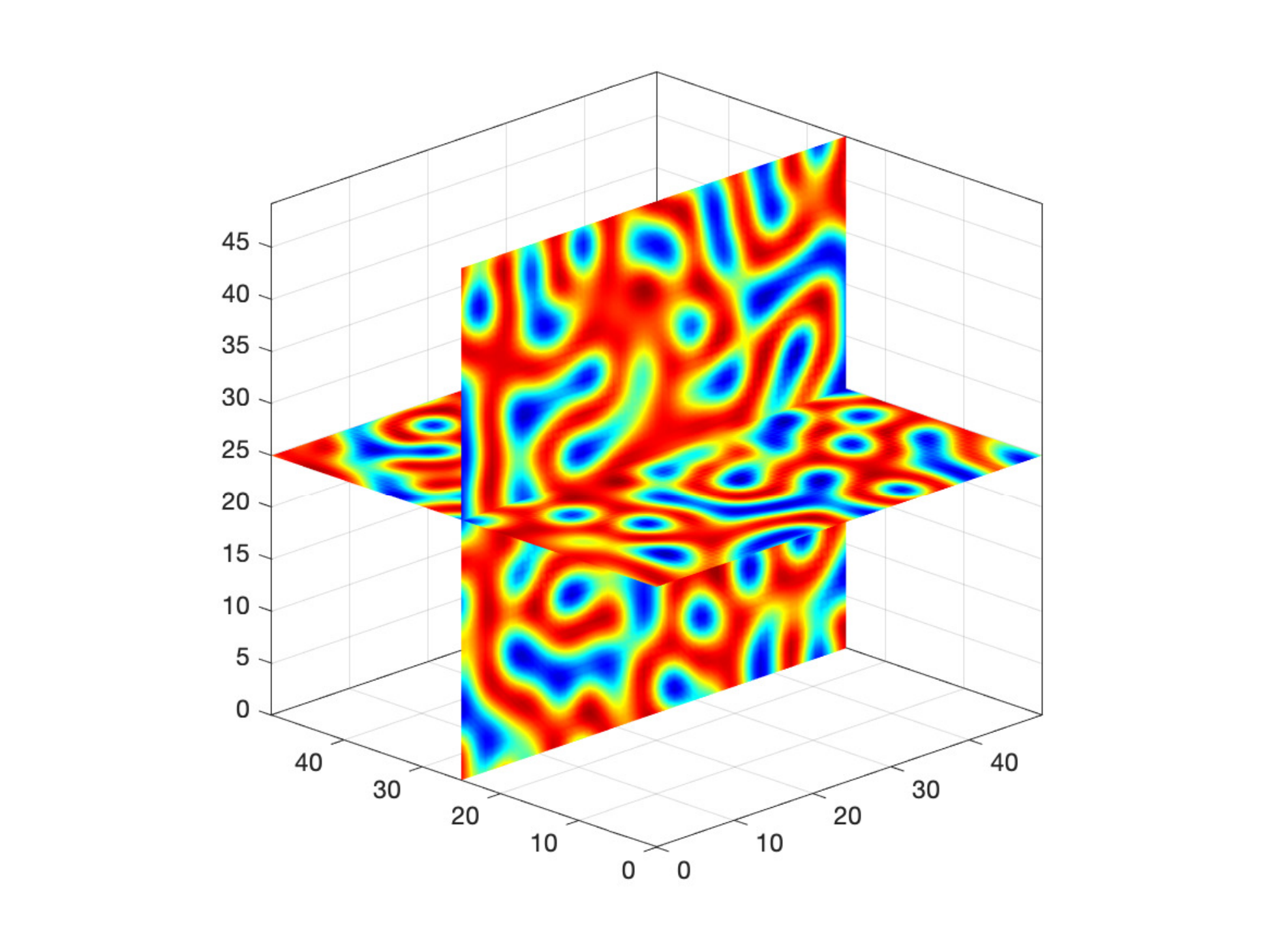}\hspace{-9mm}
	\includegraphics[width=5.3cm]{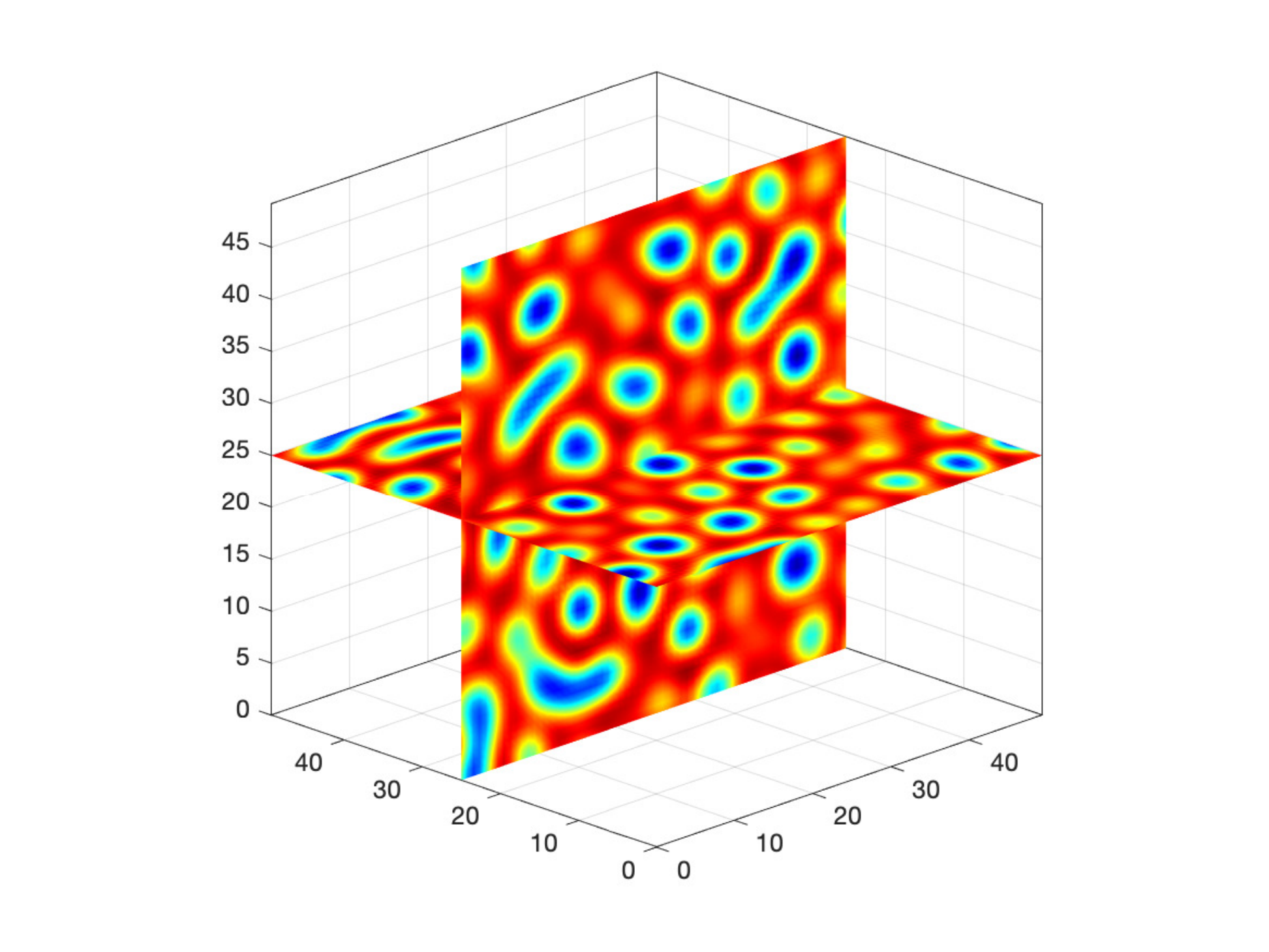}\hspace{-9mm}
	\includegraphics[width=5.3cm]{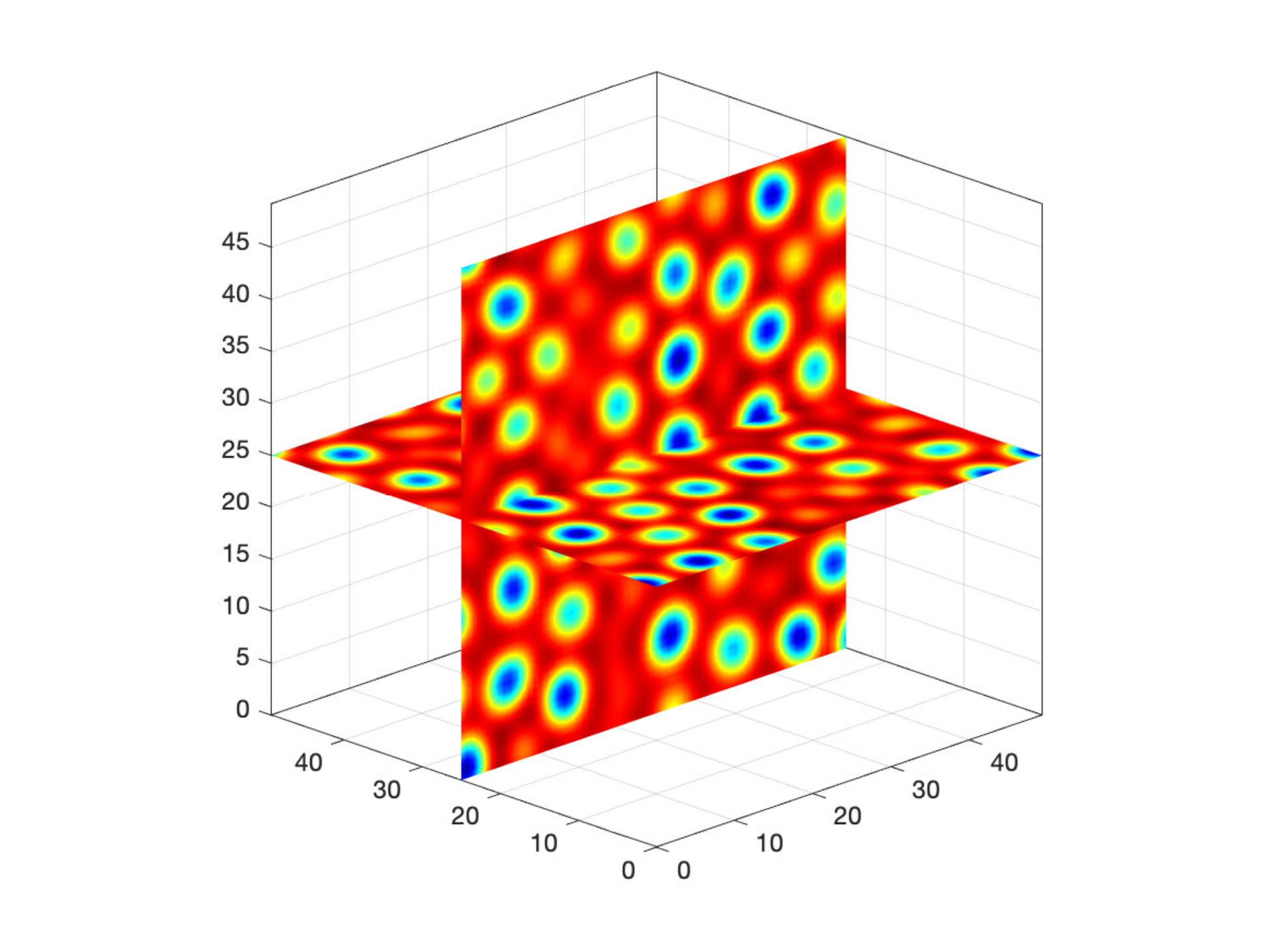}
\label{Fig:PFC-3D-rand-ESAV1}
\caption{Example 3 (\romannumeral4). Evolution of $\phi$ in $3$D driven by the PFC equation obtained by R-ESAV-1/BDF$2$ scheme with different $\bar \phi$ at $T=3000$. First: $\bar{\phi}=0.2$; Second: $\bar{\phi}=0.35$; Third: $\bar{\phi}=0.43$.}
\end{figure}

 \textbf{Example 4.} 
 In this example, we simulate the phase-field vesicle membrane (PFVM) model \cite{cheng2018multiple, cheng2020global} to demonstrate the effectiveness of the relaxed MESAV schemes. We consider the following penalized free energy to preserve the  area and volume of vesicle membrane, 
\begin{equation}
	E(\phi)=E_{b}(\phi)+\frac{1}{2 \sigma_{1}}(V(\phi)-v_0)^{2}+\frac{1}{2 \sigma_{2}}(S(\phi)-s_0)^{2},
\end{equation}
where $\sigma_{1}, \sigma_{2}$ are two small parameters, and $v_0, s_0$ are the initial volume and initial surface area, respectively. The definition of  bending energy $E_b(\phi)$, volume $V(\phi)$ and surface area $S(\phi)$ of the vesicle are as follows
\begin{equation}
	E_{b}(\phi)=\frac{\epsilon}{2} \int_{\Omega}\left(-\Delta \phi+\frac{1}{\epsilon^{2}} G(\phi)\right)^{2} \mathrm{d} \boldsymbol{x}=\frac{\epsilon}{2} \int_{\Omega} H^{2} \mathrm{d} \boldsymbol{x},
\end{equation}
\begin{equation}
	V(\phi)=\int_{\Omega}(\phi+1) \mathrm{d} \boldsymbol{x} \quad \text { and } \quad S(\phi)=\int_{\Omega}\left(\frac{\epsilon}{2}|\nabla \phi|^{2}+\frac{1}{\epsilon} F(\phi)\right) \mathrm{d} \boldsymbol{x},
\end{equation}
where
$$
H:=-\Delta \phi+\frac{1}{\epsilon^{2}} G(\phi), \quad F(\phi)=\frac{1}{4}\left(\phi^{2}-1\right)^{2}, \quad G(\phi):=F^{\prime}(\phi).
$$
Then,  the dynamic equation based on the above total energy can be described by 
\begin{equation}\label{eq:PFVM-model}
	\left\{
	\begin{array}{l}\phi_{t}=-M \mu, \\ 
		\mu=-\epsilon \Delta H+\frac{1}{\epsilon} G^{\prime}(\phi) H+\frac{1}{\sigma_{1}}(V(\phi)-v_0)+\frac{1}{\sigma_{2}}(S(\phi)-s_0)\left(-\epsilon \Delta \phi+\frac{1}{\epsilon} F^{\prime}(\phi)\right), \\ 
		H=-\Delta \phi+\frac{1}{\epsilon^{2}} G(\phi),
	\end{array} \right.
\end{equation}
with the periodic boundary condition, 
and $M$ is the mobility constant. 
Then, it can easily obtain that the system \eqref{eq:PFVM-model} satisfy the energy law as follows
\begin{equation}
	\frac{\mathrm{d}}{\mathrm{d} t} E(\phi)=-M\|\mu\|^{2}.
\end{equation}
It can be observed that the system \eqref{eq:PFVM-model} contains two nonlinear terms associated with two small parameters $\epsilon$ and $\sigma_2$ respectively. 
Therefore, two SAVs are needed to introduce to deal with the different nonlinear terms. 
In the following simulations, we set computational domain as $\Omega=(-\pi, \pi)^{3}$, and the model parameters are $\sigma_{1}=\sigma_{2}=0.01, \epsilon=\frac{6\pi}{128}, M=1$. 
We compute the results using  the  R-MESAV-1/BDF$2$ scheme with time step $\delta t=1e-4$ and $N^3=128^3$  Fourier modes.

(\romannumeral1)
We first consider the interaction of 
 four closeby spheres as the initial condition given by
\begin{equation}
\phi(x, y, z, 0)=\sum_{j=1}^{4} \tanh \left(\frac{R_{j}-\sqrt{\left(x-x_{j}\right)^{2}+\left(y-y_{j}\right)^{2}+\left(z-z_{j}\right)^{2}}}{\sqrt{2} \epsilon}\right)+3,
\end{equation}
where $R_{j}=\frac{\pi}{6}, x_{j}=0,\left(y_{1}, y_{2}, y_{3}, y_{4}\right)=\left(-\frac{\pi}{4},\frac{\pi}{4},-\frac{3 \pi}{4}, \frac{3 \pi}{4}\right),$ and $z_{j}=0$ for $j=1, 2, 3, 4$. 

Snapshots of iso-surfaces of $\phi=0$ at $t=0, 0.02, 1$ are presented in Fig.\,\ref{Fig:PFVM-RESAV1-four-balls}, which indicates that  four small spheres gradually linked the shape of `ice sugar gourd', and finally merge into a cylinder shape. 
The results are consistent with those presented in \cite{cheng2020global}. 
We also plot the evolution of relaxation factor $\theta_{0}^{n+1}$ in Fig.\,\ref{Fig:PFVM-RESAV1-four-balls-zeta}, and observe that, except the first several steps, $\theta_{0}^{n+1}$ always takes the value zeros. 

\begin{figure}[htbp]
\centering
	\includegraphics[width=5.3cm]{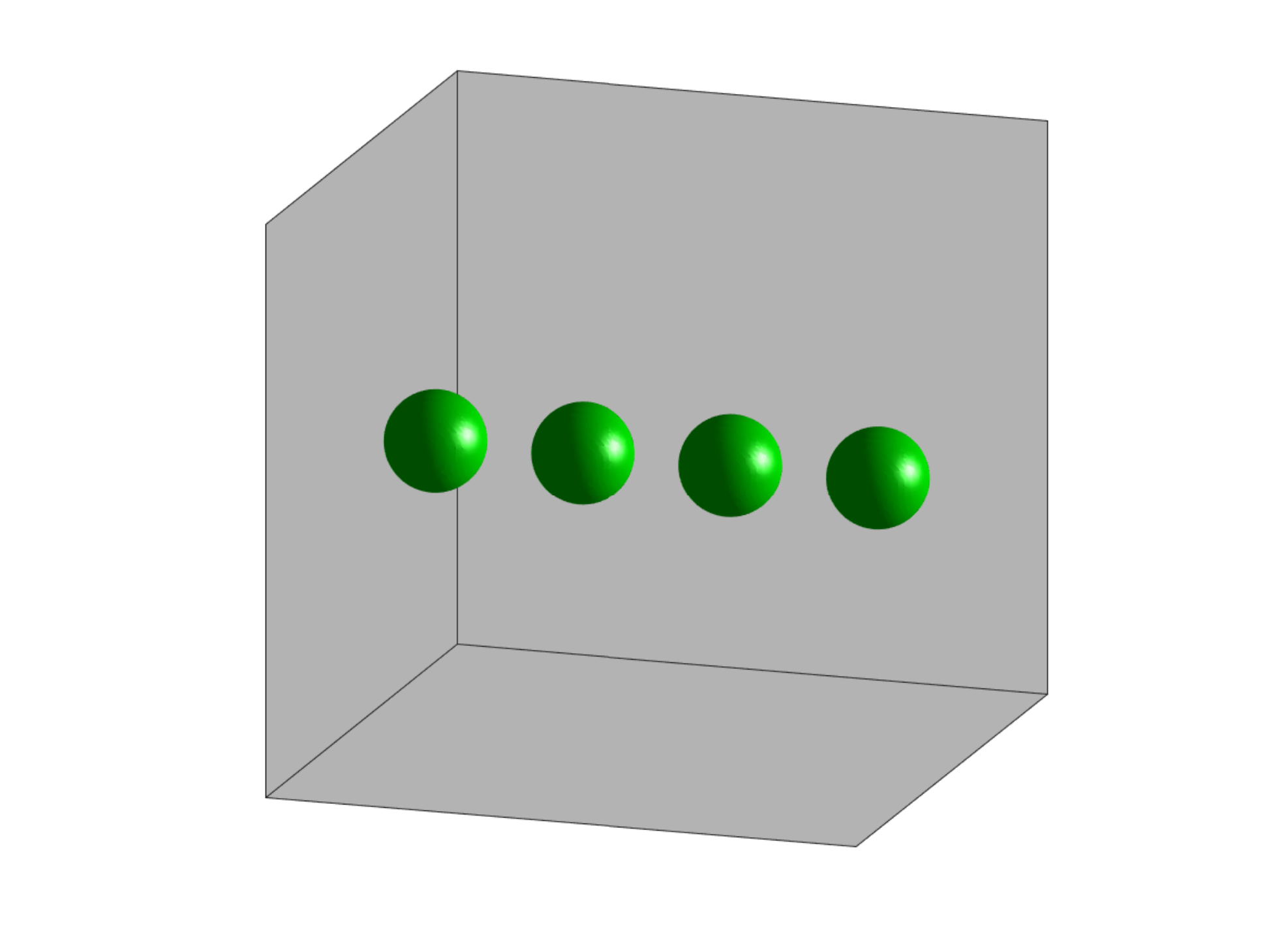}\hspace{-9mm}
	\includegraphics[width=5.3cm]{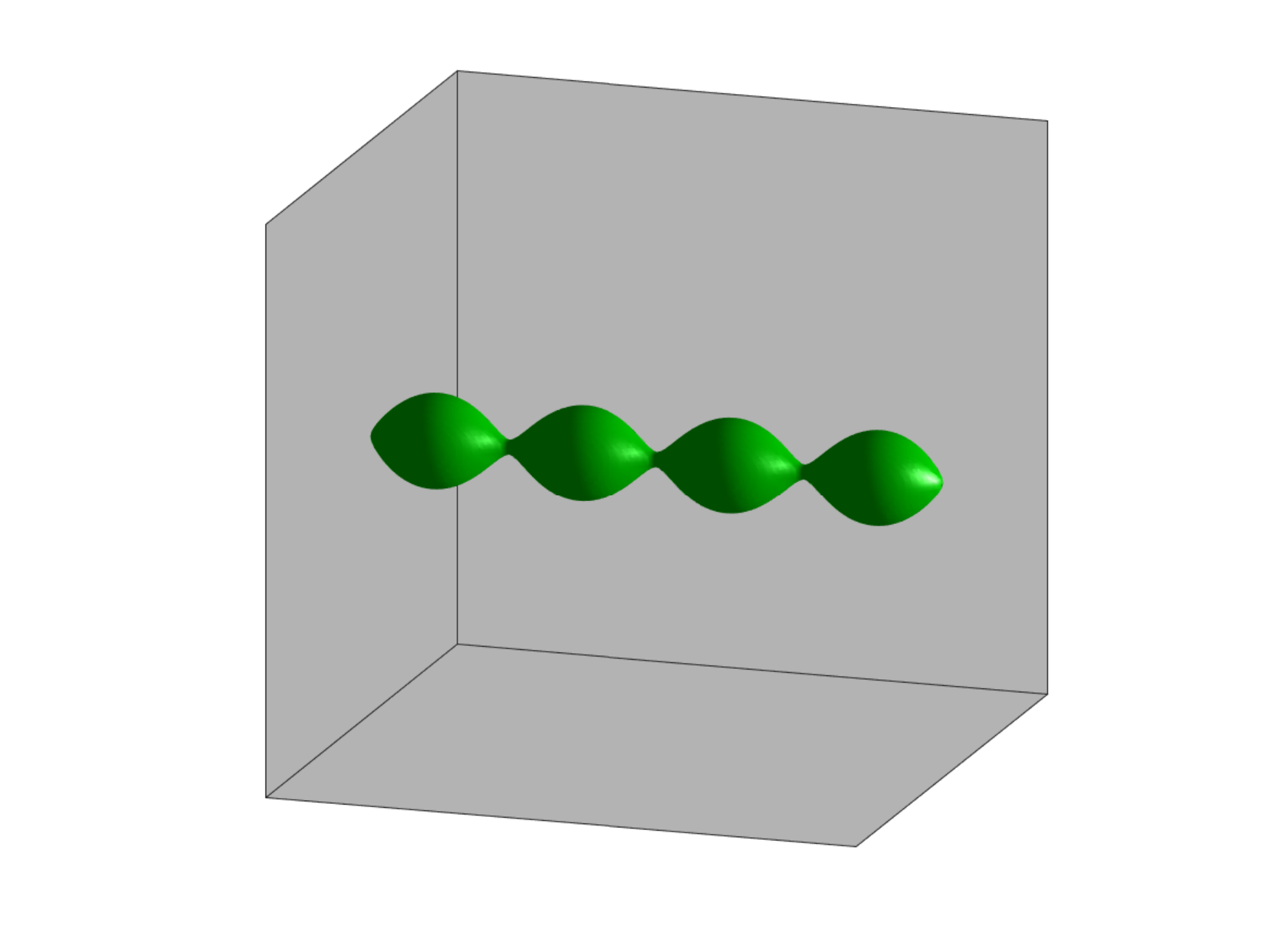}\hspace{-9mm}
	\includegraphics[width=5.3cm]{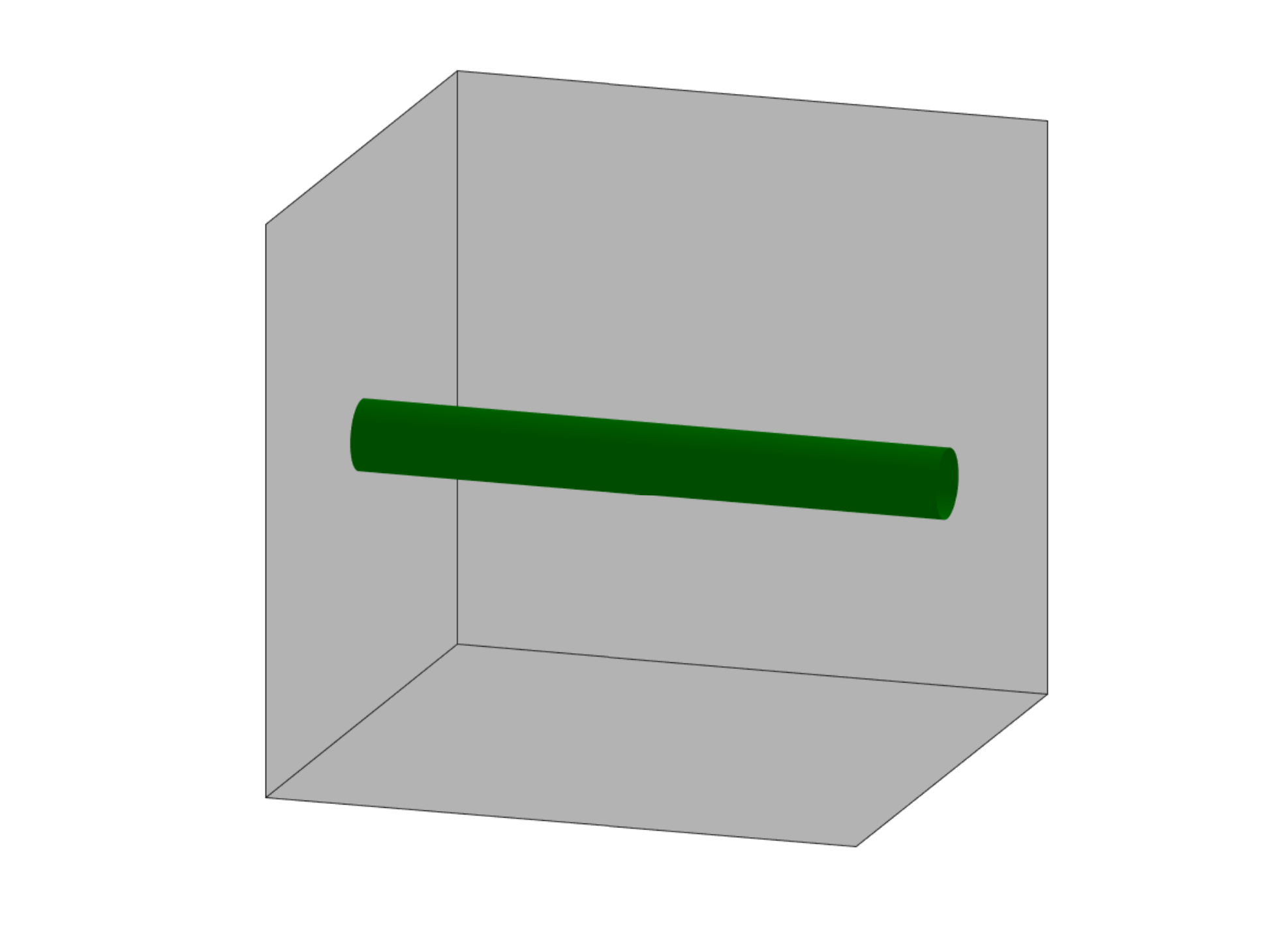}
\caption{Example 4 (\romannumeral1). The evolution of four close-by spherical vesicles. Snapshots of iso-surfaces of $\phi=0$ driven by the PFVM equation at $t=0, 0.02, 1$.}
\label{Fig:PFVM-RESAV1-four-balls}
\end{figure}

\begin{figure}[htbp]
\centering
	\includegraphics[width=5.3cm]{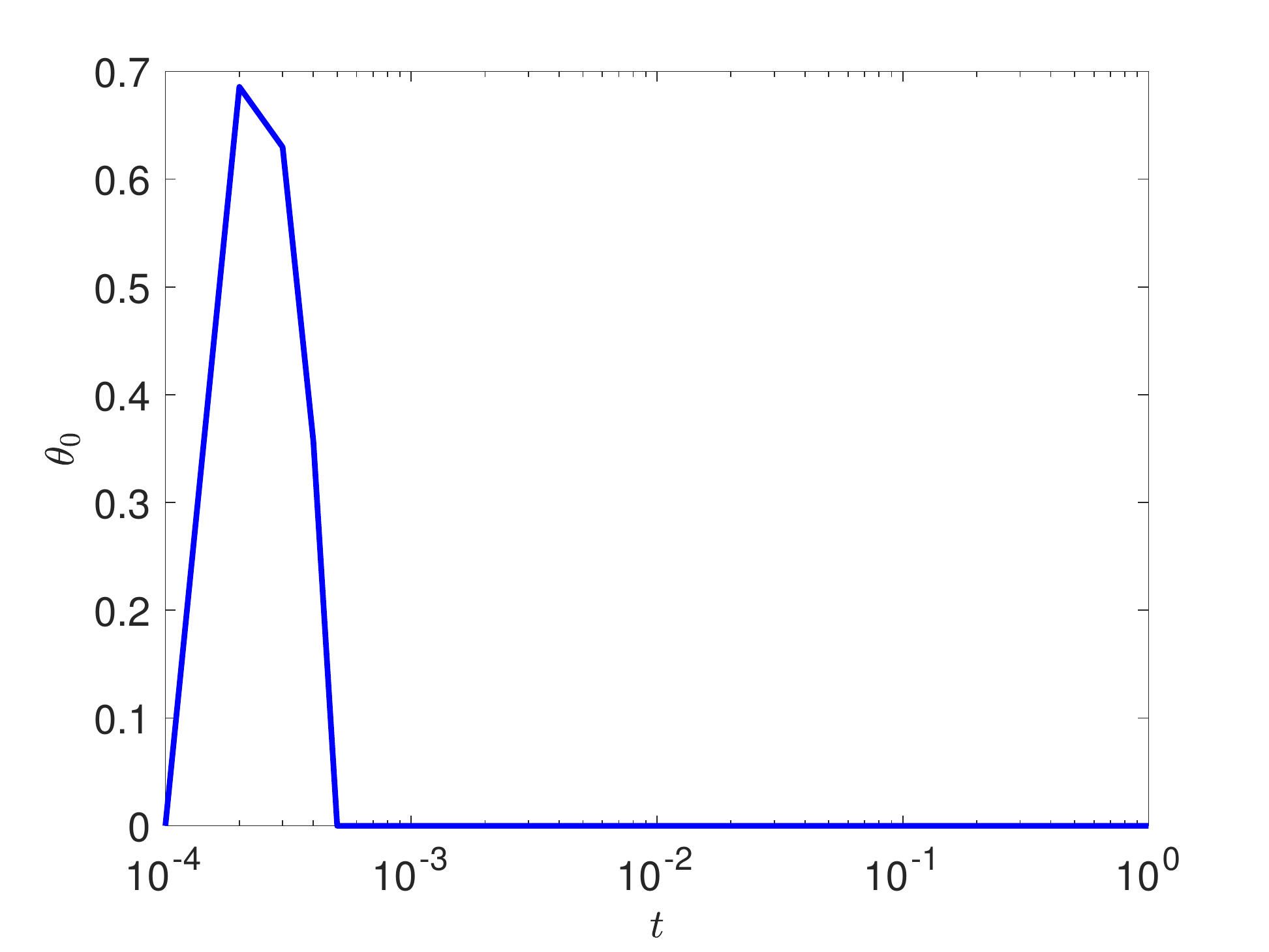}
\caption{Example 4 (\romannumeral1). The evolution of relaxation factor $\theta_{0}^{n+1}$.}
\label{Fig:PFVM-RESAV1-four-balls-zeta}
\end{figure}

(\romannumeral2) 
We simulate the evolution of five close-by spherical vesicles by choosing the initial condition 
\begin{equation}
\phi(x, y, z, 0)=\sum_{j=1}^{5} \tanh \left(\frac{R_{j}-\sqrt{\left(x-x_{j}\right)^{2}+\left(y-y_{j}\right)^{2}+\left(z-z_{j}\right)^{2}}}{\sqrt{2} \epsilon}\right)+4,
\end{equation}
where $R_{j}=\frac{\pi}{6}, z_{j}=0$ for $j=1,2, \ldots, 5$,  $\left(x_{1}, x_{2}, x_{3}, x_{4}, x_{5}\right)=\left(-\frac{\pi}{3}, \frac{\pi}{3}, 0, -\frac{\pi}{3}, \frac{\pi}{3}\right)$, and \\
$\left(y_{1}, y_{2}, y_{3}, y_{4}, y_{5}\right)=\left(-\frac{\pi}{3}, -\frac{\pi}{3}, 0, \frac{\pi}{3}, \frac{\pi}{3}\right)$. 

We represent the evolution process in Fig.\,\ref{Fig:PFVM-RESAV1-five-balls}. 
It can be observed that five spheres connect within a small time interval, gradually form a doughnut shape which is a final state.   

\begin{figure}[htbp]
\centering
\subfigure[profiles of $\phi=0$ at $T=0, 0.02, 0.05$]{
	\includegraphics[width=5.3cm]{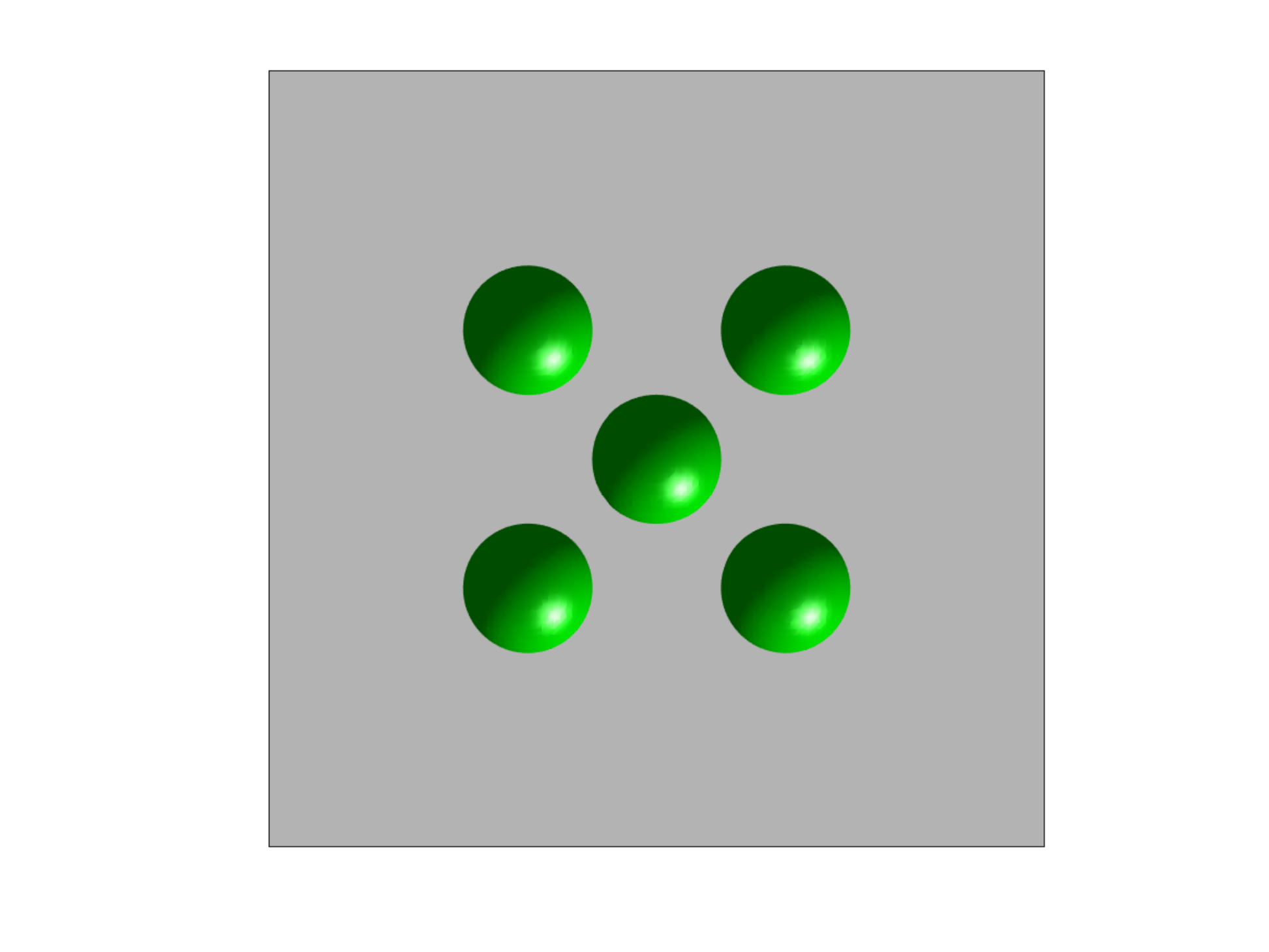}\hspace{-9mm}
	\includegraphics[width=5.3cm]{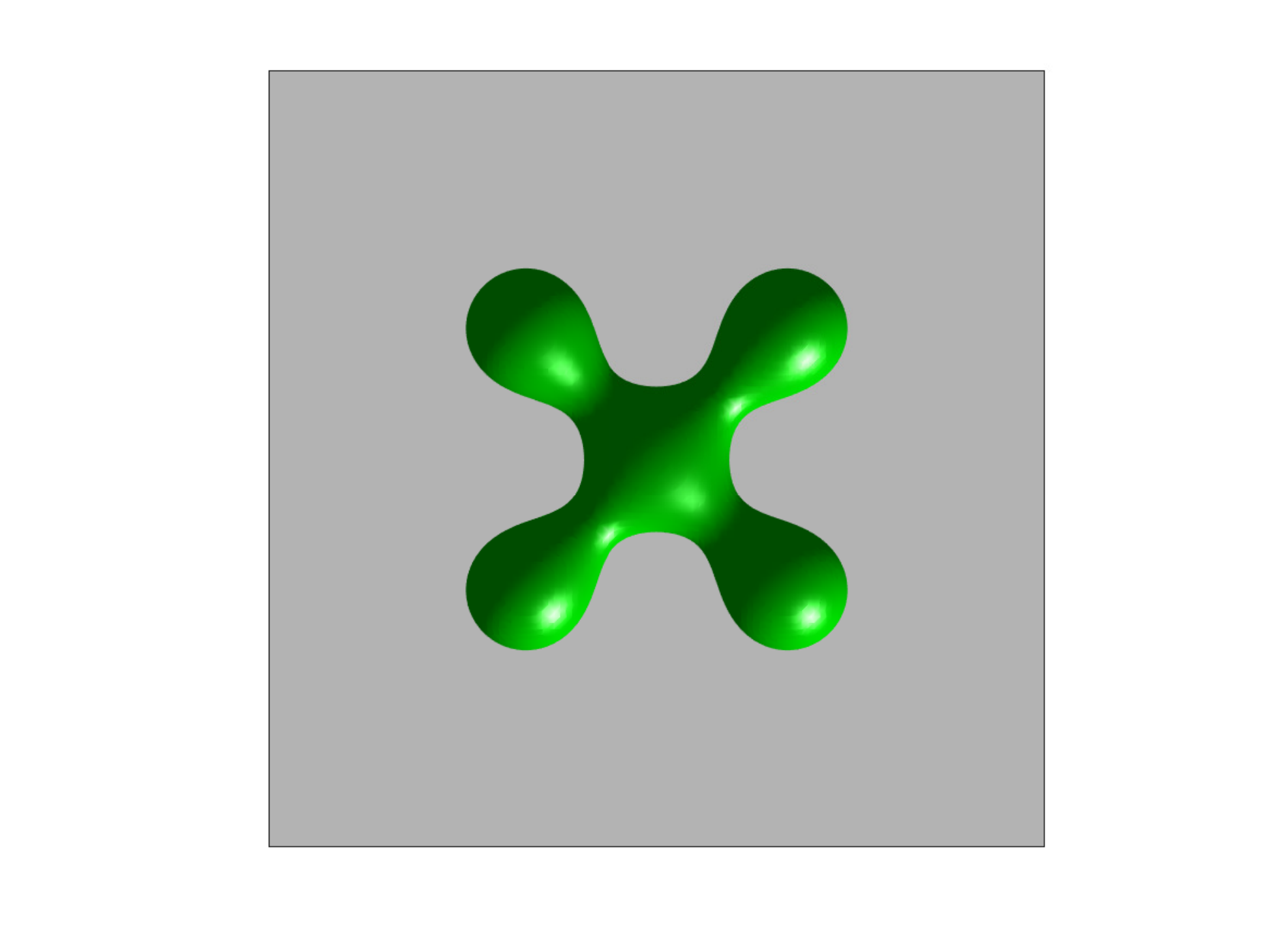}\hspace{-9mm}
	\includegraphics[width=5.3cm]{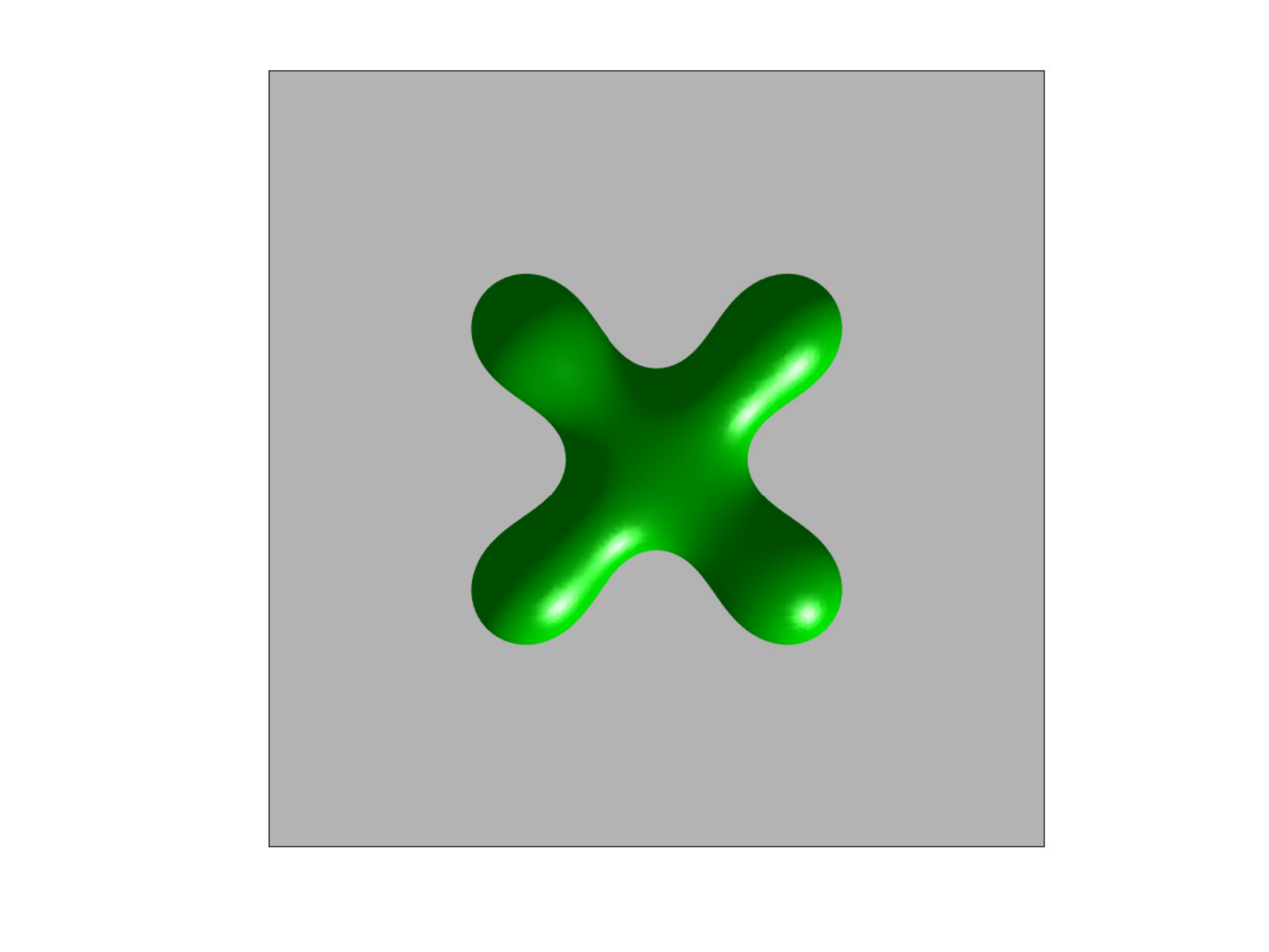}
} 
\subfigure[profiles of $\phi=0$ at $T=0.2, 0.3, 1$]{
	\includegraphics[width=5.3cm]{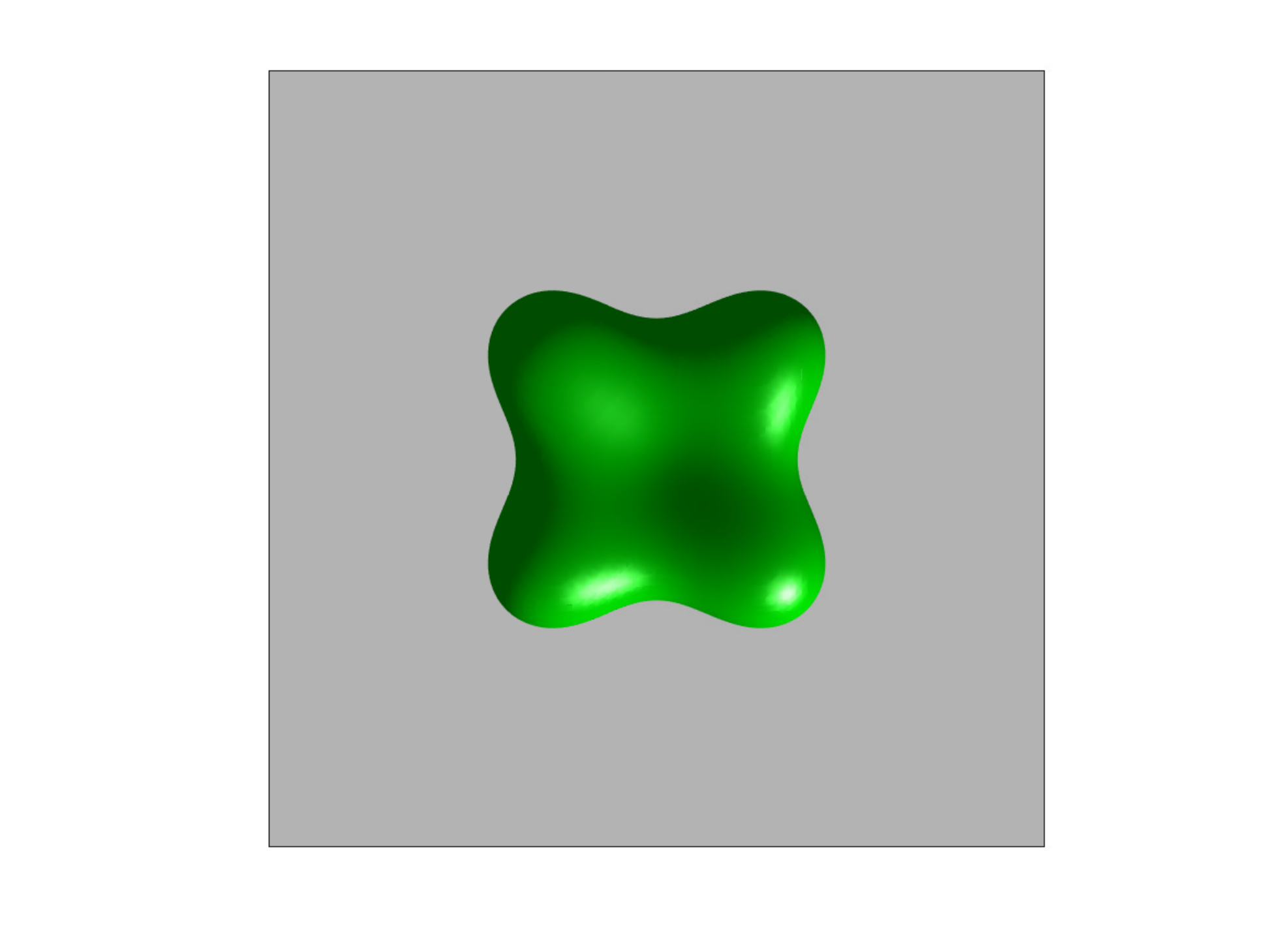}\hspace{-9mm}
	\includegraphics[width=5.3cm]{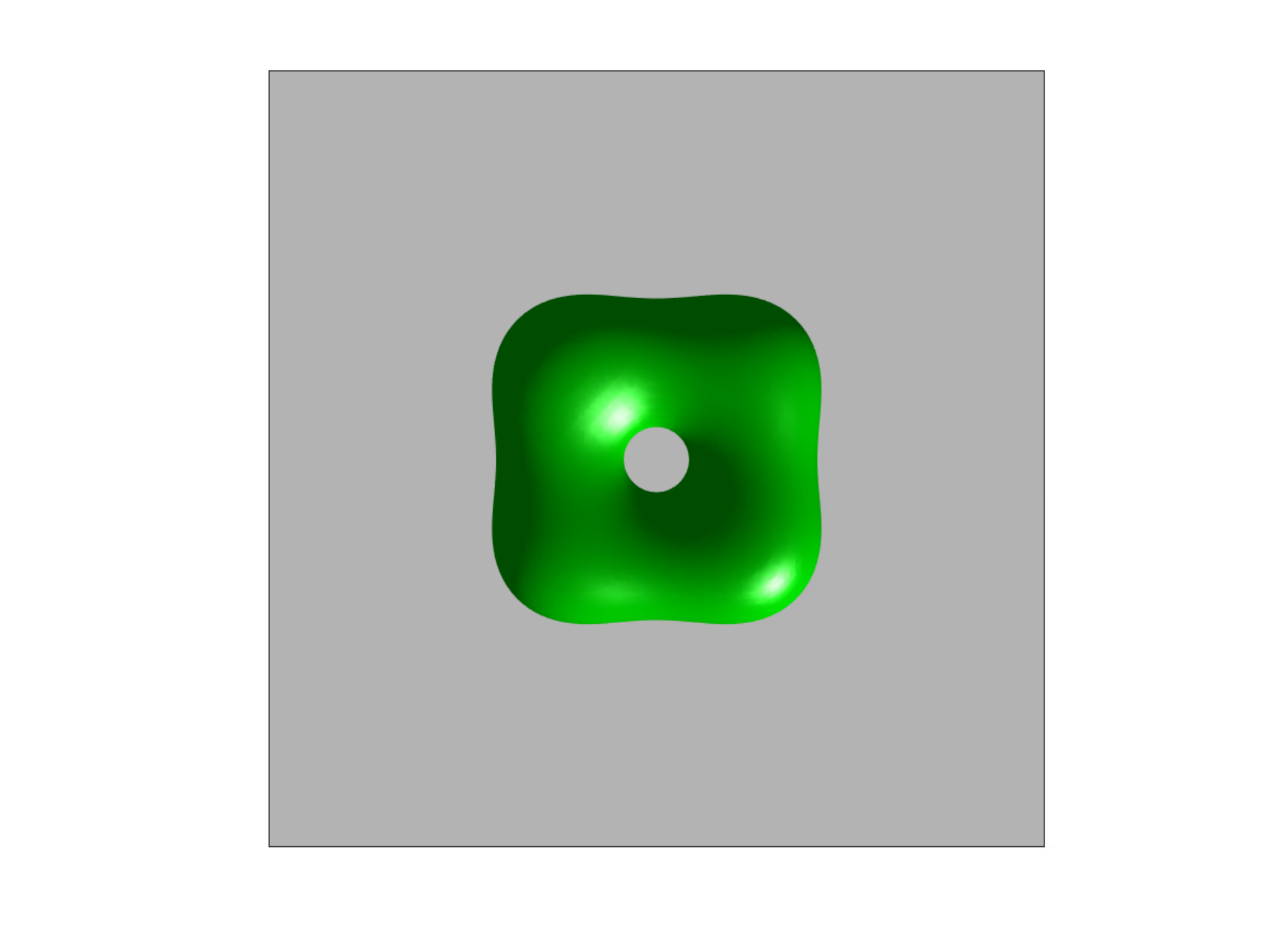}\hspace{-9mm}
	\includegraphics[width=5.3cm]{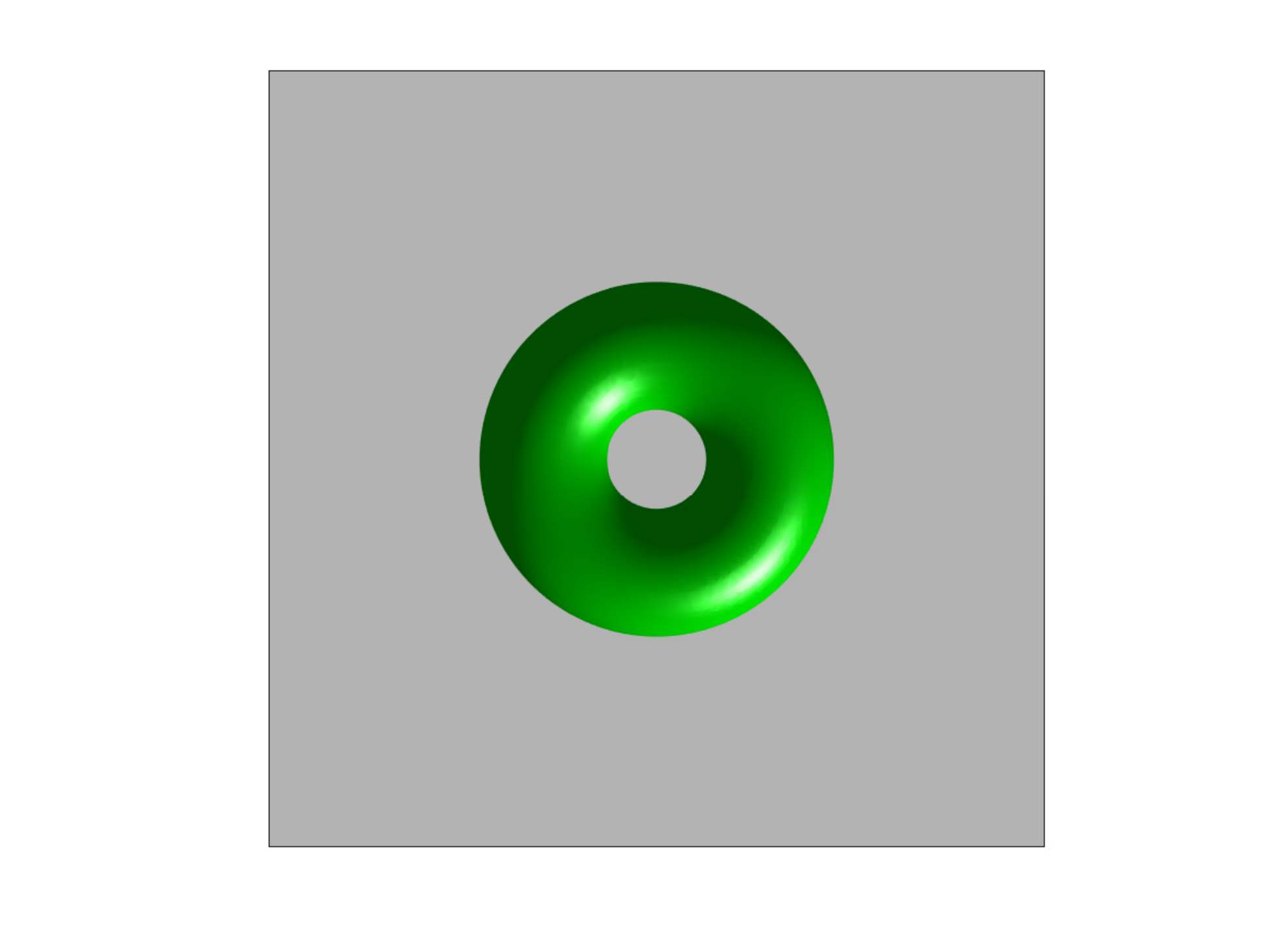}
}
\label{Fig:PFVM-RESAV1-five-balls}
\caption{Example 4 (\romannumeral2). The evolution of five close-by spherical vesicles. Snapshots of iso-surfaces of $\phi=0$ driven by the PFVM equation at $T=0, 0.02, 0.05, 0.2, 0.3, 1$.}
\end{figure}

(\romannumeral3)
Then we consider a more complicated initial condition which is nine close-by spherical vesicles given by 
\begin{equation}
\phi(x, y, z, 0)=\sum_{j=1}^{9} \tanh \left(\frac{R_{j}-\sqrt{\left(x-x_{j}\right)^{2}+\left(y-y_{j}\right)^{2}+\left(z-z_{j}\right)^{2}}}{\sqrt{2} \epsilon}\right)+8,
\end{equation}
where $R_{j}=\frac{\pi}{6}, z_{j}=0$ for $j=1,2, \ldots, 9$,  $\left(x_{1}, x_{2}, x_{3}, x_{4}, x_{5}, x_{6}, x_{7}, x_{8}, x_{9}\right)=\left(-\frac{\pi}{2}, 0, \frac{\pi}{2},  -\frac{\pi}{2}, 0, \frac{\pi}{2}, -\frac{\pi}{2}, 0, \frac{\pi}{2}\right)$, and
$\left(y_{1}, y_{2}, y_{3}, y_{4}, y_{5}, y_{6}, y_{7}, y_{8}, y_{9}\right)=\left(-\frac{\pi}{2}, -\frac{\pi}{2}, -\frac{\pi}{2}, 0, 0, 0, \frac{\pi}{2}, \frac{\pi}{2}, \frac{\pi}{2}\right)$. 

The evolutions of nine close-by spherical vesicles are demonstrated in Fig.\,\ref{Fig:PFVM-RESAV1-nine-balls}, which represents that the initially nine spheres gradually connect with each other and finally form a big vesicle. 

\begin{figure}[htbp]
\centering
\subfigure[profiles of $\phi=0$ at $T=0, 0.02, 0.05$]{
	\includegraphics[width=5.3cm]{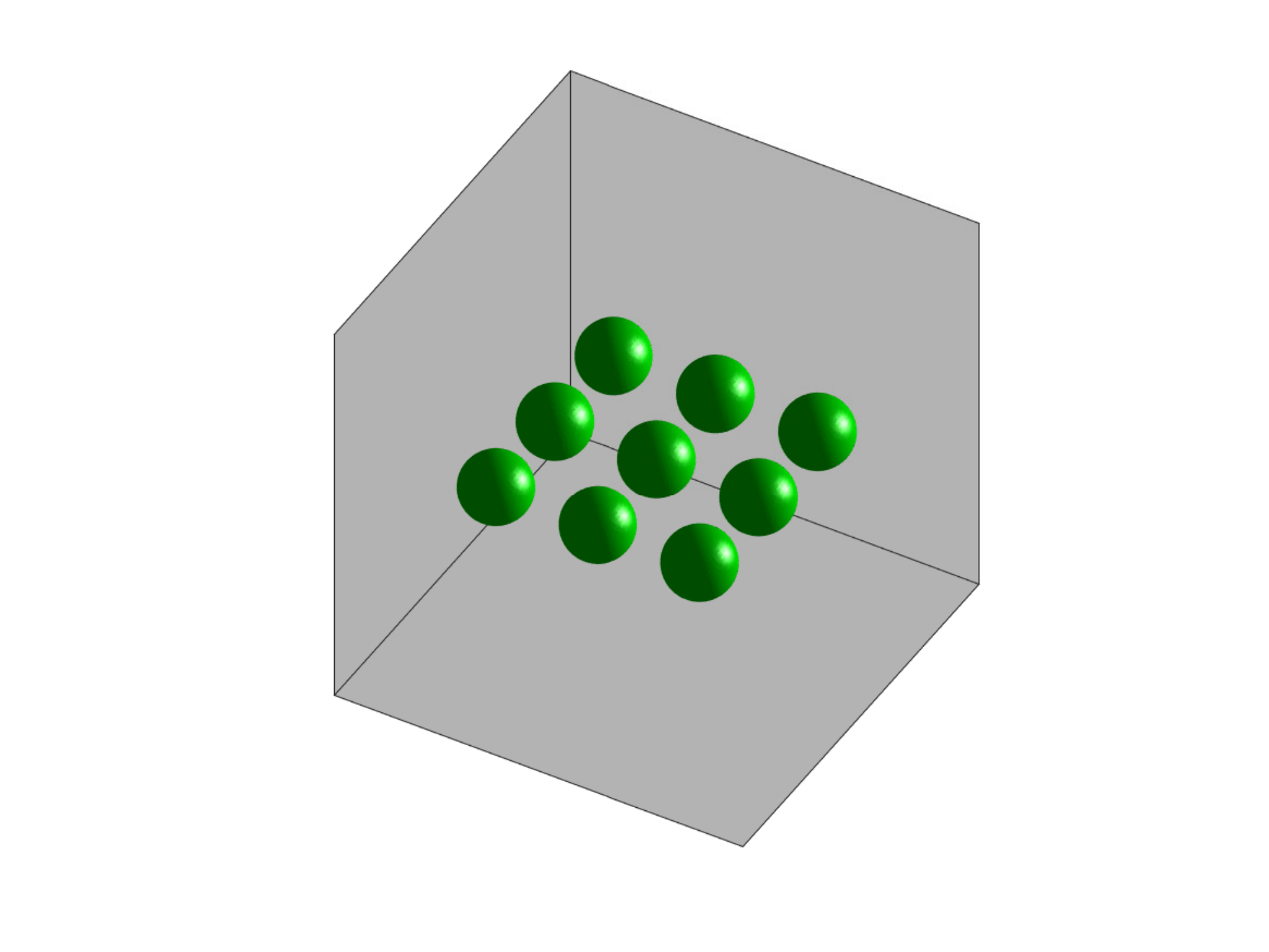}\hspace{-9mm}
	\includegraphics[width=5.3cm]{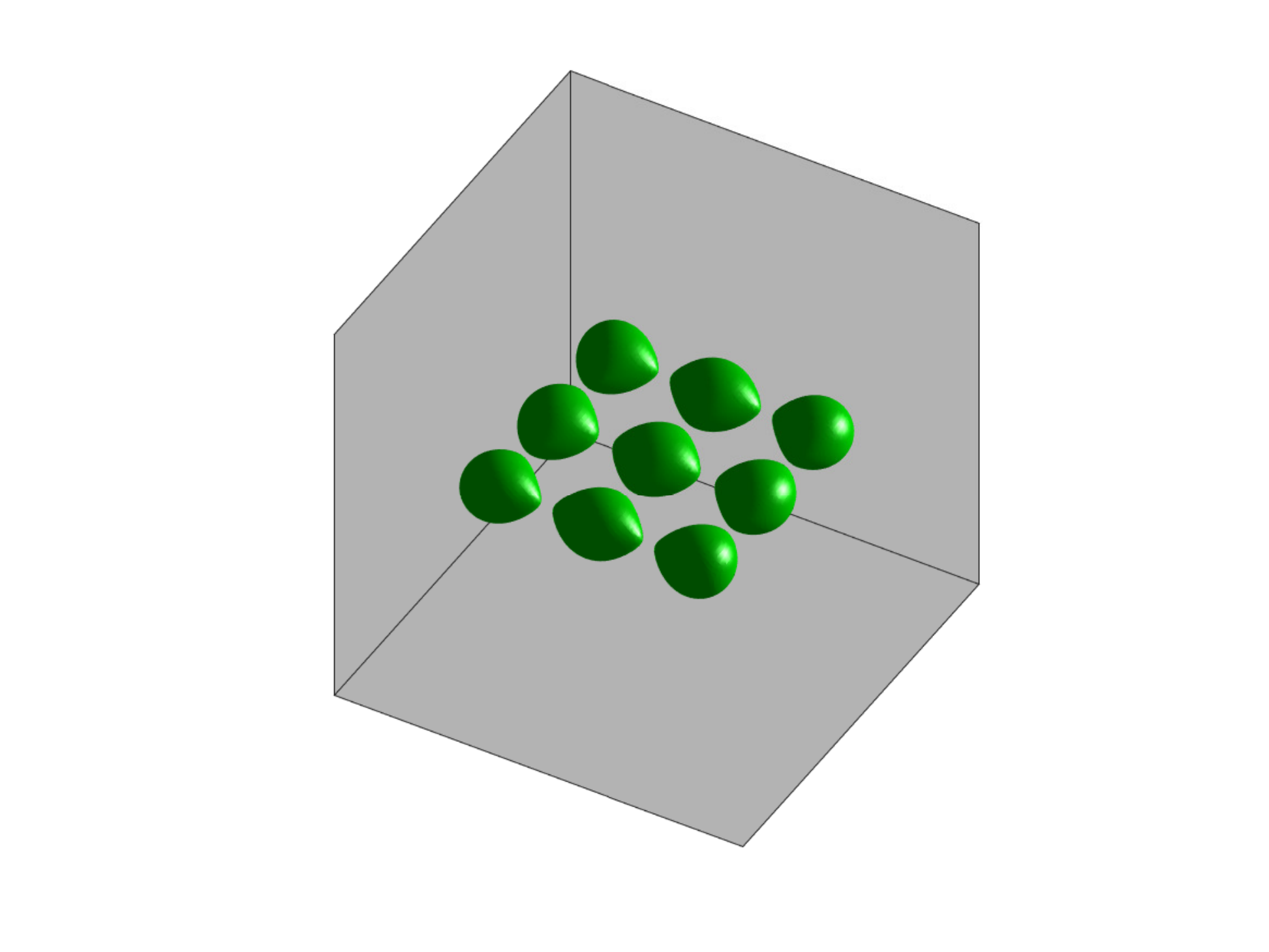}\hspace{-9mm}
	\includegraphics[width=5.3cm]{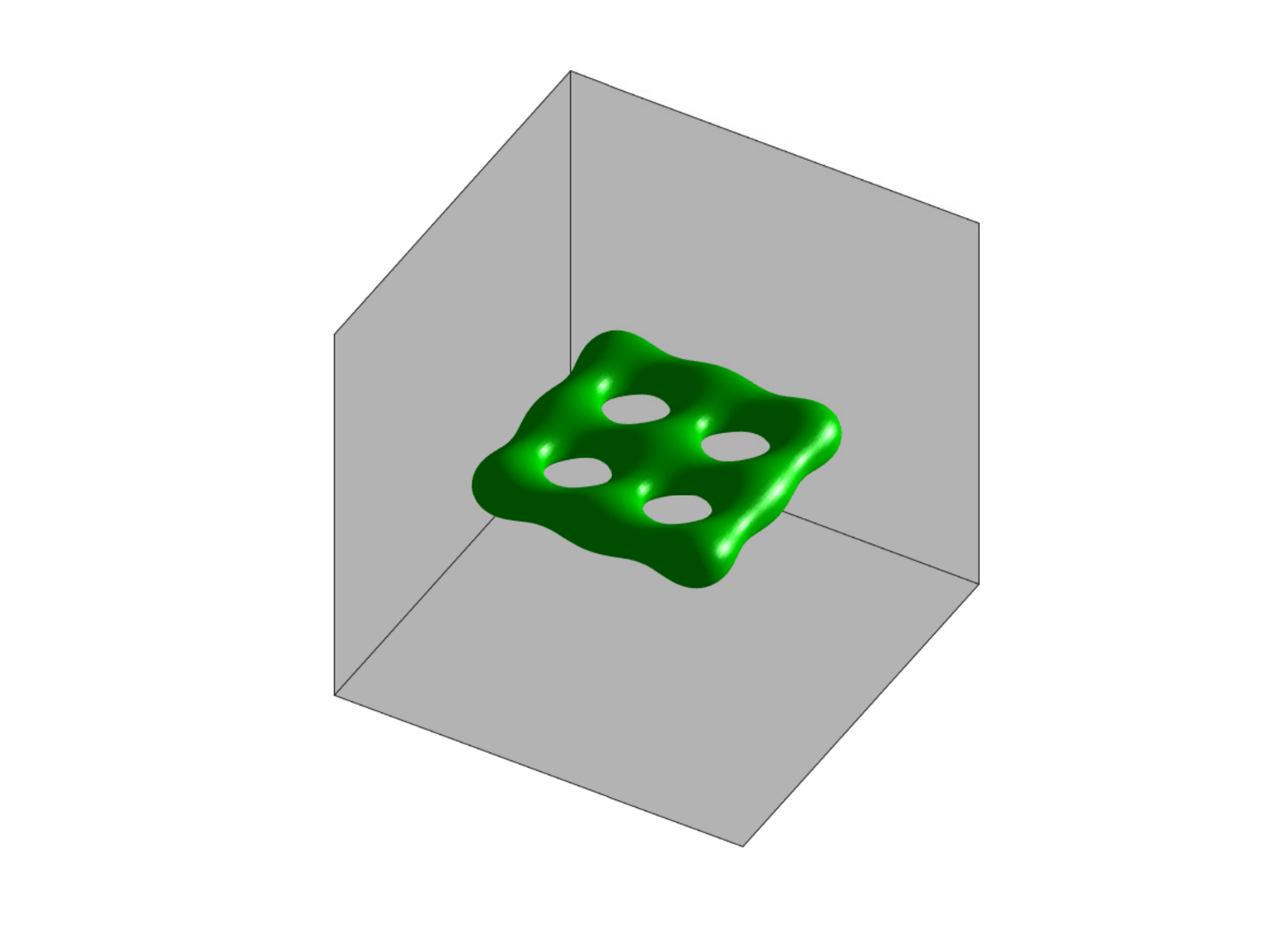}
} 
\subfigure[profiles of $\phi=0$ at $T=0.1, 0.2, 1$]{
	\includegraphics[width=5.3cm]{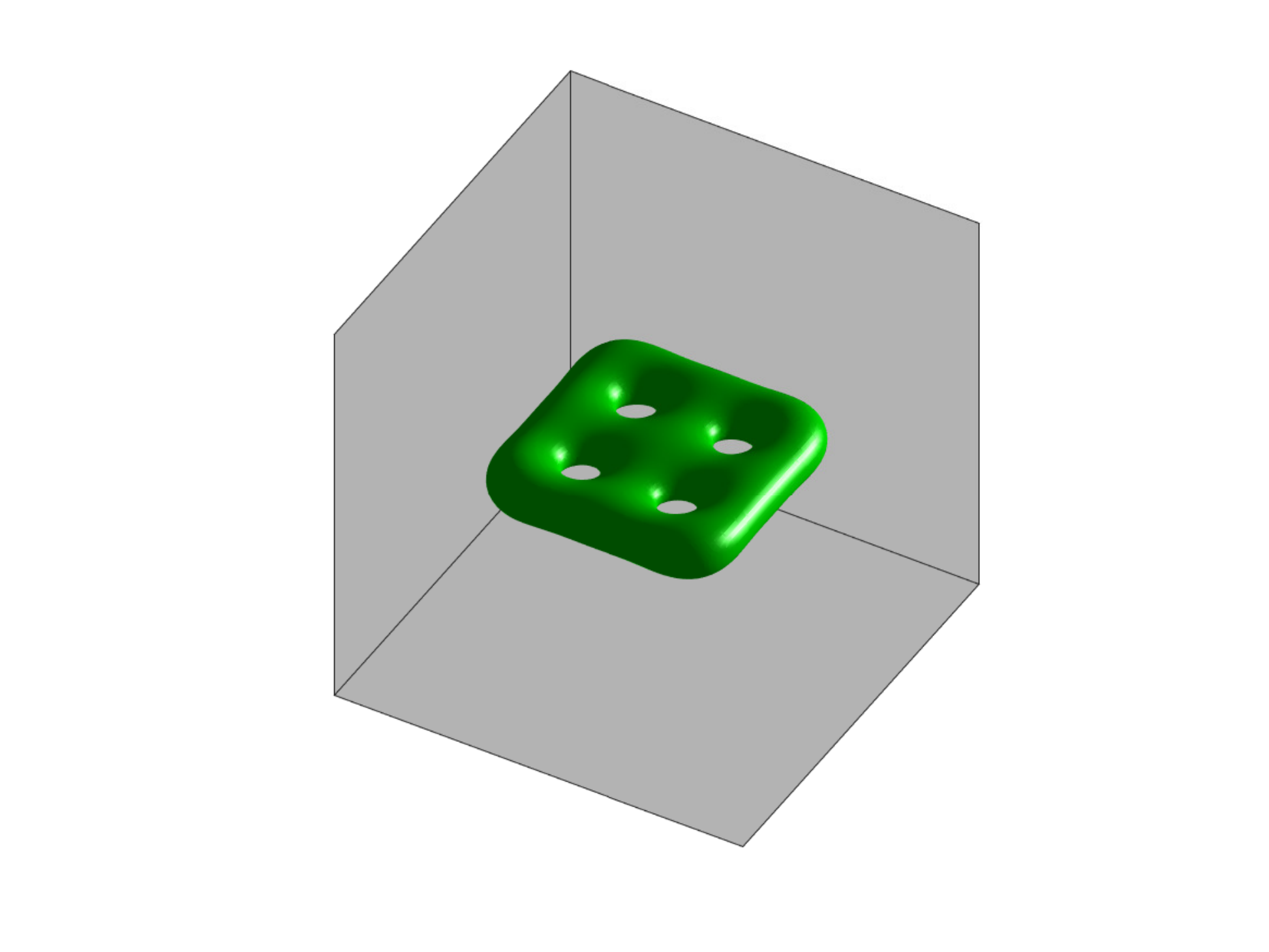}\hspace{-9mm}
	\includegraphics[width=5.3cm]{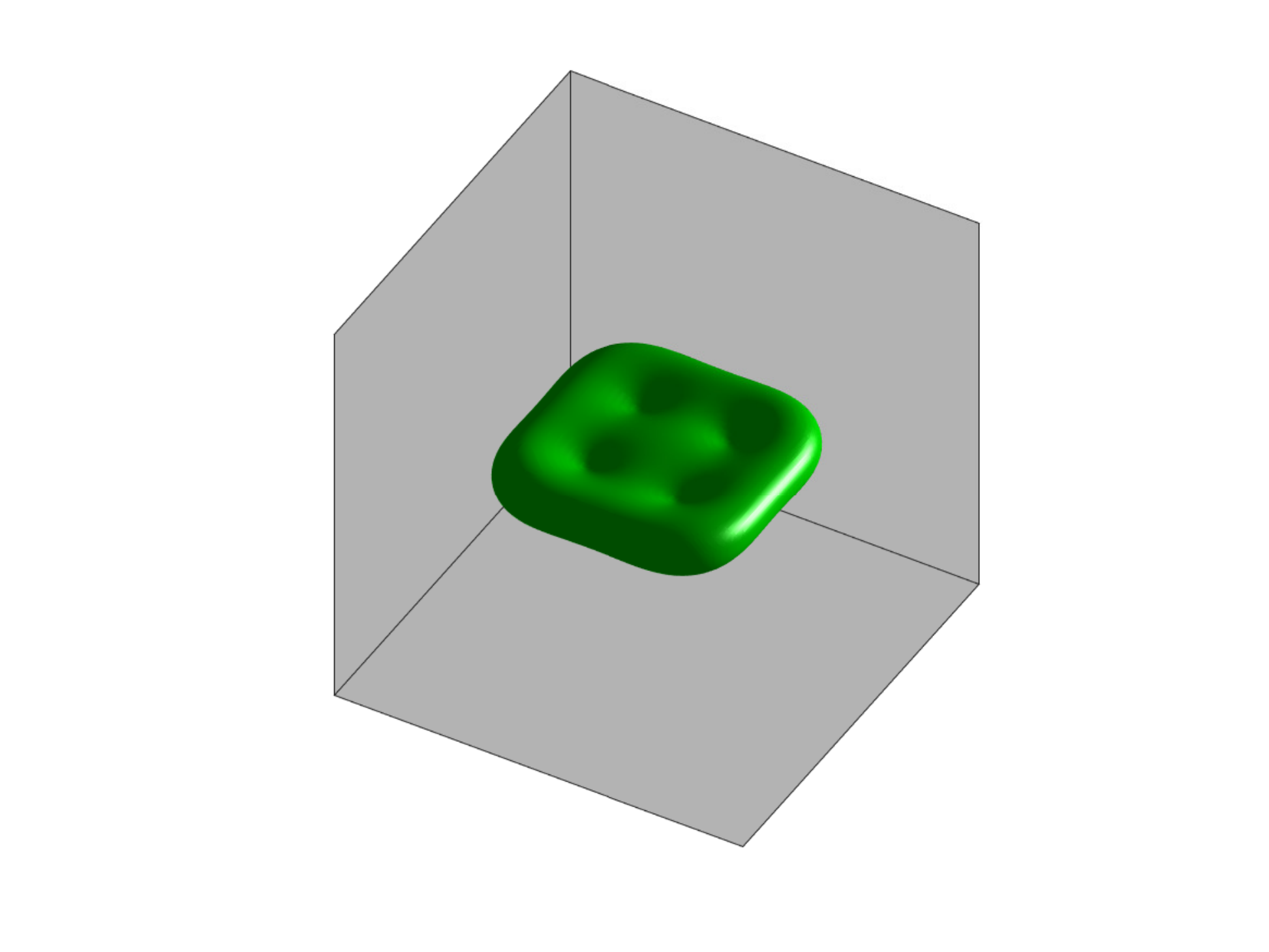}\hspace{-9mm}
	\includegraphics[width=5.3cm]{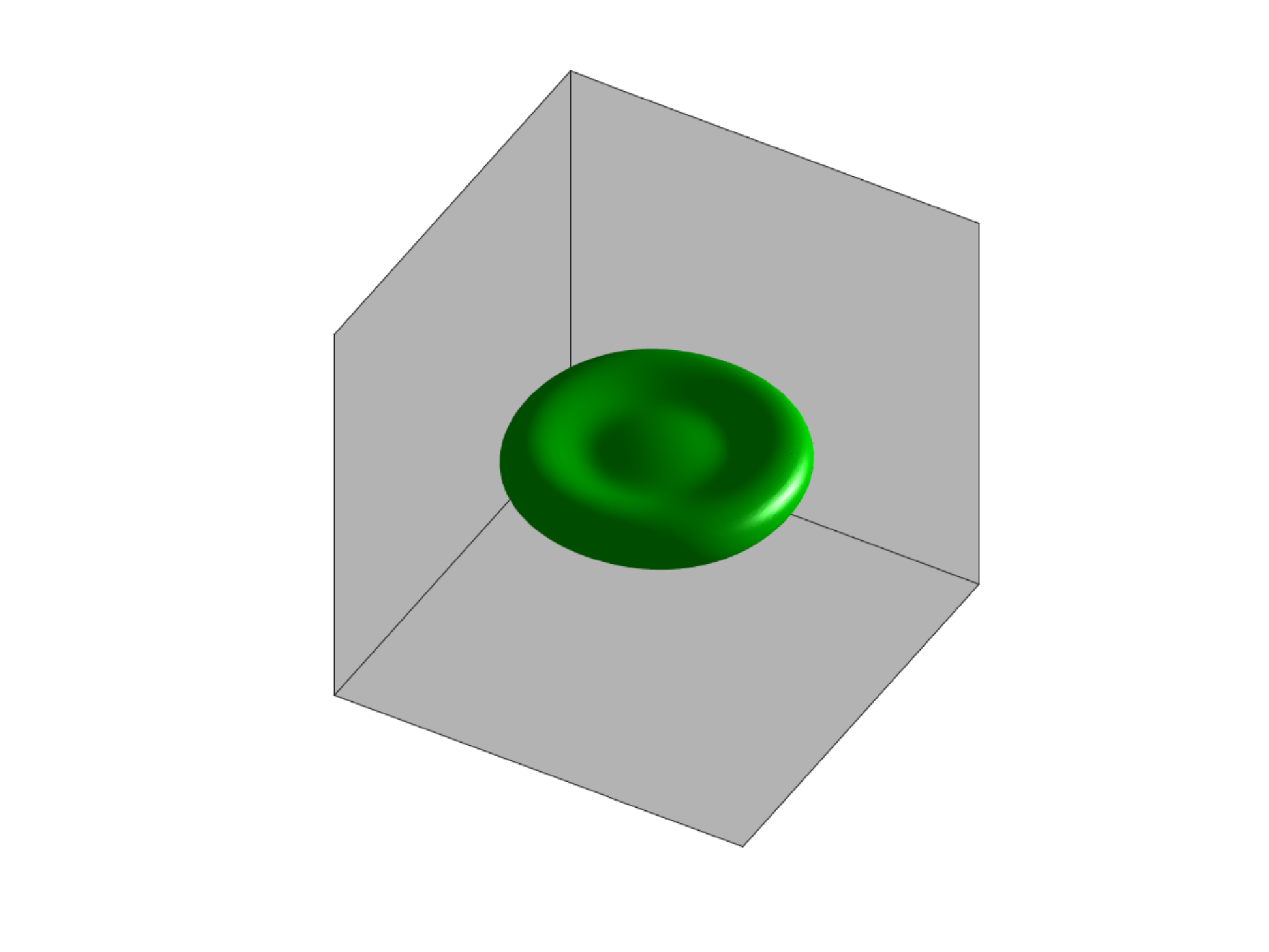}
}
\label{Fig:PFVM-RESAV1-nine-balls}
\caption{Example 4 (\romannumeral3). The evolution of nine close-by spherical vesicles. Snapshots of iso-surfaces of $\phi=0$ driven by the PFVM equation at $T=0, 0.02, 0.05, 0.1, 0.2, 1$.}
\end{figure}
 
\textbf{Example 5.} In this numerical simulation, we test the Navier-Stokes equation. 

(\romannumeral1) 
We start with the accuracy test. The right hand side is computed according to the following 
analytic solution
\begin{equation}
\left\{\begin{array}{l}
p(x, y, t)=\exp (t)\sin (\pi y), \\ 
u_{1}(x, y, t)=\exp (t) \sin ^{2}(\pi x) \sin (2 \pi y), \\ 
u_{2}(x, y, t)=-\exp (t) \sin (2 \pi x) \sin ^{2}(\pi y),
\end{array}\right.
\end{equation}
the computational domain is $\Omega=(-1, 1)^{2}$, other parameters are chosen as $\nu=0.1, T=1$. 
We use Legendre spectral method to discretize space and $N^2=64^2$. 
The numerical results for BDF$1$ and BDF$2$ of Schemes \uppercase\expandafter{\romannumeral1} and  \uppercase\expandafter{\romannumeral2} are presented in Tables \ref{table:NS-scheme1-1st}-\ref{table:NS-scheme2-2nd} respectively.
It can be seen that the errors of the solution will be greatly reduced by using the relaxation factor. 

(\romannumeral2)  Next we simulate double shear layer problem. 
We consider the Navier-Stokes equation with periodic boundary condition, and the initial condition provided by 
\begin{equation}
	\begin{array}{l}u_{1}(x, y, 0)=\left\{\begin{array}{l}\tanh (\sigma(y-0.25)), y \leq 0.5, \\ 
	\tanh (\sigma(0.75-y)), y>0.5,
	\end{array}\right. \\ 
	u_{2}(x, y, 0)=\epsilon \sin (2 \pi x),
\end{array}
\end{equation}
where $\sigma$  is the parameter of the shear layer width and $\epsilon$ is the size of the perturbation. 
We choose $\sigma=30, \epsilon=0.05$ and computational domain $\Omega=(0, 2)^2$ in the next simulations. 
We use  Fourier modes $N^2=128^{2}$  and set $\delta t = 6.7e-4$ to test the Navier-Stokes equation with $\nu=5e-5$. 
The vorticity contours of velocity field $\mathbf{u}$ at $T = 1.2$ using BDF$k$ ($k=1, 2, 3, 4$) schemes are presented in Fig.\,\ref{Fig:NS-double-shear-RESAV2-periodic-scheme-comparison}. 
It can be observed that  the BDF$3$ and BDF$4$ schemes give correct solutions, while the BDF$1$ scheme leads to a totally wrong result and the result of the BDF$2$ scheme is also inaccurate. 
The numerical phenomenon demonstrates the  superiorities of high-order schemes. 
The evolution of contours of vorticity with $\nu  = 1e-4$, and $\delta t = 6e- 4$ obtained by BDF$2$ scheme is shown in Fig.\,\ref{Fig:NS-double-shear-RESAV2-periodic-evolution}. 
We can observe from the results that the vortex increases gradually.

\linespread{1.2}
\begin{table}[!h] 
	\centering
	\caption{Example 5 (\romannumeral1). Convergence test for Navier-Stokes equation using Scheme \uppercase\expandafter{\romannumeral1}/BDF$1$ with and without relaxation.}
	 \label{table:NS-scheme1-1st}
	\begin{tabular}{||c||cccc||cccc||}
		\hline
	& & ESAV-2 & & & & R-ESAV-2 & & \\
		\hline
	$\delta t$ & $\|e_{\mathbf{u}}\|_{L^{2}}$ & Rate & $\|e_{p}\|_{L^{2}}$ & Rate & $\|e_{\mathbf{u}}\|_{L^{2}}$ & Rate & $\|e_{p}\|_{L^{2}}$ & Rate\\
		\hline
	2.50E-2  & 4.29E-02 &  --    & 1.40       & --    & 1.08E-02 &  --    & 2.62E-01   & -- \\
	1.25E-2  & 1.96E-02 &  1.13  & 6.46E-01   & 1.11  & 4.85E-03 &  1.15  & 1.15E-01   & 1.19\\
	6.25E-3  & 9.40E-03 &  1.06  & 3.10E-01   & 1.06  & 2.32E-03 &  1.06  & 5.35E-02   & 1.10\\
	3.13E-3  & 4.61E-03 &  1.03  & 1.52E-01   & 1.03  & 1.14E-03 &  1.03  & 2.58E-02   & 1.05\\
	1.56E-3  & 2.28E-03 &  1.01  & 7.52E-02   & 1.01  & 5.63E-04 &  1.01  & 1.26E-02   & 1.03\\
	\hline
	\end{tabular}
\end{table}

\linespread{1.2}
\begin{table}[!h] 
	\centering
	\caption{Example 5 (\romannumeral1). Convergence test for Navier-Stokes equation using Scheme \uppercase\expandafter{\romannumeral1}/BDF$2$ with and without relaxation.}
	 \label{table:NS-scheme1-2nd}
	\begin{tabular}{||c||cccc||cccc||}
		\hline
	& & ESAV-2 & & & & R-ESAV-2 & & \\
		\hline
	$\delta t$ & $\|e_{\mathbf{u}}\|_{L^{2}}$ & Rate & $\|e_{p}\|_{L^{2}}$ & Rate & $\|e_{\mathbf{u}}\|_{L^{2}}$ & Rate & $\|e_{p}\|_{L^{2}}$ & Rate\\
		\hline
	2.50E-2  & 1.94E-03 &  --    & 4.47E-02   & --    & 1.54E-03 &  --    & 9.25E-03   & -- \\
	1.25E-2  & 4.64E-04 &  2.06  & 9.96E-03   & 2.17  & 3.88E-04 &  1.99  & 2.77E-03   & 1.74\\
	6.25E-3  & 1.14E-04 &  2.02  & 2.40E-03   & 2.05  & 9.73E-05 &  1.99  & 8.79E-04   & 1.66\\
	3.13E-3  & 2.84E-05 &  2.00  & 6.19E-04   & 1.96  & 2.46E-05 &  1.99  & 3.03E-04   & 1.54\\
	1.56E-3  & 7.28E-06 &  1.97  & 1.70E-04   & 1.86  & 6.37E-06 &  1.95  & 1.07E-04   & 1.50\\
	\hline
	\end{tabular}
\end{table}

\linespread{1.2}
\begin{table}[!h] 
	\centering
	\caption{Example 5 (\romannumeral1). Convergence test for Navier-Stokes equation using Scheme \uppercase\expandafter{\romannumeral2}/BDF$1$ with and without relaxation.}
	 \label{table:NS-scheme2-1st}
	\begin{tabular}{||c||cccc||cccc||}
		\hline
	& & ESAV-2 & & & & R-ESAV-2 & & \\
		\hline
	$\delta t$ & $\|e_{\mathbf{u}}\|_{L^{2}}$ & Rate & $\|e_{p}\|_{L^{2}}$ & Rate & $\|e_{\mathbf{u}}\|_{L^{2}}$ & Rate & $\|e_{p}\|_{L^{2}}$ & Rate\\
		\hline
	2.50E-2  & 4.86E-01 &  --    & 9.34E-01   & --    & 1.99E-01 &  --    & 2.89E-01   & -- \\
	1.25E-2  & 3.12E-01 &  0.64  & 6.03E-01   & 0.63  & 9.51E-02 &  1.07  & 1.34E-01   & 1.11\\
	6.25E-3  & 1.86E-01 &  0.75  & 3.58E-01   & 0.75  & 4.60E-02 &  1.05  & 6.37E-02   & 1.07\\
	3.13E-3  & 1.03E-01 &  0.85  & 1.99E-01   & 0.85  & 2.26E-02 &  1.03  & 3.09E-02   & 1.04\\
	1.56E-3  & 5.47E-02 &  0.91  & 1.05E-01   & 0.92  & 1.15E-02 &  1.02  & 1.52E-02   & 1.02\\
	\hline
	\end{tabular}
\end{table}

\linespread{1.2}
\begin{table}[!h] 
	\centering
	\caption{Example 5 (\romannumeral1). Convergence test for Navier-Stokes equation using Scheme \uppercase\expandafter{\romannumeral2}/BDF$2$ with and without relaxation.}
	 \label{table:NS-scheme2-2nd}
	\begin{tabular}{||c||cccc||cccc||}
		\hline
	& & ESAV-2 & & & & R-ESAV-2 & & \\
		\hline
	$\delta t$ & $\|e_{\mathbf{u}}\|_{L^{2}}$ & Rate & $\|e_{p}\|_{L^{2}}$ & Rate & $\|e_{\mathbf{u}}\|_{L^{2}}$ & Rate & $\|e_{p}\|_{L^{2}}$ & Rate\\
		\hline
	2.50E-2  & 3.83E-02 &  --    & 6.23E-02   & --    & 8.25E-03 &  --    & 1.59E-02   & -- \\
	1.25E-2  & 7.60E-03 &  2.33  & 1.30E-02   & 2.26  & 2.05E-03 &  2.01  & 3.96E-03   & 2.00\\
	6.25E-3  & 1.73E-03 &  2.14  & 2.98E-03   & 2.12  & 5.15E-04 &  2.00  & 9.93E-04   & 1.99\\
	3.13E-3  & 4.13E-04 &  2.06  & 7.17E-04   & 2.06  & 1.29E-04 &  2.00  & 2.49E-04   & 2.00\\
	1.56E-3  & 1.01E-04 &  2.03  & 1.76E-04   & 2.03  & 2.23E-05 &  2.00  & 6.24E-05   & 2.00\\
	\hline
	\end{tabular}
\end{table}

\begin{figure}[!h]
	\centering
	\includegraphics[width=5.3cm]{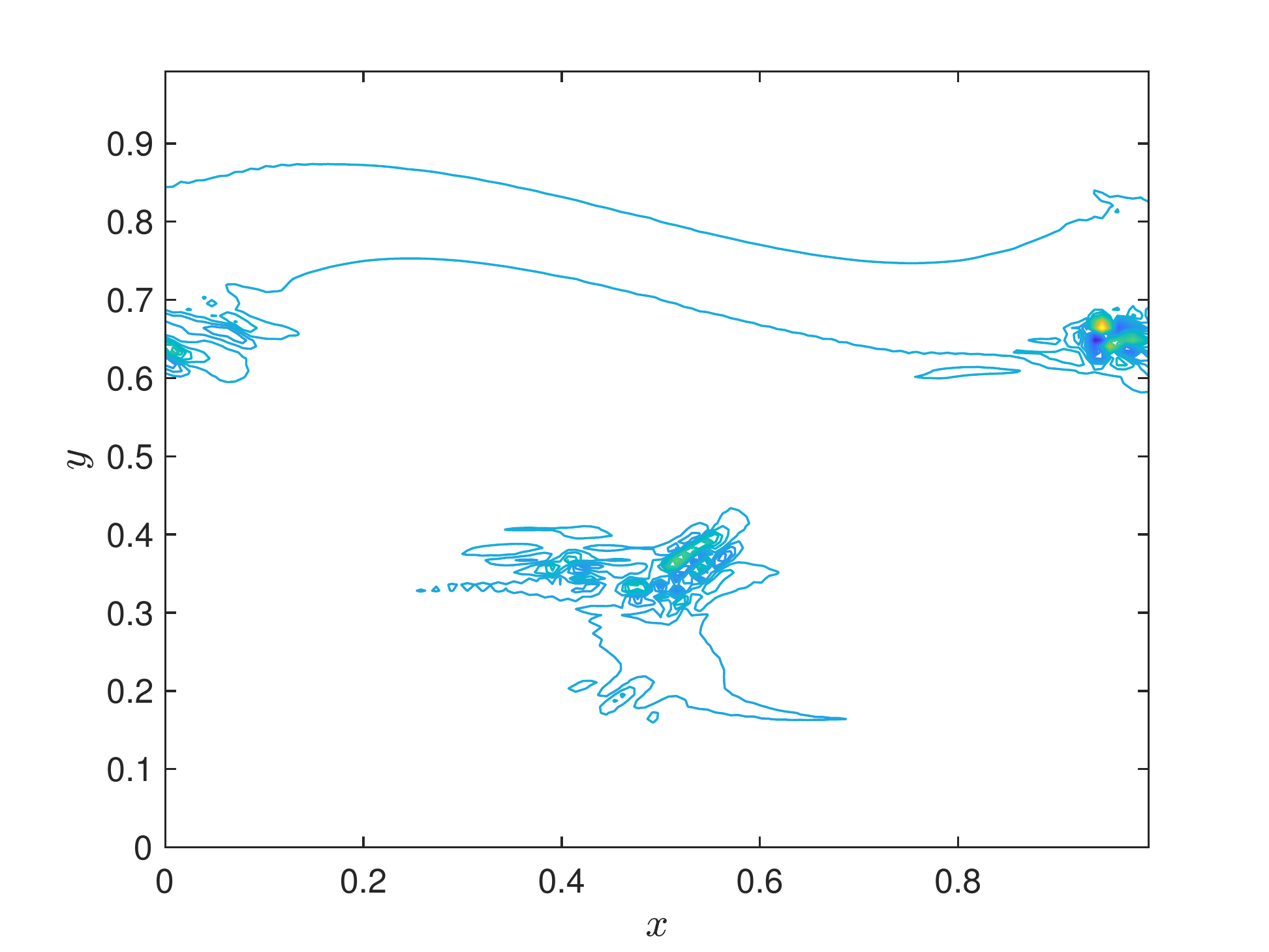}
	\includegraphics[width=5.3cm]{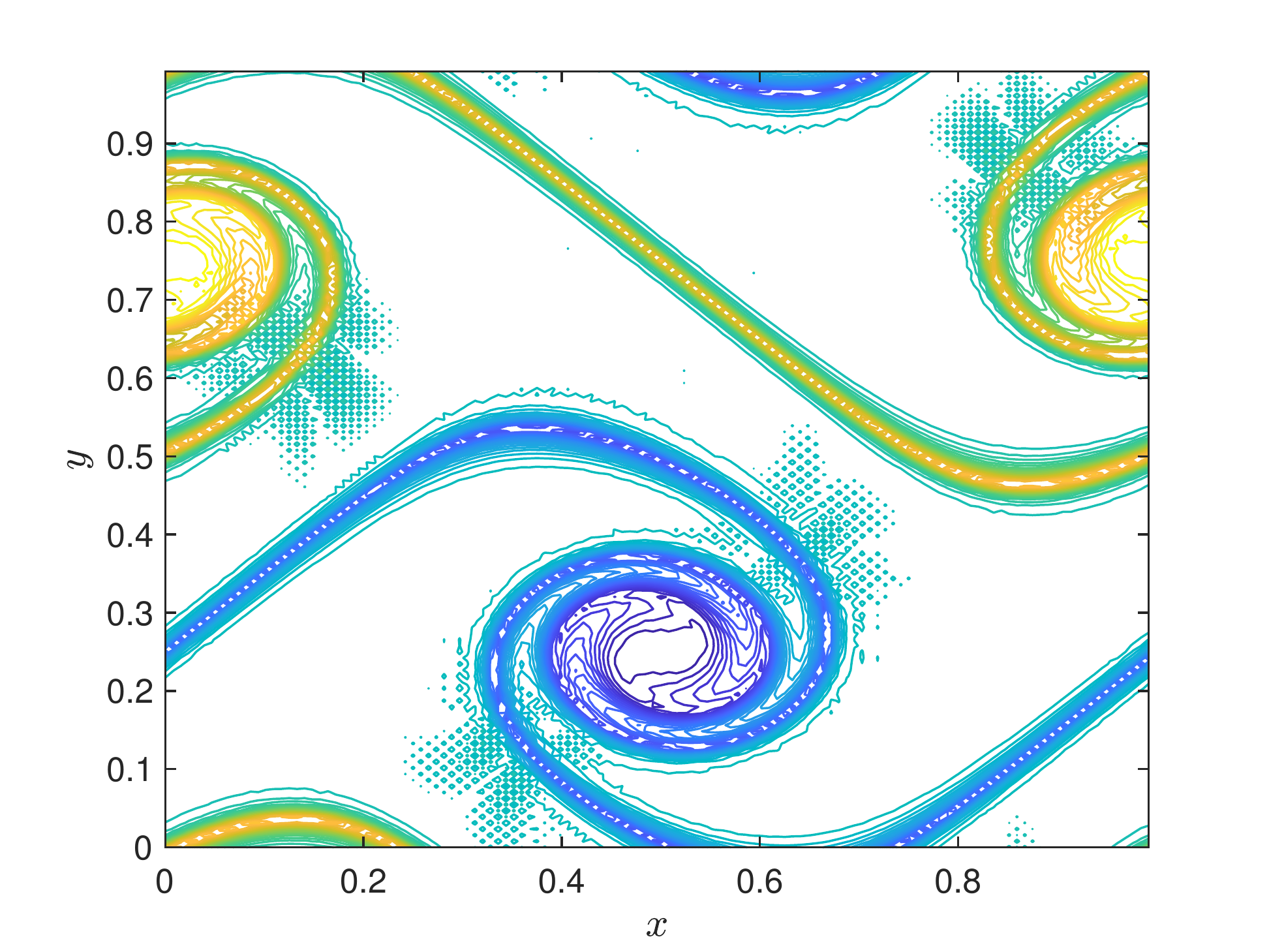}\\
	\includegraphics[width=5.3cm]{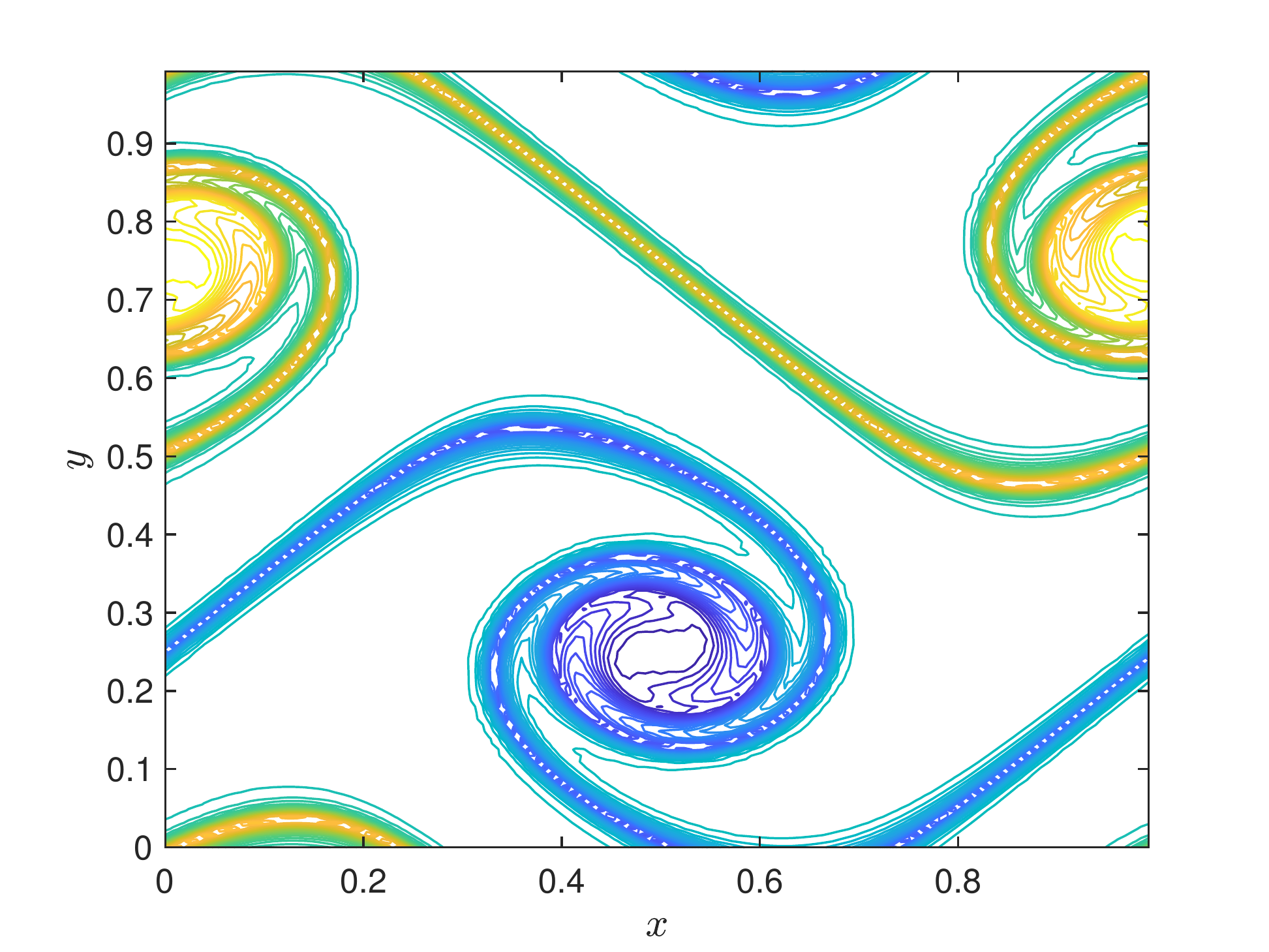}
	\includegraphics[width=5.3cm]{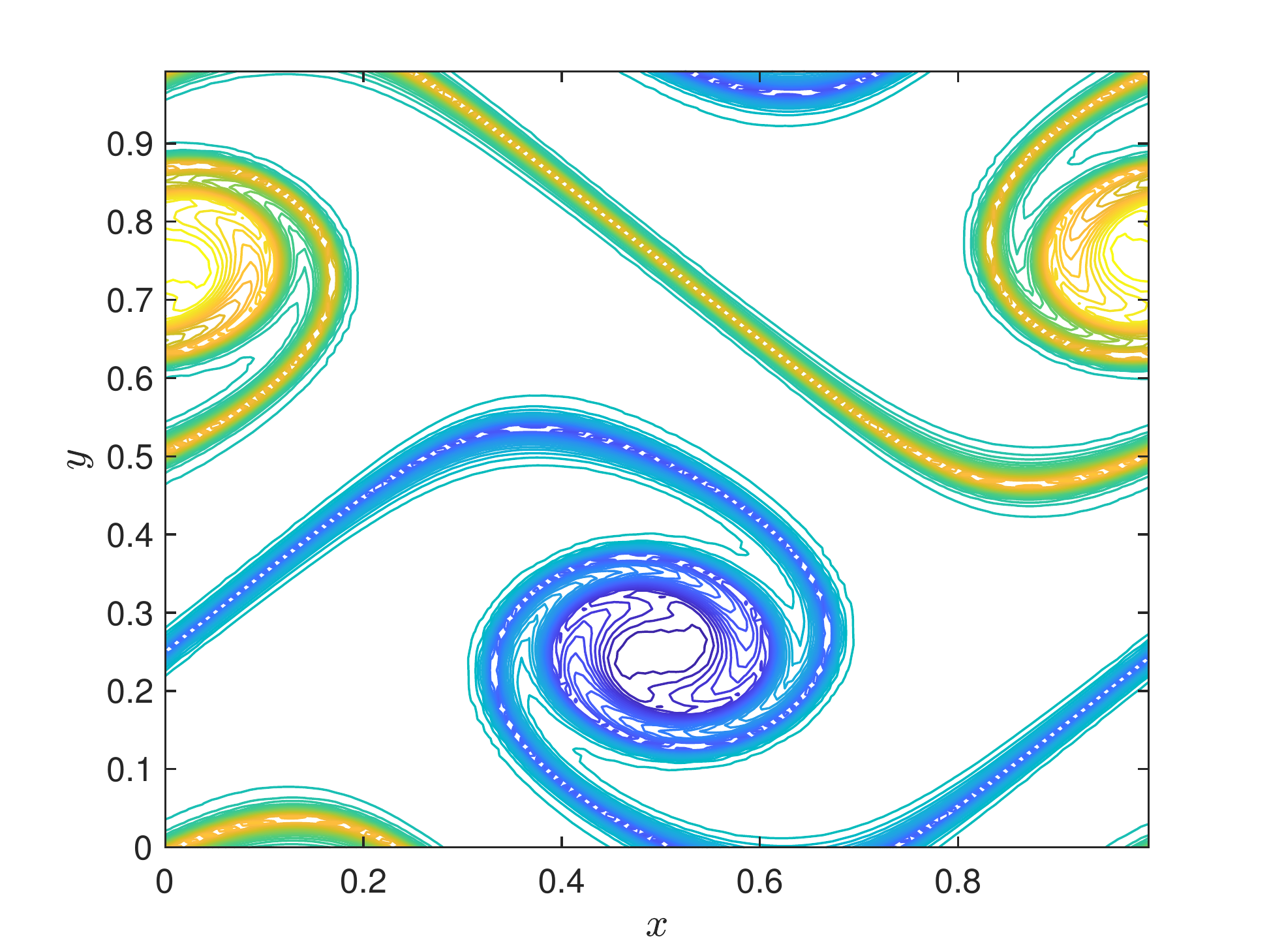}
	\caption{Example 5 (\romannumeral2). The vorticity contours at $T = 1.2$ with $\nu  = 5e-5$, and
		$\delta t = 6.7 e- 4$ obtained by Scheme \uppercase\expandafter{\romannumeral2}/BDF$k$ ($k=1, 2, 3, 4$).}
	\label{Fig:NS-double-shear-RESAV2-periodic-scheme-comparison}
\end{figure}

\begin{figure}[!h]
\centering
\subfigure[contours of vorticity at $T=0.8, 1, 1.2$]{
	\includegraphics[width=5.3cm]{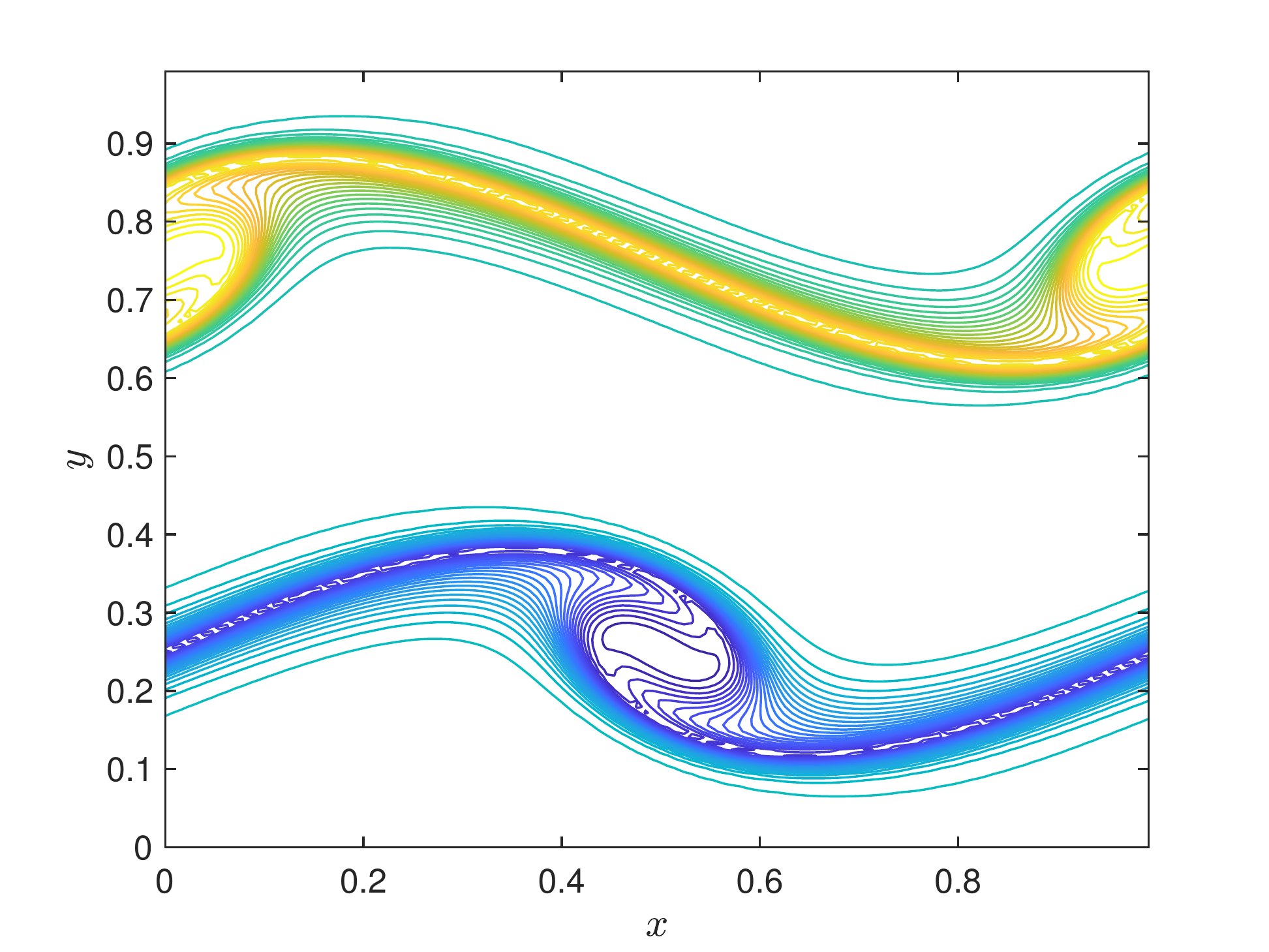}\hspace{-6mm}
	\includegraphics[width=5.3cm]{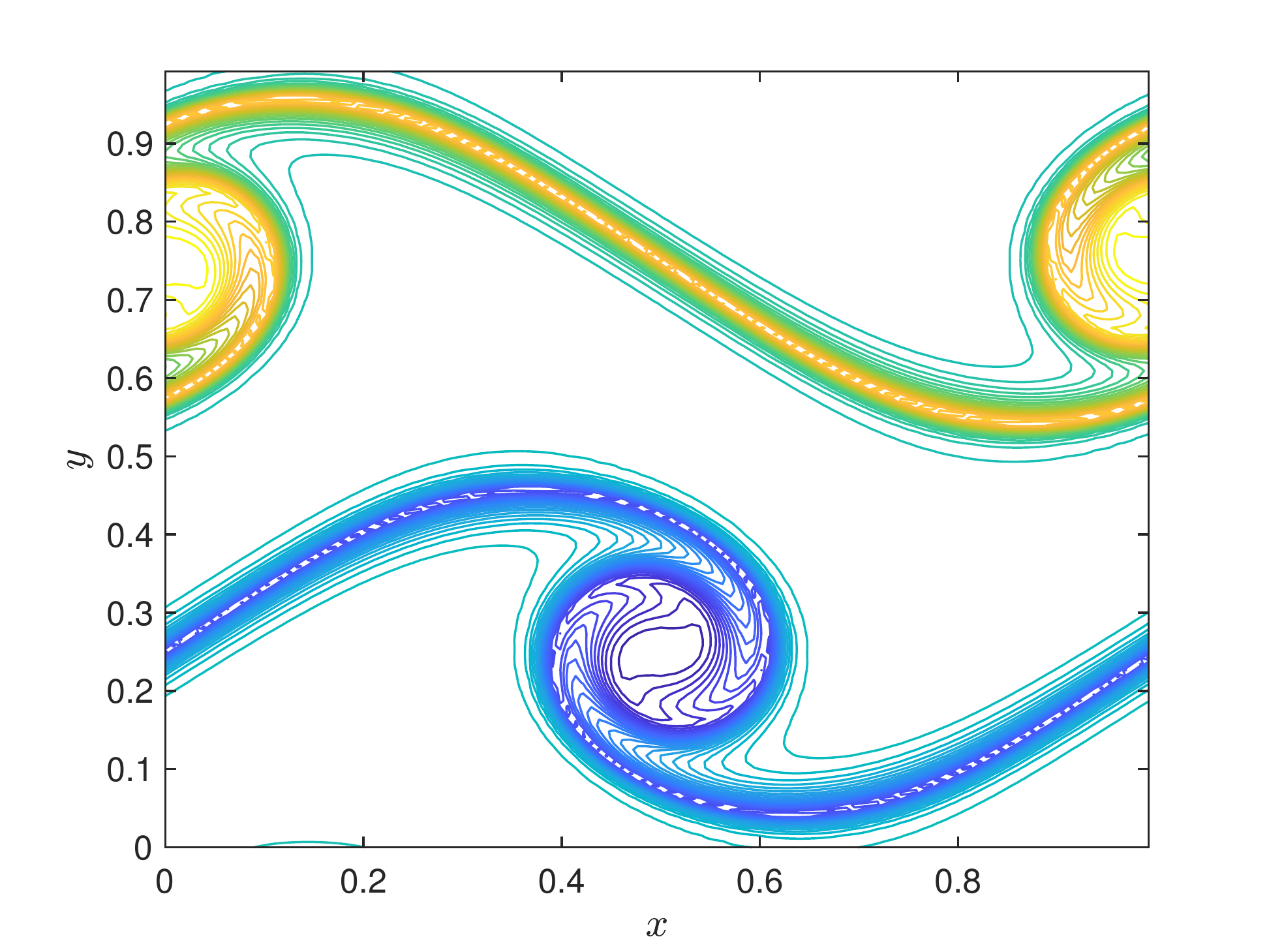}\hspace{-6mm}
	\includegraphics[width=5.3cm]{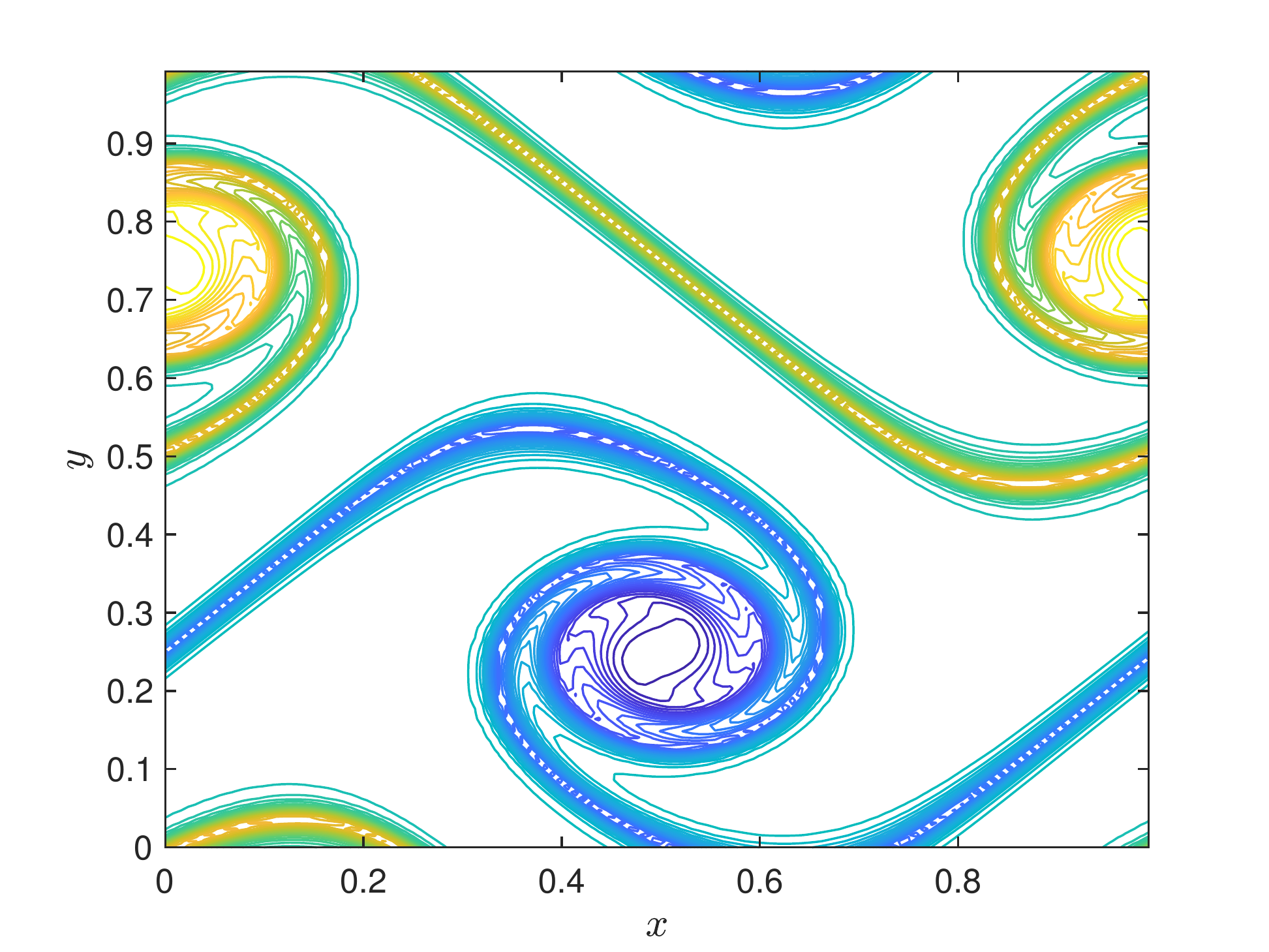}
} 
\subfigure[color-filled contours of vorticity at $T=0.8, 1, 1.2$]{
	\includegraphics[width=5.3cm]{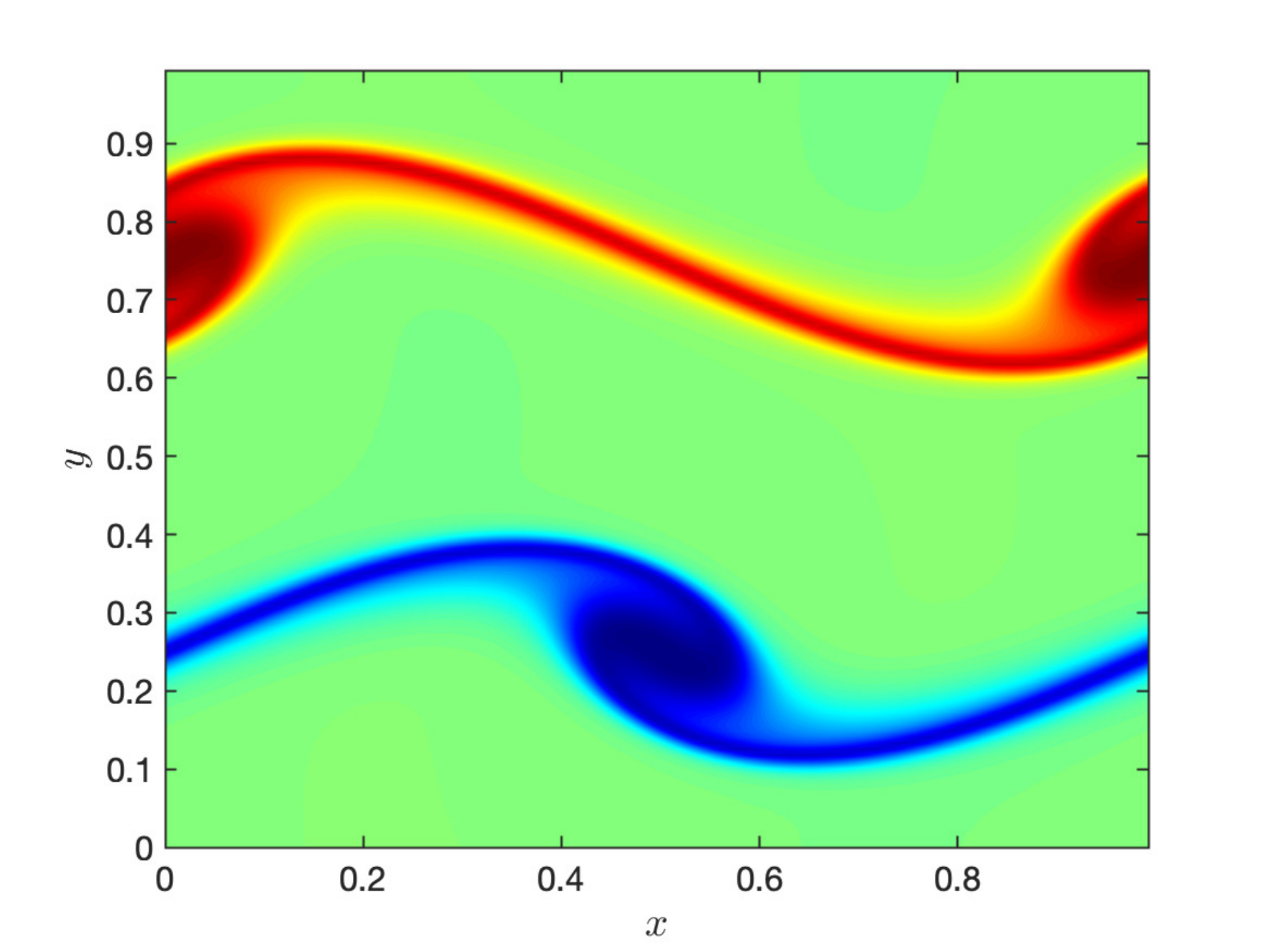}\hspace{-6mm}
	\includegraphics[width=5.3cm]{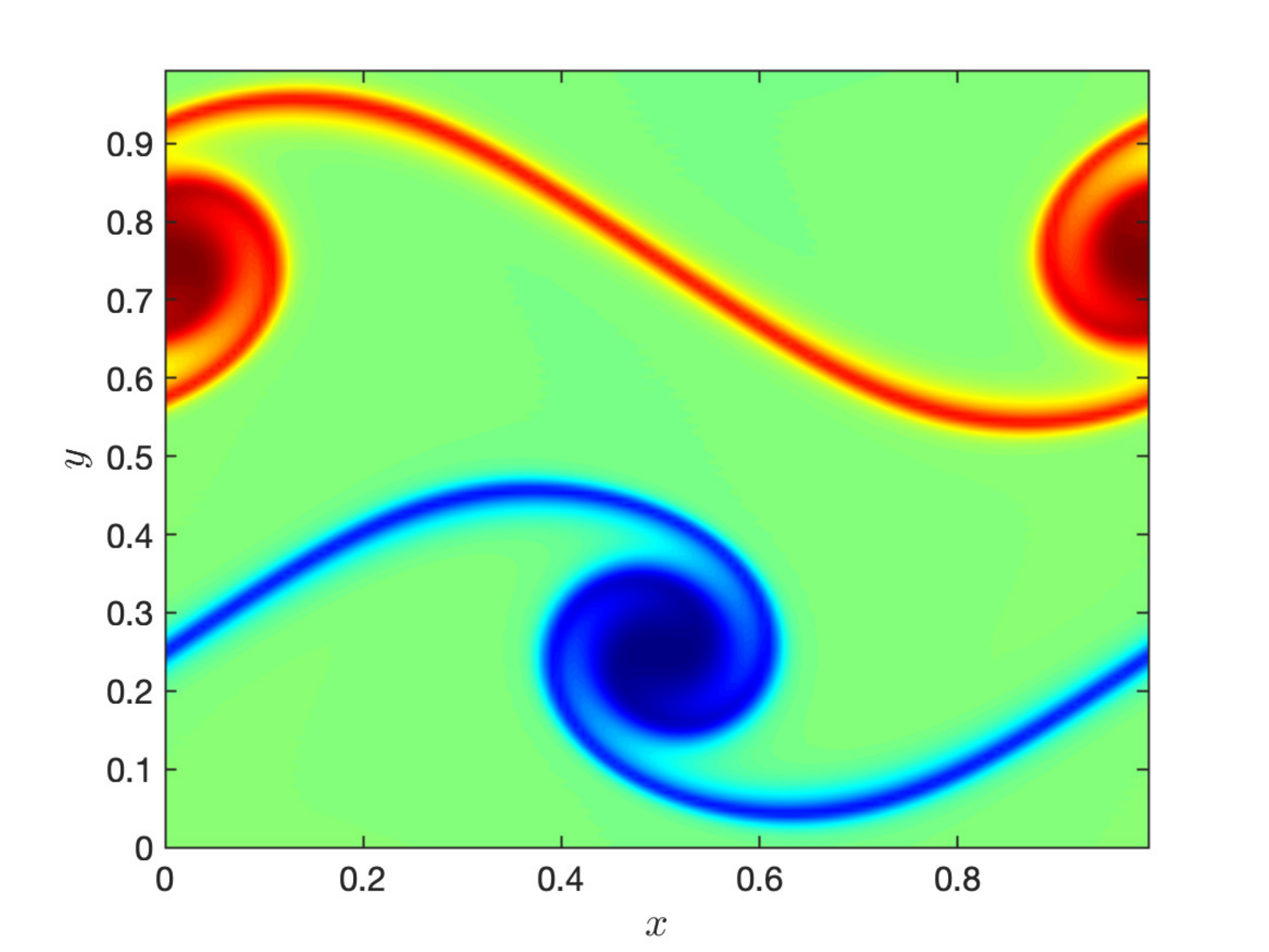}\hspace{-6mm}
	\includegraphics[width=5.3cm]{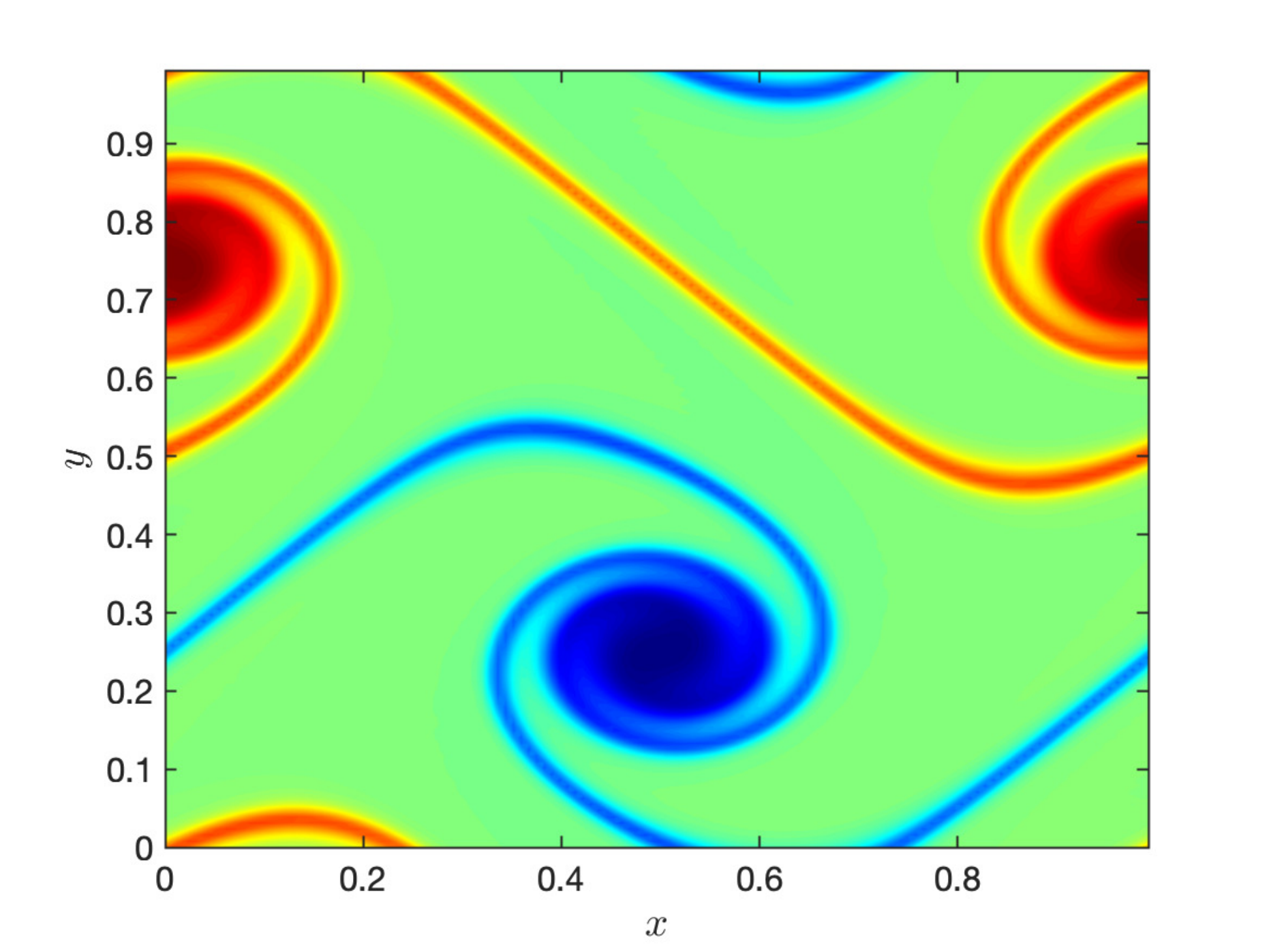}
}
\label{Fig:NS-double-shear-RESAV2-periodic-evolution}
\caption{Example 5 (\romannumeral2). The evolution of vorticity with $\nu  = 1e-4$, and $\delta t = 6e- 4$ obtained by Scheme \uppercase\expandafter{\romannumeral2}/BDF$2$.}
\end{figure}

\section{Conclusions}
\label{sec:conclusion}
In this paper, we constructed two kind of R-ESAV/BDF$k$ schemes which can improve the accuracy significantly by introducing a relaxation factor to improve the accuracy for the classical SAV method. The constructed schemes are linear and unconditionally energy stable. They can guarantee the positive property of SAV without any assumption compared with R-SAV and R-GSAV approach. Moreover  the constructed R-ESAV-2 approach is easy to construct high-order BDF$k$ schemes and can be applied to general dissipative system. Finally we proved that the constructed RESAV scheme with relaxation are more accuracy and closer to original energy by using ample numerical examples.

\clearpage
\bibliographystyle{plain}
\bibliography{myref}

\end{document}